\documentclass{cocv}
\usepackage{graphicx} 
\usepackage{amsmath,amsfonts,amssymb,bm, mathtools}
\numberwithin{equation}{section}
\usepackage{xcolor}
\usepackage{tikz}
\usepackage{stanli}
\usepackage{subcaption}
\usepackage[ruled,vlined]{algorithm2e}
\usetikzlibrary{decorations,decorations.pathreplacing,patterns,calc}
\usepackage{pgfplots}
\usepackage[numbers,sort&compress]{natbib}
\pgfplotsset{compat=1.18}
\pgfplotsset{warning/approx empty range enlarged/.code args={#1#2#3}{}}
\pgfplotsset{colormap={mopedcolor}{%
rgb=(1.000000,1.000000,1.000000)
rgb=(1.000000,0.968627,1.000000)
rgb=(0.988235,0.909804,1.000000)
rgb=(0.968627,0.850980,1.000000)
rgb=(0.945098,0.800000,1.000000)
rgb=(0.905882,0.741176,1.000000)
rgb=(0.866667,0.690196,1.000000)
rgb=(0.819608,0.650980,1.000000)
rgb=(0.760784,0.600000,1.000000)
rgb=(0.698039,0.560784,1.000000)
rgb=(0.631373,0.521569,1.000000)
rgb=(0.568627,0.478431,1.000000)
rgb=(0.486275,0.439216,1.000000)
rgb=(0.411765,0.411765,1.000000)
rgb=(0.380392,0.431373,1.000000)
rgb=(0.349020,0.458824,1.000000)
rgb=(0.321569,0.478431,1.000000)
rgb=(0.290196,0.513725,1.000000)
rgb=(0.270588,0.560784,1.000000)
rgb=(0.250980,0.611765,1.000000)
rgb=(0.219608,0.662745,1.000000)
rgb=(0.200000,0.705882,1.000000)
rgb=(0.188235,0.768627,1.000000)
rgb=(0.168627,0.835294,1.000000)
rgb=(0.149020,0.901961,1.000000)
rgb=(0.141176,0.972549,1.000000)
rgb=(0.121569,1.000000,0.972549)
rgb=(0.109804,1.000000,0.898039)
rgb=(0.101961,1.000000,0.819608)
rgb=(0.090196,1.000000,0.741176)
rgb=(0.078431,1.000000,0.662745)
rgb=(0.070588,1.000000,0.596078)
rgb=(0.058824,1.000000,0.513725)
rgb=(0.050980,1.000000,0.431373)
rgb=(0.050980,1.000000,0.349020)
rgb=(0.039216,1.000000,0.262745)
rgb=(0.031373,1.000000,0.192157)
rgb=(0.031373,1.000000,0.109804)
rgb=(0.031373,1.000000,0.031373)
rgb=(0.101961,1.000000,0.019608)
rgb=(0.184314,1.000000,0.019608)
rgb=(0.247059,1.000000,0.019608)
rgb=(0.325490,1.000000,0.011765)
rgb=(0.407843,1.000000,0.011765)
rgb=(0.490196,1.000000,0.011765)
rgb=(0.572549,1.000000,0.011765)
rgb=(0.635294,1.000000,0.011765)
rgb=(0.717647,1.000000,0.000000)
rgb=(0.800000,1.000000,0.000000)
rgb=(0.882353,1.000000,0.000000)
rgb=(0.964706,1.000000,0.000000)
rgb=(1.000000,0.968627,0.000000)
rgb=(1.000000,0.882353,0.000000)
rgb=(1.000000,0.800000,0.000000)
rgb=(1.000000,0.717647,0.000000)
rgb=(1.000000,0.635294,0.000000)
rgb=(1.000000,0.568627,0.000000)
rgb=(1.000000,0.482353,0.000000)
rgb=(1.000000,0.400000,0.000000)
rgb=(1.000000,0.317647,0.000000)
rgb=(1.000000,0.235294,0.000000)
rgb=(1.000000,0.164706,0.000000)
rgb=(1.000000,0.082353,0.000000)
rgb=(1.000000,0.000000,0.000000)
}}%

\usepackage[colorlinks,linkcolor=blue,citecolor=blue]{hyperref}

\newcommand{\tens}[1]{\bm{\mathit{#1}}}             
\newcommand{\tenss}[1]{\bm{\mathbf{#1}}}            
\newcommand{\tensf}[1]{\bm{\mathbf{#1}}}            

\newcommand{\trn}{\mathsf{T}}                       
\newcommand{\strain}{\tenss{\varepsilon}}           
\newcommand{\stress}{\tenss{\sigma}}                
\newcommand{\abstrstrain}{\lvert\operatorname{Tr}(\strain)\rvert}
\newcommand{\Ei}{\tensf{E}_i}                       
\newcommand{\EV}{\tensf{E}_i^{\mathrm{V}}(\revisionA{\rho_i})}   
\newcommand{\EVs}{\tensf{E}_i^{\mathrm{V}}(\revisionA{\rho_i^\star})}   
\newcommand{\Emin}{\tensf{E}^{-}}                   
\newcommand{\Emax}{\tensf{E}^{+}}                   
\newcommand{\Estar}{\tensf{E}^{*}}
\newcommand{\nel}{n_\mathrm{e}}                     
\newcommand{\nlc}{n_\mathrm{lc}}                    
  
\newcommand{\semtrx}[1]{\ifcat\noexpand#1\relax\bm{#1}\else\mathsf{#1}\fi}
\newcommand{\sevek}[1]{\ifcat\noexpand#1\relax\bm{#1}\else\mathsf{#1}\fi} 

\newcommand{\EiM}{\semtrx{E}_i^\mathrm{M}}          
\newcommand{\EiMb}{\semtrx{E}_i^{\mathrm{M,b}}}  

\newcommand{\Conv}{\operatorname{Conv}}
\newcommand{\Tr}{\operatorname{Tr}}
\newcommand{\Det}{\operatorname{Det}}
\DeclareMathOperator*{\argmin}{arg\,min}

\newif\ifdebuglayout
\debuglayouttrue 

\newcommand{\revisionA}[1]{#1}

\makeatletter
\def\@NMJRNL@E{} 
\def\@NMJRNL@F{} 
\def\@JRNL@X{}   
\let\ps@firstpage\ps@plain 
\makeatother

\begin{document}

\title{Towards realizable free material optimization via energy bound hierarchy}
\author{Marek Tyburec}\address{Experimental Center, Faculty of Civil Engineering, CTU in Prague, Czech Republic, \email{marek.tyburec@cvut.cz}}
\author{Michael Stingl}\address{Department of Mathematics, Friedrich-Alexander-University, Erlangen-Nürnberg, Germany, \email{michael.stingl@fau.de}}
\author{Shenyuan Ma}\address{Department of Mechanics, Faculty of Civil Engineering, CTU in Prague, Czech Republic, \email{ma.shenyuan@fel.cvut.cz}}

\begin{abstract}
\revisionA{We study free orthotropic material optimization for two-dimensional plane-stress compliance minimization with two well-ordered isotropic phases. To connect free material optimization and composite realizability, we construct a hierarchy of admissible sets using zeroth-order, Voigt, and Hashin--Shtrikman energy bounds. The zeroth-order and Voigt models are convex, with admissible sets that coincide at global volume-budget endpoints and satisfy strict inclusion at intermediate volume budgets. The Voigt model reduces to isotropic variable-thickness-sheet optimization, while the single-loadcase continuum zeroth-order problem admits an orthotropic optimum. The classical Hashin--Shtrikman bounds yield a tighter admissible set that is nonconvex jointly in the elasticity tensor and volume fraction and whose convex hull equals the Voigt set. We derive reduced formulations by explicitly characterizing the minimum volume fraction compatible with these bounds. For orthotropic tensors, this volume-fraction estimator reduces to a one-dimensional search over the strain trace. In the single-loadcase continuum setting, the Hashin--Shtrikman-constrained optimal value equals that of the classical Allaire--Kohn relaxation, attained by a homogenized design with sequential-laminate effective tensors. For multiple loadcases, the bounds remain necessary but are generally insufficient for realizability. Numerical results illustrate the compliance hierarchy and close agreement with single-loadcase laminate references.}
\end{abstract}




\maketitle

\section{Introduction}

Free material optimization (FMO)\revisionA{, also known as free material design (FMD),} is a class of topology-optimization methods in which the elasticity tensor field is directly optimized. By treating the local constitutive response as the design variable, FMO offers significantly greater design freedom than density-based topology optimization with prescribed material interpolation laws \citep{Ringertz1993,Bendsoe1994,Bendsoe1995,Zowe1997}.

Typically, FMO formulations impose positive semidefiniteness of elasticity tensors and invariant-based material budget constraints. These constraints control stiffness magnitude, but they do not directly characterize phase-wise realizability as effective properties of composites with prescribed constituent phases \citep{Wu2021,Tyburec2022}. A complementary asymptotic result is due to \citet{Milton1995}: in the extreme high-contrast limit, arbitrary positive-definite effective tensors can be realized by hierarchical constructions. However, this asymptotic realizability result does not by itself provide computable finite-contrast admissible sets for fixed constituent phases.

This paper addresses the fixed-phase inverse question relevant for material-constrained design: for given well-ordered isotropic phases, which effective tensors are admissible in the model? We pursue this question through computable admissible sets induced by energy bounds and organize them into a hierarchy: zeroth-order, Voigt, and Hashin--Shtrikman bounds. This hierarchy provides progressively tighter, realizability-informed relaxations, but it does not provide a complete characterization of realizability in general. We focus on orthotropy because, \revisionA{for two well-ordered isotropic phases in two dimensions}, bound-attaining laminate constructions are orthotropic \citep{Allaire1993wo,Allaire1993}, making orthotropy both physically motivated and analytically tractable.

\revisionA{The hierarchy applies to both single- and multi-loadcase compliance minimization. Its connection to exact relaxation by homogenization is established in the single-loadcase continuum setting. For multiple loadcases, the bounds remain necessary conditions for realizability, but do not capture the coupled energy constraints required for a tight relaxation \citep{Allaire1996}.}

The remainder of this introduction is organized as follows: Section~\ref{sec:intro_fmo} reviews FMO, Section~\ref{sec:intro_bounds} recalls energy-based bounds, Section~\ref{sec:intro_homogenization} summarizes homogenization-based optimization, and Section~\ref{sec:intro_aims} states the aims and contributions.

\subsection{Free Material Optimization}\label{sec:intro_fmo}
\revisionA{
Early FMO formulations treated positive-semidefinite constitutive tensors as design variables and controlled material usage through invariant-based resource constraints, typically involving the trace or Frobenius norm \citep{Ringertz1993,Bendsoe1994}. Within this setting, \citet{Bendsoe1994} analytically derived local laws for optimal material properties and a reduced single-loadcase formulation. \citet{Bendsoe1995} extended the approach to multiple loadcases using weighted mean compliance. Mathematical-programming formulations incorporating multiple loadcases and unilateral contact were developed by \citet{Zowe1997}, followed by a semidefinite programming treatment of worst-case compliance with accompanying existence results \citep{BenTal2000}. A detailed mathematical analysis of trace-constrained FMO with pointwise trace bounds was provided by \citet{Werner2000}, covering existence, duality, and finite-element convergence in the single-loadcase setting, together with existence and duality results for multiple loadcases. These developments are reviewed in \citep{Kocvara2002,Kocvara2008}.
}

The framework was subsequently extended in several practically relevant directions. Formulations with local stress constraints, together with tailored scalable numerical solution methods, were developed in \citep{Kocvara2007,Kocvara2012}. Vibration requirements were addressed by imposing constraints on the fundamental eigenfrequency \citep{Stingl2009}. On the theoretical side, \citet{Haslinger2010} established existence and convergence results for such multidisciplinary extensions using $H$-convergence. In parallel, stress-based trace-constrained free material formulations were developed for single and multiple loadcases \citep{Czarnecki2012,Czarnecki2014}. These developments were further unified across anisotropic, isotropic, and fixed-Poisson-ratio settings \citep{Lewinski2021}. \revisionA{For isotropic material design, \citet{Yang2019} formulated the trace-constrained problem as a second-order cone program.}

\revisionA{A mathematical framework based on optimal mass distribution was developed by \citet{Bolbotowski2021}. It represents constitutive fields as tensor-valued measures and accommodates closed convex cones of admissible constitutive tensors and norm-based material costs. It provides existence and displacement--stress duality results, together with a construction of optimal constitutive designs from dual stress solutions; see also \citep{Bolbotowski2022}. More recently, \citet{Bolbotowski2026} introduced $L^p$-based cost constraints for isotropic material design. They derived mutually dual stress- and displacement-based formulations and established convergence to the measure-valued formulation of the original trace-constrained isotropic material design problem as $p\downarrow1$.}

Despite this breadth, most FMO formulations share a common structural feature: admissibility is primarily encoded through positive-semidefiniteness and trace- or Frobenius-norm-based material-resource constraints. These conditions control stiffness levels but do not directly characterize phase-wise realizability for fixed constituent phases.

\subsection{Energy-Based Bounds}\label{sec:intro_bounds}

To address the realizability gap left by FMO formulations that control material usage through invariant-based measures, one can use admissible sets induced by energy bounds for two-phase composites. Classical interval estimates of Voigt and Reuss type were discussed for reinforced solids by \citet{Hill1963}, and substantially tightened by the Hashin--Shtrikman variational framework \citep{Hashin1963}. \revisionA{Explicit energy bounds and extremal microstructures for optimal design of plates in bending were developed by Gibiansky and Cherkaev \citep{Gibiansky1997plates} in 1984, connecting energy bounds with optimal composite constructions through the translation method. A detailed account of its derivation for thin plates in bending is given in \citep[Section~26]{Lewinski2000}.} 

For two \revisionA{well-ordered} isotropic phases with prescribed volume fractions, \citet{Avellaneda1987} established optimal bounds for imposed uniform strain or stress ensembles and proved attainability by finite-rank laminates, making the link between optimal bounds and realizable microgeometries explicit. This framework was then sharpened in directions directly relevant here. \citet{Allaire1993wo} treated mixtures of two possibly anisotropic well-ordered phases and derived optimal upper and lower bounds for elastic and complementary energies (and related sums of energies), with parallel Hashin--Shtrikman and translation-method interpretations. In the two-dimensional isotropic two-phase case, \citet{Allaire1993} gave explicit optimal formulas and attainability results that also cover non-well-ordered regimes. In the orthotropic setting, \citet{Lipton1994} derived optimal correlated bounds for the six independent orthotropic effective moduli of two well-ordered isotropic phases and identified finite-rank orthotropic laminates that attain these bounds.

Later work focused on turning bound theory into optimization-ready ingredients. \citet{Bendsoe1999} analyzed interpolation schemes through the lens of energy bounds and realizability constraints, clarifying the role of Voigt and Reuss-Voigt descriptions \revisionA{and their relation to} tighter Hashin--Shtrikman \revisionA{bounds}. \citet{Burazin2021} derived explicit lower Hashin--Shtrikman complementary-energy formulas in 2D well-ordered settings, and later provided explicit 3D Hashin--Shtrikman upper primal and lower complementary bounds together with saturating laminate constructions in \citep{Burazin2024}. At the opposite end, \citet{Lobos2016} revisited constituent-only zeroth-order bounds, provided lower-bound constructions, and derived explicit bounds for individual tensor components. A broad synthesis of variational bounds, translation methods, attainability, and $G$-closure viewpoints is provided by \citet{Milton2022}.

\subsection{Homogenization-Based Optimization}\label{sec:intro_homogenization}
Building on these energy bounds and microstructures attaining them, the homogenization literature developed topology optimization formulations in which microstructured composites appear as relaxed design variables. The seminal contribution of \citet{Bendose1988} introduced a material-distribution approach on a fixed reference domain, where design variables parameterized effective anisotropic properties of periodic rectangular-hole cells, thereby avoiding repeated remeshing while permitting topology changes. In a related relaxed setting, \citet{Allaire1993od} considered a plane-stress formulation with a weighted objective combining compliance and mass through extremal microstructures and used an alternating-minimization stress-based framework. In a complementary setting, \citet{Jog1994} restricted the admissible composites to rank-2 laminates, analytically eliminated microstructural variables through a sequence of reduced optimization problems, and solved the resulting mixed finite-element formulation with coupled macroscopic-microscopic updates.

The same period also established inverse and multi-loadcase perspectives. \citet{Sigmund1994} formulated inverse homogenization, namely the optimization of a periodic base cell to match prescribed effective constitutive properties. \citet{Allaire1996} extended homogenization-based optimization to multiple loadcases, emphasizing that explicit optimal composite parameters available in single-loadcase settings are not available in the multi-loadcase regime. This direction was strengthened in \citep{allaire1997shape}, which proved a general relaxation result in two and three dimensions and introduced a penalized topology-optimization algorithm to recover classical phase-separated designs. \citet{Allaire2002} consolidated these developments into a unified framework linking nonexistence of classical minimizers, relaxation by homogenization, and practical numerical algorithms. \revisionA{For planar elasticity, \citet{Lewinski2004} presented the relaxed problem in both stress and displacement variables, using explicit inversion of the constitutive relations within each regime. Building on the explicit lower Hashin--Shtrikman complementary-energy bounds, \citet{Burazin2021} derived a two-dimensional optimality-criteria method for compliance minimization using} laminates made of two well-ordered isotropic phases, while \citet{Burazin2023} extended this strategy to 3D settings and combined it with penalization to recover classical designs.

Finally, dehomogenization and projection methods have addressed the reconstruction of manufacturable designs from relaxed optima. \citet{Pantz2008} proposed a post-treatment that projects not only density but also lamination information, producing single-scale geometries with controlled detail. \citet{Groen2017} simplified and scaled this idea to high-resolution manufacturable lattices with performance close to homogenization-based predictions.  \citet{Allaire2019} improved optimization-projection consistency for modulated periodic microstructures, including conformal orientation treatment through a harmonic angle field.

\subsection{Aims and Contributions}\label{sec:intro_aims}
The aim of this paper is to develop a realizability-aware free material optimization framework for two well-ordered isotropic phases in plane stress. \revisionA{Realizability is considered in the homogenization sense, with effective tensors belonging to the $G$-closure of the prescribed phases at a given volume fraction. We approximate this set through a hierarchy of necessary energy-bound conditions, progressing from zeroth-order constraints to Voigt and then to Hashin--Shtrikman (HS) bounds.} We make this progression explicit through a pointwise inclusion hierarchy and use it as the organizing principle of both the analysis and the algorithms.

To achieve this aim, the main contributions of this work are:
\begin{enumerate}
    \item We formulate a realizability-aware free material optimization framework for two-dimensional plane stress, organized as a hierarchy of admissible sets induced by zeroth-order, Voigt, and HS energy bounds. We establish well-posedness by proving existence of minimizers for the zeroth-order and Voigt models and existence in the $H$-relaxed sense for the HS model.
    \item For the convex zeroth-order and Voigt models, we prove strict inclusion of admissible sets for intermediate volume \revisionA{budgets} and coincidence at pure-phase endpoints. We also prove that the Voigt model reduces to an isotropic variable-thickness-sheet formulation, and that in the single-loadcase continuum zeroth-order setting an optimal solution can be chosen orthotropic.
    \item \revisionA{Using a Voigt-minus-correction form of the classical Allaire--Kohn energy bound \citep{Allaire1993,Allaire1993wo}, we} prove that the convex hull of the HS feasible set equals the Voigt admissible set. We further derive reduced HS formulations via active-constraint analysis and explicit elementwise volume characterization, with dedicated reductions for orthotropic effective tensors.
    \item \revisionA{We establish the connection between the HS-constrained FMO formulation and classical relaxation by homogenization: in the two-dimensional single-loadcase continuum setting with well-ordered isotropic phases, their optimal values coincide. For multiple loadcases, the energy bounds remain necessary conditions for realizability, but generally do not provide a tight relaxation.}
    \item We \revisionA{apply} sequential global programming \citep{Semmler-SIAM-2018} \revisionA{to} the resulting nonconvex orthotropic HS problem and validate the hierarchy numerically across phase contrasts and mesh resolutions, confirming the predicted compliance ordering and close agreement with finite-rank laminate references.
\end{enumerate}

The paper is organized as follows. Section~\ref{sec:notation} introduces notation, and Section~\ref{sec:cont_fmo} presents the continuum formulation of the problem together with existence results. Sections~\ref{sec:fmo_zero}--\ref{sec:fmo_hashin} then develop the zeroth-order, Voigt, and Hashin--Shtrikman model classes, with inclusion relations, geometric properties, and reduced formulations. Section~\ref{sec:fmo_ortho} introduces the orthotropic Hashin--Shtrikman formulation. Section~\ref{sec:sgp} presents the sequential global programming strategy for the resulting nonconvex optimization problem, while Section~\ref{sec:allaire_alg} describes the \revisionA{classical laminate-based alternating-minimization method \citep{Allaire1993od,Allaire2002,Burazin2021}} used for single-loadcase comparison. Finally, Section~\ref{sec:results} reports numerical experiments, and Section~\ref{sec:summary} concludes.

\section{Notation}\label{sec:notation}
We \revisionA{consider} the planar setting, $d=2$, and adopt the following conventions for the notation used throughout the paper. Vectors and matrices are written in boldface sans-serif font. Vectors are written in lowercase, e.g., $\sevek{u}\in\mathbb{R}^n$, and matrices in uppercase, e.g., $\semtrx{A}\in\mathbb{R}^{m\times n}$. Tensors are written in boldface serif font. First-order tensors are written in lowercase italics, e.g., $\tens{x}$. For second-order tensors we use lowercase upright, e.g., $\tenss{\varepsilon}\in\mathbb{R}^{d\times d}$, while fourth-order tensors are written in uppercase upright, e.g., $\tensf{E}\in\mathbb{R}^{d\times d\times d\times d}$.

For vectors and tensors of compatible sizes, $\langle\cdot,\cdot\rangle$ denotes the Frobenius inner product. We write $\lvert\cdot\rvert$ for the associated norm and use $\lVert\cdot\rVert_F$ when emphasizing the Frobenius norm. For $\tenss{a}\in\mathbb{R}^{d\times d}$ we write $\Tr(\tenss{a})$ for its trace and $\Det(\tenss{a})$ for its determinant. Moreover, for $\tenss{a},\tenss{b}\in\mathbb{R}^{d\times d}$ we denote by $\tenss{a}\otimes \tenss{b}$ the tensor product defined by $(\tenss{a}\otimes \tenss{b})_{ijkl}:=a_{ij}b_{kl}$. For a fourth-order tensor $\tensf{E}$ acting on a second-order tensor $\tenss{a}$, we write $\tensf{E}:\tenss{a}$, and also use the shorthand $\tensf{E}\tenss{a}$. In particular, $\langle\tensf{E}\strain,\strain\rangle:=\langle\tensf{E}:\strain,\strain\rangle$.

By $\mathbb{S}^{n\times n}$ we denote the space of real symmetric $n\times n$ matrices and by $\mathbb{S}^{n\times n}_{\succeq 0}$ the cone of symmetric positive semidefinite matrices. More generally, let $V$ be a finite-dimensional real inner product space with inner product $\langle\cdot,\cdot\rangle$. For self-adjoint linear maps $A,B:V\to V$ we write $A\succeq 0$ if $\langle Au,u\rangle\ge 0$ for all $u\in V$ and define the (Loewner) order $A\preceq B$ by $B-A\succeq 0$ (strict inequalities $\prec,\succ$ analogously). This convention applies, in particular, to symmetric matrices ($V=\mathbb{R}^n$) and to fourth-order elasticity tensors acting as self-adjoint maps on $V=\mathbb{S}^{d\times d}$. In this sense, $\mathbb{S}^{d\times d\times d\times d}_{\succ 0}$ denotes the subset of $\mathbb{S}^{d\times d\times d\times d}$ that is positive definite as a self-adjoint linear map on $\mathbb{S}^{d\times d}$.

Let $\tenss{I}_{d\times d}$ denote the identity matrix in $\mathbb{R}^{d\times d}$. Moreover, let $\tensf{I}_{d\times d\times d\times d}$ denote the fourth-order identity on $\mathbb{S}^{d\times d}$, i.e., $(\tensf{I}_{d\times d\times d\times d}:\tenss{\varepsilon})=\tenss{\varepsilon}$ for all $\tenss{\varepsilon}\in\mathbb{S}^{d\times d}$. For a fourth-order tensor $\tensf{E}\in\mathbb{S}^{d\times d\times d\times d}$ acting as a self-adjoint linear map on $\mathbb{S}^{d\times d}$, we define its trace by $\Tr(\tensf{E}) := \tensf{E}:\tensf{I}_{d\times d\times d\times d}$. Equivalently, $\Tr(\tensf{E})$ coincides with the trace of the Kelvin--Mandel matrix representation of $\tensf{E}$ on $\mathbb{S}^{d\times d}$.

For an isotropic fourth-order tensor $\tensf{E}$ with shear modulus $\mu$ and bulk modulus $\kappa$, the constitutive law $\stress=\tensf{E}:\strain$ takes the volumetric--deviatoric form
\begin{equation}
    \forall\strain\in\mathbb{S}^{d\times d}, \qquad
    \tensf{E}:\strain=\kappa \Tr(\strain)\tenss{I}_{d\times d}+2\mu \left(\strain-\frac{1}{d}\Tr(\strain)\tenss{I}_{d\times d}\right).
\end{equation}
The inverse tensor is then
\begin{equation}
    \forall\stress\in\mathbb{S}^{d\times d},\qquad {\tensf{E}}^{-1}:\stress=\frac{1}{\kappa d^{\revisionA{2}}} \Tr(\stress)\tenss{I}_{d\times d}+\frac{1}{2\mu} \left(\stress-\frac{1}{d}\Tr(\stress)\tenss{I}_{d\times d}\right).
\end{equation}
Given two isotropic tensors $\Emin,\Emax$, we say they are well ordered if $\Emin\prec \Emax$, equivalently $\mu^-<\mu^+$ and $\kappa^-<\kappa^+$.

\section{Methods}

This section formulates the free material optimization problem and the corresponding energy bounds. We work in plane stress ($d=2$) with two isotropic constituent tensors $\Emin,\Emax \in \mathbb{S}^{d\times d\times d\times d}_{\succ 0}$ satisfying $\Emin \prec \Emax$. For a local stiff-phase volume fraction $\revisionA{\rho}\in[0,1]$, mixtures of these phases generate effective tensors, and our goal is to characterize as tightly as possible the admissible set of orthotropic effective tensors through energy bounds. The design variables are the effective tensor field $\tensf{E}(\tens{x})$ and the local volume fraction $\revisionA{\rho}(\tens{x})$, coupled by pointwise bound constraints.

We start from a continuum multi-loadcase compliance-minimization problem and a unified strain-energy representation of admissibility of the energy bounds (Section~\ref{sec:cont_fmo}). We then derive discrete models with increasing tightness: zeroth-order bounds (Section~\ref{sec:fmo_zero}), Voigt upper bounds (Section~\ref{sec:fmo_voigt}), and Hashin--Shtrikman-type upper bounds (Section~\ref{sec:fmo_hashin}). Next, we specialize to orthotropic effective tensors (Section~\ref{sec:fmo_ortho}) and solve the resulting nonconvex problems by sequential global programming (Section~\ref{sec:sgp}). Section~\ref{sec:allaire_alg} recalls the alternating-minimization laminate method of \citet{Burazin2021} \revisionA{and \citet{Allaire1993od}}, used as a comparison reference.

\subsection{Continuum Formulation and Solution Existence}\label{sec:cont_fmo}

Let $\Omega\subset\mathbb{R}^d$ be a bounded Lipschitz domain. Its boundary is decomposed as $\partial\Omega=\overline{\Gamma_\mathrm{D}}\cup\overline{\Gamma_\mathrm{N}}$ with $\Gamma_\mathrm{D}\cap\Gamma_\mathrm{N}=\emptyset$ and $\lvert\Gamma_\mathrm{D}\rvert>0$. We prescribe homogeneous Dirichlet conditions on $\Gamma_\mathrm{D}$ and set
\begin{equation}\label{eq:U0}
\mathcal{U}_0
:=\left\{\tens{v}\in H^1(\Omega;\mathbb{R}^d)\;:\;\tens{v}=\tens{0}\ \text{on }\Gamma_\mathrm{D}\right\}.
\end{equation}
For $\tens{u}\in\mathcal{U}_0$ we set the linearized strain
\begin{equation}\label{eq:strain_def}
\strain(\tens{u}) :=\tfrac12\left(\nabla\tens{u}+\nabla\tens{u}^{\trn}\right)\in \mathbb{S}^{d\times d}.
\end{equation}
Let $\nlc\in\mathbb{N}$ denote the number of load cases.
For each $j=1,\dots,\nlc$, we prescribe a traction $\tens{f}_j\in L^2(\Gamma_\mathrm{N};\mathbb{R}^d)$ and define
\begin{equation}\label{eq:load_functional}
\ell_j(\tens{v}) := \int_{\Gamma_\mathrm{N}}\langle \tens{f}_j,\tens{v}\rangle\,\mathrm{d}\Gamma, \qquad \forall \tens{v}\in\mathcal{U}_0.
\end{equation}
Given a design $\tensf{E}\in L^\infty(\Omega;\mathbb{S}^{d\times d\times d\times d}_{\succ0})$, we introduce the bilinear form
\begin{equation}\label{eq:bilinear_form}
a_{\tensf{E}}(\tens{u},\tens{v})
:=\int_{\Omega}\left\langle \tensf{E}(\tens{x})\,\strain(\tens{u}(\tens{x})),\strain(\tens{v}(\tens{x}))\right\rangle\,\mathrm{d}\tens{x},
\end{equation}
where $\tensf{E}(\tens{x})$ has the usual minor and major symmetries and acts as a self-adjoint linear map on
$\mathbb{S}^{d\times d}$. For each loadcase $j$, we define $\tens{u}_j=\tens{u}_j(\tensf{E})\in\mathcal{U}_0$ as the solution of the weak equilibrium problem
\begin{equation}\label{eq:state_cont}
a_{\tensf{E}}(\tens{u}_j,\tens{v})=\ell_j(\tens{v}),\qquad \forall \tens{v}\in\mathcal{U}_0.
\end{equation}
The compliance associated with the $j$-th loadcase is defined as the work done by the external loads,
\begin{equation}\label{eq:compliance_single}
c_j(\tensf{E}) := \ell_j(\tens{u}_j(\tensf{E})) = a_{\tensf{E}}(\tens{u}_j(\tensf{E}),\tens{u}_j(\tensf{E})),
\end{equation}
and the objective functional is the sum of the individual compliance terms
\begin{equation}\label{eq:multi_compliance}
J(\tensf{E}) := \sum_{j=1}^{\nlc} c_j(\tensf{E}).
\end{equation}

To encode energy bounds on the effective tensors, we introduce a local volume fraction $\revisionA{\rho}:\Omega\to[0,1]$ of the stiffer isotropic phase $\Emax$ and restrict the effective elasticity tensors through pointwise strain-energy inequalities. Specifically, for a.e.\ $\tens{x}\in\Omega$, we require
\begin{equation}\label{eq:strain_bounds_general}
\Phi_{\mathrm{lb}}\big(\strain; \revisionA{\rho}(\tens{x})\big)
\le
\left\langle \tensf{E}(\tens{x})\,\strain,\strain\right\rangle
\le
\Phi_{\mathrm{ub}}\big(\strain; \revisionA{\rho}(\tens{x})\big), \qquad \forall \strain\in \mathbb{S}^{d\times d}.
\end{equation}
We assume the chosen energy bounds satisfy the sandwich condition
\begin{equation}\label{eq:Phi_sandwich}
\langle \Emin\,\strain,\strain\rangle
\le
\Phi_{\mathrm{lb}}(\strain;\revisionA{\rho})
\le
\Phi_{\mathrm{ub}}(\strain;\revisionA{\rho})
\le
\langle \Emax\,\strain,\strain\rangle, \quad
\forall\,\strain\in\mathbb{S}^{d\times d},\ \forall\,\revisionA{\rho}\in[0,1],
\end{equation}
with fixed $\tensf{0} \prec \Emin\prec \Emax$. In particular, there exist constants $0<\alpha\le\beta<\infty$ such that every admissible design satisfies 
\begin{equation}\label{eq:uniform_ellipticity_E}
\alpha\,\lvert\strain\rvert^2 \le \left\langle \tensf{E}(\tens{x})\,\strain,\strain\right\rangle \le \beta\,\lvert\strain\rvert^2,
\qquad \forall \strain\in\mathbb{S}^{d\times d},\ \text{for a.e. }\tens{x}\in\Omega.
\end{equation}
For instance, one may take $\alpha:=\lambda_{\min}(\Emin)$ and $\beta:=\lambda_{\max}(\Emax)$
with respect to the Frobenius inner product on $\mathbb{S}^{d\times d}$.

Different choices of $\big(\Phi_{\mathrm{lb}},\Phi_{\mathrm{ub}}\big)$ correspond to increasingly tight descriptions of realizable effective tensors. Denoting by $\mathcal{A}^{(0)}(\revisionA{\rho})$, $\mathcal{A}^{(1)}(\revisionA{\rho})$, $\mathcal{A}^{(2)}(\revisionA{\rho})$ the admissible sets induced, respectively, by the zeroth-, first-, and second-order bounds used in this work, we have the pointwise hierarchy
\begin{equation}\label{eq:hierarchy_bounds}
\mathcal{A}^{(0)}(\revisionA{\rho}) \supseteq \mathcal{A}^{(1)}(\revisionA{\rho}) \supseteq \mathcal{A}^{(2)}(\revisionA{\rho}),
\qquad\text{for all }\revisionA{\rho}\in[0,1].
\end{equation}
The explicit expressions for these bounds and their inclusion relations are developed in Sections~\ref{sec:fmo_zero}--\ref{sec:fmo_hashin}.

After imposing the global volume constraint
\begin{equation}\label{eq:vol_cont}
\int_{\Omega} \revisionA{\rho}(\tens{x})\,\mathrm{d}\tens{x} \le \overline{V}\,\lvert\Omega\rvert,\qquad \overline{V}\in[0,1],
\end{equation}
\revisionA{or the trace resource constraint 
\begin{equation}\label{eq:fmo_continuum_trace}
\frac{1}{\lvert\Omega\rvert}\int_{\Omega} \Tr(\tensf{E}(\tens{x}))\,\mathrm{d}\tens{x} \le (1-\overline{V})\Tr(\Emin) + \overline{V} \Tr(\Emax),
\end{equation}
in the zeroth-order setting,} the continuum free material optimization problem reads as
\begin{subequations}\label{eq:fmo_continuum}
\begin{align}
\inf_{\tensf{E},\,\revisionA{\rho}}\quad & J(\tensf{E})\label{eq:fmo_continuum_obj}\\
\text{s.t.}\quad
& \tens{u}_j(\tensf{E})\in\mathcal{U}_0 \text{ solves \eqref{eq:state_cont} for } j=1,\dots,\nlc,\label{eq:fmo_continuum_state}\\
& \tensf{E}(\tens{x}) \text{ satisfies \eqref{eq:strain_bounds_general} for a.e. }\tens{x}\in\Omega,\\
& 0\le \revisionA{\rho}(\tens{x})\le 1 \ \text{for a.e. }\tens{x}\in\Omega,\label{eq:fmo_continuum_bounds}\\
& \revisionA{\eqref{eq:vol_cont} \text{ or } \eqref{eq:fmo_continuum_trace}}.\label{eq:fmo_continuum_volume}
\end{align}
\end{subequations}
By \eqref{eq:uniform_ellipticity_E}, the bilinear form $a_{\tensf{E}}$ is continuous and coercive on $\mathcal{U}_0$. Indeed, for all $\tens{u},\tens{v}\in\mathcal{U}_0$,
\begin{equation}\label{eq:a_continuity}
\lvert a_{\tensf{E}}(\tens{u},\tens{v})\rvert
\le \beta\,\lVert\strain(\tens{u})\rVert_{L^2(\Omega)}\,\lVert\strain(\tens{v})\rVert_{L^2(\Omega)}.
\end{equation}
Moreover, Korn's inequality \citep[Definition~3.1]{NecasHlavacek_1981} on $\mathcal{U}_0$ (using $\lvert\Gamma_\mathrm D\rvert>0$) yields a constant $C_\mathrm{K}>0$ such that
\begin{equation}\label{eq:a_coercivity}
a_{\tensf{E}}(\tens{u},\tens{u})
\ge \alpha\,\lVert\strain(\tens{u})\rVert_{L^2(\Omega)}^2
\ge \frac{\alpha}{C_\mathrm{K}^2}\,\lVert\tens{u}\rVert_{H^1(\Omega)}^2
\qquad \forall \tens{u}\in\mathcal{U}_0.
\end{equation}
Hence, \eqref{eq:state_cont} has a unique solution $\tens{u}_j(\tensf{E})\in\mathcal{U}_0$ for each $j$ due to the Lax--Milgram theorem \citep[Corollary 5.8]{Brezis2011}.

To discuss existence of minimizers of \eqref{eq:fmo_continuum}, we introduce the admissible set
\begin{equation}\label{eq:admissible_set_D}
\mathcal{D}
:=\Big\{(\revisionA{\rho},\tensf{E})\in L^\infty(\Omega) \times L^\infty(\Omega;\mathbb{S}^{d\times d\times d\times d})\;:\;
0\le \revisionA{\rho}\le 1\ \text{a.e.},\ \eqref{eq:fmo_continuum_volume}\ \text{holds, and }\tensf{E}\ \text{satisfies }\eqref{eq:strain_bounds_general}\Big\}.
\end{equation}
For each fixed $\revisionA{\rho}\in[0,1]$, define the local (pointwise) admissible set of tensors by
\begin{equation}\label{eq:local_admissible_set}
\mathcal{A}(\revisionA{\rho})
:=\Big\{\tensf{E}\in\mathbb{S}^{d\times d\times d\times d}\;:\;
\Phi_{\mathrm{lb}}(\strain;\revisionA{\rho})\le \langle\tensf{E}\strain,\strain\rangle\le \Phi_{\mathrm{ub}}(\strain;\revisionA{\rho})\ \ \forall \strain\in\mathbb{S}^{d\times d}\Big\},
\end{equation}
so that \eqref{eq:strain_bounds_general} is equivalently the pointwise constraint $\tensf{E}(\tens{x})\in \mathcal{A}(\revisionA{\rho}(\tens{x}))$ for a.e.\ $\tens{x}\in\Omega$.

The upper bound in \eqref{eq:uniform_ellipticity_E} implies $\lVert\tensf{E}(\tens{x})\rVert_{\mathrm{op}}\le \beta$ for a.e.\ $\tens{x}\in\Omega$, where $\lVert\tensf{E}\rVert_{\mathrm{op}}:=\sup_{\strain\in\mathbb{S}^{d\times d}\setminus\{0\}} \langle\tensf{E}\strain,\strain\rangle/\lvert\strain\rvert^2$. Consequently $\lVert\tensf{E}\rVert_{L^\infty(\Omega;\mathrm{op})}\le \beta$ for all admissible designs, and together with $0\le \revisionA{\rho}\le 1$ this yields boundedness of $\mathcal{D}$ in $L^\infty(\Omega) \times L^\infty(\Omega;\mathbb{S}^{d\times d\times d\times d})$. Hence, by the Banach--Alaoglu theorem \citep[Theorem 3.16]{Brezis2011}, any minimizing sequence admits a subsequence
\begin{equation}\label{eq:weak_star_subsequence}
(\revisionA{\rho_k},\tensf{E}_k)\ \overset{\star}{\rightharpoonup}\ (\revisionA{\rho},\tensf{E})
\qquad\text{in }L^\infty(\Omega) \times L^\infty(\Omega;\mathbb{S}^{d\times d\times d\times d}).
\end{equation}
The simple constraints $0\le \revisionA{\rho}\le 1$ and the volume constraint \eqref{eq:fmo_continuum_volume} are weak-$\star$ closed. It remains to discuss the weak-$\ast$ lower semicontinuity of the objective and the stability of the pointwise \revisionA{energy-bound admissibility} constraint $\tensf{E}(\tens{x})\in\mathcal{A}(\revisionA{\rho}(\tens{x}))$.

\subsubsection{Zeroth-order and Voigt bounds}
For the zeroth-order and Voigt admissible sets defined later in \eqref{eq:fmo_zo_cont} and \eqref{eq:A1_def}, respectively, the graph
\begin{equation}
\mathcal{G}:=\{(\revisionA{\rho},\tensf{E})\in[0,1]\times \mathbb{S}^{d\times d\times d\times d}:\ \tensf{E}\in\mathcal{A}(\revisionA{\rho})\}
\end{equation}
is a closed convex subset of a finite-dimensional space. Consider the induced set of admissible fields
\begin{equation}\label{eq:K_G}
\mathcal{K}_{\mathcal{G}}
:=\big\{(\revisionA{\rho},\tensf{E})\in L^\infty(\Omega) \times L^\infty(\Omega;\mathbb{S}^{d\times d\times d\times d}): 
(\revisionA{\rho}(\tens{x}),\tensf{E}(\tens{x}))\in\mathcal{G}\ \text{for a.e. }\tens{x}\in\Omega\big\}.
\end{equation}
Let $\mathrm{dist}_{\mathcal{G}}$ be the Euclidean distance to $\mathcal{G}$. Since $\mathcal{G}$ is closed and convex, $\mathrm{dist}_{\mathcal{G}}$ is convex and continuous, and, therefore, the functional
\begin{equation}
(\revisionA{\rho},\tensf{E}) \ \mapsto\ \int_\Omega \mathrm{dist}_{\mathcal{G}}\big(\revisionA{\rho}(\tens{x}),\tensf{E}(\tens{x})\big)\,\mathrm{d}\tens{x}
\end{equation}
is weak-$\star$ lower semicontinuous on $L^\infty\times L^\infty$. Hence, if $(\revisionA{\rho_k}, \tensf{E}_k)\in \mathcal{K}_{\mathcal{G}}$ and $(\revisionA{\rho_k},\tensf{E}_k)\rightharpoonup^\star(\revisionA{\rho},\tensf{E})$, then the above integral is $0$ along the sequence and thus also at the limit, which implies $(\revisionA{\rho},\tensf{E})\in\mathcal{K}_{\mathcal{G}}$. Consequently, for these convex bound choices, the full admissible set $\mathcal{D}$ in \eqref{eq:admissible_set_D} is weak-$\star$ closed, and since it is bounded it is weak-$\star$ compact by the Banach--Alaoglu theorem.

It remains to verify lower semicontinuity of the objective. For each loadcase $j$, the compliance admits the dual representation
\begin{equation}\label{eq:compliance_sup}
c_j(\tensf{E})=\sup_{\tens{u}\in\mathcal{U}_0}\Big(2\,\ell_j(\tens{u})-a_{\tensf{E}}(\tens{u},\tens{u})\Big),
\end{equation}
obtained from the minimum principle for the potential energy. For fixed $\tens{u}$, the map $\tensf{E}\mapsto a_{\tensf{E}}(\tens{u},\tens{u})$ is weak-$\star$ continuous (it is linear in $\tensf{E}$), hence \eqref{eq:compliance_sup} is the supremum of weak-$\star$ continuous affine functionals and, therefore, weak-$\star$ lower semicontinuous. Thus, $J=\sum_{j=1}^{\nlc}c_j$ is weak-$\star$ lower semicontinuous as well, and the direct method yields existence of a minimizer $(\revisionA{\rho^\star},\tensf{E}^\star)\in\mathcal{D}$. The associated states $\tens{u}_j(\tensf{E}^\star)$ exist uniquely by the Lax--Milgram theorem.

\subsubsection{Hashin--Shtrikman-type bounds}
For Hashin--Shtrikman-type constraints (with the pointwise admissible set $\mathcal A^{(2)}(\revisionA{\rho})$ defined in \eqref{eq:A2_def}), the feasible set is convex in $\tensf E$ for fixed $\revisionA{\rho}$ but in general not jointly convex in $(\revisionA{\rho},\tensf E)$ (cf.\ Lemma~\ref{lem:nonconvex}). In particular, the pointwise constraint set \eqref{eq:K_G} need not be weak-$\star$ closed. Moreover, even under the uniform ellipticity bounds \eqref{eq:uniform_ellipticity_E}, weak-$\star$ convergence of coefficients is in general not compatible with passing to the limit in the state equation; see, e.g., the discussion in \citep[Sections~2.2--2.3]{Haslinger2010}.

Following \citep{Haslinger2010}, we thus formulate an existence statement in the homogenization topology of $H$-convergence for the stiffness field. Let
\begin{equation}
\mathcal E_{\alpha,\beta}
:=\{\tensf E\in L^\infty(\Omega;\mathbb{S}^{d\times d\times d\times d}) : 
\alpha \tensf{I}_{d\times d\times d\times d} \preceq \tensf{E}(\tens x)\preceq \beta \tensf{I}_{d\times d\times d\times d}\ \text{a.e. in }\Omega\},
\end{equation}
where $\alpha,\beta$ are as in \eqref{eq:uniform_ellipticity_E}. Then any admissible stiffness field belongs to $\mathcal E_{\alpha,\beta}$, and $\mathcal E_{\alpha,\beta}$ is sequentially compact with respect to $H$-convergence (cf.\ \citep[Theorem~2.2]{Haslinger2010}).

Define the $H$-relaxed admissible set as the closure of $\mathcal D$ in the product topology $(\revisionA{\rho_k}\rightharpoonup^\star \revisionA{\rho}) \times (\tensf E_k \to_H \tensf E)$, i.e.,
\begin{equation}
\mathcal D^{H}:=\overline{\mathcal D}^{\,(\rightharpoonup^\star)\times(\to_H)}.
\end{equation}
Let $(\revisionA{\rho_k}, \tensf E_k)\subset\mathcal D^{\revisionA{H}}$ be a minimizing sequence \revisionA{for the relaxed problem}. By the Banach--Alaoglu theorem, we may assume (up to a subsequence) $\revisionA{\rho_k}\rightharpoonup^\star \revisionA{\rho^\star}$ in $L^\infty$. By $H$-compactness of $\mathcal E_{\alpha,\beta}$ we may further assume $\tensf E_k\to_H \tensf E^\star$. Hence $(\revisionA{\rho^\star}, \tensf E^\star)\in\mathcal D^{H}$ \revisionA{because $\mathcal D^H$ is closed in the product topology}.

Finally, along any $H$-convergent sequence $\tensf E_k \xrightarrow{H} \tensf E^\star$, the corresponding state solutions satisfy
\begin{equation}
\tens u_{j}(\tensf E_k)\rightharpoonup \tens u_{j}(\tensf E^\star)\quad\text{in }\mathcal U_0,\qquad j=1,\dots,\nlc,
\end{equation}
and the compliance is lower semicontinuous with respect to this convergence; see the $H$-lower semicontinuity condition \citep[(2.14)]{Haslinger2010} and the fact that the compliance is a typical example satisfying it. Therefore,
\begin{equation}
J(\tensf E^\star)\le \liminf_{k\to\infty} J(\tensf E_k),
\end{equation}
and the direct-method existence argument (cf.\ \citep[Theorem~2.8]{Haslinger2010}) yields existence of a minimizer of the relaxed problem \eqref{eq:fmo_continuum} posed over $\mathcal D^H$. We emphasize that $(\revisionA{\rho^\star}, \tensf E^\star)$ is an $H$-relaxed limit of admissible designs; in general, it need not satisfy the original Hashin--Shtrikman pointwise constraints.

\begin{rmrk}[Orthotropic restriction]
If we further restrict admissible tensors to the orthotropic symmetry class with free rotations, we mean fields $\tensf E$ such that, for a.e.\ $\tens x\in\Omega$, there exist an angle $\varphi(\tens x)$ and an orthotropic base tensor $\tensf E^{\mathrm b}(\tens x)$ with $(\tensf E^{\mathrm b})_{1112}=(\tensf E^{\mathrm b})_{2212}=0$, and $\tensf E(\tens x)$ is obtained from $\tensf E^{\mathrm b}(\tens x)$ by planar rotation; see Section~\ref{sec:fmo_ortho}. \revisionA{The corresponding set of admissible tensor fields} is not, in general, weak-$\star$ closed: weak-$\star$ limits may average rapidly oscillating orientation fields, while convex combinations of differently rotated orthotropic tensors need not remain orthotropic. Existence is, therefore, understood in the relaxed sense, i.e., by minimizing over the closure of orthotropic-admissible designs in $(\rightharpoonup^\star) \times (\to_H)$.

For the model classes considered in this paper:
\begin{itemize}
    \item Zeroth-order bounds: in the single-loadcase continuum setting, at least one minimizer of \eqref{eq:zo_continuum_single} can be chosen orthotropic (Proposition~\ref{prop:zo_orthotropic_representative}). This is a property of the original model without imposing orthotropy a priori; consequently, adding an orthotropic restriction does not increase the optimal value in that setting.
    \item Voigt bounds: isotropic optimal minimizers exist both in the discrete model \eqref{eq:fmo_voigt} and in the continuum Voigt setting (Proposition~\ref{prop:voigt_iso}). Thus, without imposing symmetry a priori, one can choose an isotropic (hence orthotropic) minimizer, and adding an orthotropic restriction does not change the optimal value.
    \item Hashin--Shtrikman bounds: here the orthotropic class is not weak-$\star$ closed, so existence is interpreted via the relaxed closure in $(\rightharpoonup^\star) \times (\to_H)$. Theorem~\ref{th:tight} shows that, in the two-dimensional single-loadcase continuum setting with well-ordered isotropic phases, the relaxed value is attained by rank-one and rank-two sequential laminates, which are orthotropic. Therefore, imposing orthotropy does not change the continuum relaxed optimal value.
\end{itemize}
\end{rmrk}

For numerical solution, we discretize the state space by $\mathcal{U}_h\subset\mathcal{U}_0$ and parametrize $(\revisionA{\rho},\tensf{E})$ by elementwise constants. Then \eqref{eq:strain_bounds_general} becomes a collection of local constraints on the element variables. For the zeroth-order and Voigt bounds, these constraints are jointly convex and can be written as linear matrix inequalities, while Hashin--Shtrikman-type bounds are convex in $\tensf{E}$ for fixed $\revisionA{\rho}$ but not jointly convex in $(\revisionA{\rho},\tensf{E})$. The resulting discrete formulations are given in Sections~\ref{sec:fmo_zero}--\ref{sec:fmo_hashin}.

\subsection{FMO Problem under Zeroth-Order Bounds}\label{sec:fmo_zero}

We start with the weakest pointwise admissible set $\mathcal{A}^{(0)}(\revisionA{\rho})$, obtained by imposing only uniform bounds on the effective elasticity tensor \citep{Lobos2016,Ringertz1993},
\begin{equation}\label{eq:fmo_zo_cont}
    \mathcal{A}^{(0)}(\revisionA{\rho}):=
    \left\{\tensf E\in \mathbb{S}^{d\times d\times d\times d}_{\succ 0}\;:\;
    \Emin \preceq \tensf E \preceq \Emax
    \right\},\qquad \revisionA{\rho}\in[0,1],
\end{equation}
where $\Emin,\Emax$ are the (isotropic) tensors of the soft and stiff constituents such that $\tensf{0} \prec \Emin \prec \Emax$. In contrast to the Voigt and Hashin--Shtrikman bounds considered later, \eqref{eq:fmo_zo_cont} does not couple $\tensf E$ to a local volume-fraction variable $\revisionA{\rho}$. \revisionA{In what follows,} we first derive the discrete FMO formulation under zeroth-order bounds and then return to the single-loadcase continuum setting to show that an orthotropic minimizer exists.

For numerical treatment, we now pass to a finite-element discretization of $\Omega$ into $\nel$ elements and approximate the effective tensor field by elementwise constants $\Ei$ \revisionA{that are collected into the ordered tuple $\boldsymbol{\mathcal E}:=(\Ei)_{i=1}^{\nel}$}. Following the standard trace-constrained resource modeling in free material design \citep{Ringertz1993,Bendsoe1995,Zowe1997,BenTal2000,Czarnecki2012,Czarnecki2014}, and since \eqref{eq:fmo_zo_cont} does not include an explicit local volume fraction variable, we introduce auxiliary elementwise volume fractions $\revisionA{\rho_i}\in[0,1]$ \revisionA{collected into the vector $\sevek{\rho}:=(\rho_1,\ldots,\rho_{\nel})^\trn \in\mathbb R^{\nel}$} and enforce a local trace-level mixture inequality
\begin{equation}\label{eq:fmo_zo_volume_local}
\Tr(\Ei)\le \revisionA{\rho_i}\Tr(\Emax) + (1-\revisionA{\rho_i})\Tr(\Emin),\qquad i=1,\dots,\nel.
\end{equation}
Assuming equisized elements and imposing the average volume-fraction bound $\frac{1}{\nel}\sum_{i=1}^{\nel} \revisionA{\rho_i}\le \overline{V}$, summing \eqref{eq:fmo_zo_volume_local} over all elements gives
\begin{equation}
\frac{1}{\nel}\sum_{i=1}^{\nel}\Tr(\Ei)
\le \overline{V}\Tr(\Emax)+(1-\overline{V})\Tr(\Emin).
\end{equation}
Conversely, if $\Emin\preceq \Ei\preceq \Emax$ and the global trace bound above holds, then defining
\begin{equation}
\revisionA{\rho_i}:=\frac{\Tr(\Ei)-\Tr(\Emin)}{\Tr(\Emax)-\Tr(\Emin)}\in[0,1]
\end{equation}
yields \eqref{eq:fmo_zo_volume_local} with equality, and $\frac1{\nel}\sum_i \revisionA{\rho_i}\le \overline V$. Hence, after eliminating $\revisionA{\sevek{\rho}}$, this auxiliary formulation is equivalent to the global trace constraint \eqref{eq:fmo_volume}.

For each loadcase $j=1,\dots,\nlc$, let $\sevek f_j\in\mathbb R^{n_{\mathrm{dof},j}}$ be the discrete load vector and $\semtrx K_j(\revisionA{\boldsymbol{\mathcal E}})\in\mathbb{S}^{n_{\mathrm{dof},j}\times n_{\mathrm{dof},j}}_{\succ 0}$ the assembled stiffness matrix. The multi-loadcase compliance admits the standard Schur-complement epigraph reformulation as a semidefinite program \citep[Appendix~A]{Boyd2004}, as commonly used in free material and compliance optimization \citep{BenTal2000,Kocvara2020}. The resulting minimum-compliance free material optimization problem reads
\begin{subequations}\label{eq:fmo}
\begin{align}
    \min_{\revisionA{\sevek{c}}, \revisionA{\boldsymbol{\mathcal E}}} \quad & \sum_{j=1}^{\nlc} c_j
    \label{eq:fmo_compliance}\\
    \mathrm{s.t.}\quad &
    \begin{pmatrix}
        c_j & -\sevek{f}_j^\trn\\
        -\sevek{f}_j & \semtrx{K}_j\!\left(\revisionA{\boldsymbol{\mathcal E}}\right)
    \end{pmatrix}\succeq \semtrx{0}, \qquad j = 1,\dots,\nlc,
    \label{eq:fmo_lmi}\\
    & \Emin \preceq \Ei \preceq \Emax, \qquad i = 1,\dots,\nel,
    \label{eq:fmo_zo}\\
    & \frac{1}{\nel}\sum_{i=1}^{\nel} \Tr(\Ei)
      \le \overline{V}\,\Tr(\Emax) + (1-\overline{V})\,\Tr(\Emin),
    \label{eq:fmo_volume}
\end{align}
\end{subequations}
\revisionA{with the vector $\sevek{c}:=(c_1,\ldots,c_{\nlc})^\trn\in\mathbb R^{\nlc}$.} The linear matrix inequality \eqref{eq:fmo_lmi} enforces $c_j \ge \sevek{f}_j^\trn \semtrx{K}_j(\revisionA{\boldsymbol{\mathcal E}})^{-1} \sevek{f}_j$ by the Schur complement, so that at optimality $c_j$ equals the compliance of loadcase $j$. The tensor bound \eqref{eq:fmo_zo} is equivalent to the strain-energy inequalities
\begin{equation}
\langle \Emin \strain,\strain\rangle \le \langle \Ei \strain,\strain\rangle \le \langle \Emax \strain,\strain\rangle,
\qquad \forall\,\strain\in\mathbb{S}^{d\times d},
\end{equation}
i.e., no element can be stiffer than $\Emax$ or softer than $\Emin$ in any direction.

Problem \eqref{eq:fmo} is a convex finite-dimensional semidefinite program and can, therefore, be solved to global optimality. We refer to \eqref{eq:fmo} as the zeroth-order free material optimization model (ZO-FMO). Its main drawback is that the admissible set \eqref{eq:fmo_zo} is typically very large; in particular, it does not encode any microstructural information beyond the uniform ellipticity range $\Emin\preceq\Ei\preceq\Emax$.

Before introducing tighter bounds, we record a structural property that connects the zeroth-order (ZO) model to our orthotropic focus at the continuum level. In the single-loadcase setting, the continuum ZO problem
\begin{subequations}\label{eq:zo_continuum_single}
\begin{align}
\inf_{\tensf E\in L^\infty(\Omega;\mathbb{S}^{d\times d\times d\times d}_{\succ0})}\quad
& c_1(\tensf E)\\
\text{s.t.}\quad
& \Emin \preceq \tensf E(\tens x)\preceq \Emax \ \text{for a.e. }\tens x\in\Omega,\\
&
\frac{1}{|\Omega|}\int_\Omega \Tr(\tensf E(\tens x))\,\mathrm d\tens x
\le \overline V\Tr(\Emax)+(1-\overline V)\Tr(\Emin)
\end{align}
\end{subequations}
is the reference model below; we will show that it admits an orthotropic optimal solution.

For later use, recall the complementary-energy representation (minimum complementary energy principle): for $\nlc=1$ and any admissible $\tensf E$,
\begin{equation}\label{eq:zo_complementary}
c_1(\tensf E)
=
\min_{\stress\in\Sigma_{\mathrm{ad},1}}
\int_\Omega \left\langle \tensf E(\tens x)^{-1}\stress(\tens x),\stress(\tens x)\right\rangle\,\mathrm d\tens x,
\end{equation}
where
\begin{equation}\label{eq:zo_sigma_ad}
\Sigma_{\mathrm{ad},1}
:=
\left\{
\stress\in L^2(\Omega;\mathbb S^{d\times d})\;:\;
\int_\Omega \langle \stress,\strain(\tens v)\rangle\,\mathrm d\tens x=\ell_1(\tens v)\quad \forall \tens v\in\mathcal U_0
\right\}.
\end{equation}
The minimum in \eqref{eq:zo_complementary} is attained. Indeed, $\Sigma_{\mathrm{ad},1}$ is a closed affine subspace of $L^2(\Omega;\mathbb S^{d\times d})$, and the functional $\stress\mapsto\int_\Omega \langle \tensf E^{-1}\stress,\stress\rangle\,\mathrm d\tens x$ is coercive and strictly convex for admissible $\tensf E$.

The structural role of principal directions in compliance-driven design is classical in orientation optimization of orthotropic materials \citep{Pedersen1989} and in homogenization-based laminate constructions \citep{Allaire1993,Allaire2002}. The next proposition gives a corresponding representative-minimizer statement directly for the single-loadcase continuum zeroth-order FMO model \eqref{eq:zo_continuum_single}.

\begin{prpstn}[Orthotropic minimizer for the single-loadcase continuum ZO problem]\label{prop:zo_orthotropic_representative}
Assume $\nlc=1$ and $d=2$. Then \eqref{eq:zo_continuum_single} admits a minimizer, and at least one minimizer $\tensf E^{\mathrm{ortho}}$ can be chosen such that, for a.e.\ $\tens x\in\Omega$, $\tensf E^{\mathrm{ortho}}(\tens x)$ is orthotropic in a principal basis of an optimal stress field for \eqref{eq:zo_complementary}.
\begin{proof}
Existence of a minimizer for \eqref{eq:zo_continuum_single} follows from Section~\ref{sec:cont_fmo}. Let $\tensf E^\star$ be such a minimizer.

Pick $\stress^\star\in\Sigma_{\mathrm{ad},1}$ minimizing \eqref{eq:zo_complementary} for $\tensf E^\star$. Using the spectral decomposition in 2D, choose a measurable orthogonal field $\semtrx U(\tens x)$ that diagonalizes $\stress^\star(\tens x)$ for a.e.\ $\tens x$ (and set $\semtrx U=\semtrx I$ on repeated-eigenvalue points). Set $\semtrx M:=\mathrm{Diag}(1,-1)$ and $\semtrx Q:=\semtrx U\semtrx M\semtrx U^\trn$. For any fourth-order tensor $\tensf A$, denote by $\mathcal T_{\semtrx Q}(\tensf A)$ its orthogonal transform:
$\big(\mathcal T_{\semtrx Q}(\tensf A)\big)_{ijkl}:=\sum_{p,q,r,s=1}^{d}(\semtrx{Q})_{ip}(\semtrx{Q})_{jq}(\semtrx{Q})_{kr}(\semtrx{Q})_{ls}(\tensf{A})_{pqrs}$.
Define
\begin{equation}
\tensf E^{\mathrm{ortho}}
:=\tfrac12\big(\tensf E^\star+\mathcal T_{\semtrx Q}(\tensf E^\star)\big)
\quad\text{pointwise in }\Omega.
\end{equation}
Because $\Emin,\Emax$ are isotropic, $\mathcal T_{\semtrx Q}$ preserves the bounds $\Emin\preceq\cdot\preceq\Emax$; by convexity, $\tensf E^{\mathrm{ortho}}$ is feasible for \eqref{eq:zo_continuum_single}. Also $\Tr(\mathcal T_{\semtrx Q}(\tensf A))=\Tr(\tensf A)$, so the trace constraint is unchanged.

For fixed $\stress^\star(\tens x)$, the map $\tensf A\mapsto\langle \tensf A^{-1}\stress^\star,\stress^\star\rangle$ is convex on positive-definite tensors. Moreover, by orthogonality and $\semtrx Q^\trn\stress^\star\semtrx Q=\stress^\star$, \revisionA{we obtain} $\big(\mathcal T_{\semtrx Q}(\tensf A)\big)^{-1} = \mathcal T_{\semtrx Q}\!\left(\tensf A^{-1}\right)$ and
\begin{equation}
\left\langle \big(\mathcal T_{\semtrx Q}(\tensf E^\star)\big)^{-1}\stress^\star,\stress^\star\right\rangle
=
\left\langle (\tensf E^\star)^{-1}\stress^\star,\stress^\star\right\rangle.
\end{equation}
Hence, pointwise in $\Omega$,
\begin{equation}
\left\langle (\tensf E^{\mathrm{ortho}})^{-1}\stress^\star,\stress^\star\right\rangle
\le
\tfrac12\left\langle (\tensf E^\star)^{-1}\stress^\star,\stress^\star\right\rangle
+\tfrac12\left\langle \big(\mathcal T_{\semtrx Q}(\tensf E^\star)\big)^{-1}\stress^\star,\stress^\star\right\rangle
=
\left\langle (\tensf E^\star)^{-1}\stress^\star,\stress^\star\right\rangle.
\end{equation}
Integrating gives
\begin{equation}
\int_\Omega \left\langle (\tensf E^{\mathrm{ortho}})^{-1}\stress^\star,\stress^\star\right\rangle \mathrm d\tens x
\le
\int_\Omega \left\langle (\tensf E^\star)^{-1}\stress^\star,\stress^\star\right\rangle \mathrm d\tens x
=c_1(\tensf E^\star).
\end{equation}
By \eqref{eq:zo_complementary},
\begin{equation}
c_1(\tensf E^{\mathrm{ortho}})
\le
\int_\Omega \left\langle (\tensf E^{\mathrm{ortho}})^{-1}\stress^\star,\stress^\star\right\rangle \mathrm d\tens x
\le c_1(\tensf E^\star).
\end{equation}
Since $\tensf E^\star$ is optimal for \eqref{eq:zo_continuum_single}, we conclude $c_1(\tensf E^{\mathrm{ortho}})=c_1(\tensf E^\star)$, so $\tensf E^{\mathrm{ortho}}$ is also optimal. In the principal basis of $\stress^\star$, averaging with $\semtrx M$ cancels the shear-coupling terms, hence $\tensf E^{\mathrm{ortho}}$ is orthotropic.
\end{proof}
\end{prpstn}

Proposition~\ref{prop:zo_orthotropic_representative} is a global continuum statement for single-loadcase ZO-FMO: at least one minimizer can be chosen orthotropic. This argument is continuum-level and does not transfer directly to fixed finite-element discretizations. For low-order elements, stresses can be elementwise constant, but such discretizations are prone to checkerboarding unless regularized; this regularization averages neighboring tensors and generally destroys local orthotropic structure. Using higher-order elements mitigates these artifacts, but stresses are then no longer elementwise constant, so local tensors must be optimal with respect to multiple stress states, which again typically breaks orthotropy. 

We now tighten ZO-FMO by adding Voigt upper bounds; in that case, we obtain an even stronger structural result (existence of isotropic optimal minimizers, Proposition~\ref{prop:voigt_iso}).

\subsection{FMO Problem under Voigt Upper Bounds}\label{sec:fmo_voigt}

To strengthen the connection between the FMO problem and the realizable elasticity tensors without sacrificing convexity, we proceed by exploiting the Voigt upper bound \citep{Hill1963} for effective elasticity tensors of composites
\begin{equation}\label{eq:voigt_bound}
    \Ei \preceq \EV := (1-\revisionA{\rho_i}) \Emin + \revisionA{\rho_i} \Emax,
\end{equation}
where $\revisionA{\rho_i}$ is the volume fraction of the stiffer material $\Emax$ in the $i$-th element. The inequality \eqref{eq:voigt_bound} is equivalent to the strain-energy bound $\langle \Ei\strain,\strain\rangle \le \langle \EV \strain,\strain\rangle$ for all $\strain\in\mathbb{S}^{d\times d}$. Moreover, since $\revisionA{\rho_i}\in[0,1]$ and $\Emin\preceq \Emax$, we also have $\EV\preceq \Emax$. The Voigt constraint \eqref{eq:voigt_bound} thus induces the first-order admissible set 
\begin{equation}\label{eq:A1_def}
\mathcal A^{(1)}(\revisionA{\rho})
:=\Big\{\tensf E\in\mathbb{S}^{d\times d\times d\times d}:\ \Emin \preceq \tensf E \preceq (1-\revisionA{\rho})\Emin + \revisionA{\rho}\Emax\Big\},\quad \revisionA{\rho}\in[0,1].
\end{equation}

In what follows, we exploit the Voigt upper bound \eqref{eq:voigt_bound} to provide a better estimate of the microstructural relative volume. In particular, due to the Loewner ordering, we have
\begin{equation}
    \Tr(\Ei) \le \Tr((1-\revisionA{\rho_i}) \Emin + \revisionA{\rho_i} \Emax) \le \Tr(\Emax),
\end{equation}
and, thus, also for the average of the elasticity tensors $\Ei$ over all elements, assuming equisized elements for simplicity,
\begin{equation}\label{eq:fmo_volume_average}
    \frac{1}{\nel}\sum_{i=1}^{\nel} \Tr(\Ei) \le \Tr(\Emin) + \frac{1}{\nel}[\Tr(\Emax)-\Tr(\Emin)] \sum_{i=1}^{\nel} \revisionA{\rho_i} \le \Tr(\Emax).
\end{equation}
Now, allowing the average trace to be at most $\overline{V}\Tr(\Emax) + (1-\overline{V}) \Tr(\Emin)$ as in \eqref{eq:fmo_volume}, the tighter volume estimator reads as
\begin{equation}\label{eq:fmo_voigt_volume_original}
    \Tr(\Emin) + \frac{1}{\nel}[\Tr(\Emax)-\Tr(\Emin)] \sum_{i=1}^{\nel} \revisionA{\rho_i} \le \overline{V} \Tr(\Emax) + (1-\overline{V}) \Tr(\Emin),
\end{equation}
which, after simplification, provides the Voigt free material optimization problem (V-FMO)
\begin{subequations}\label{eq:fmo_voigt}
\begin{align}
    \min_{\revisionA{\sevek{c}}, \revisionA{\boldsymbol{\mathcal E}}, \revisionA{\sevek{\rho}}} \; & \sum_{j=1}^{\nlc} c_j \label{eq:fmo_voigt_compliance}\\
    \mathrm{s.t.}\; & \begin{pmatrix}
        c_j & -\sevek{f}_j^\trn\\
        -\sevek{f}_j & \semtrx{K}_j\left(\revisionA{\boldsymbol{\mathcal E}}\right)
    \end{pmatrix}\succeq \semtrx{0}, \quad j = 1,\dots,\nlc,\label{eq:fmo_voigt_lmi}\\
    & \Emin \preceq \Ei \preceq (1-\revisionA{\rho_i}) \Emin + \revisionA{\rho_i} \Emax, \quad i = 1,\dots,\nel,\label{eq:fmo_voigt_fo}\\
    & \frac{1}{\nel}\sum_{i=1}^{\nel} \revisionA{\rho_i} \le \overline{V},\label{eq:fmo_voigt_volume}\\
    & 0 \le \revisionA{\rho_i} \le 1, \quad i = 1,\dots,\nel.
\end{align}
\end{subequations}
Notice that the constraint \eqref{eq:fmo_voigt_volume} now directly limits the average volume fraction of the stiffer material $\Emax$.

\begin{figure}[!b]
    \includegraphics[width=\linewidth]{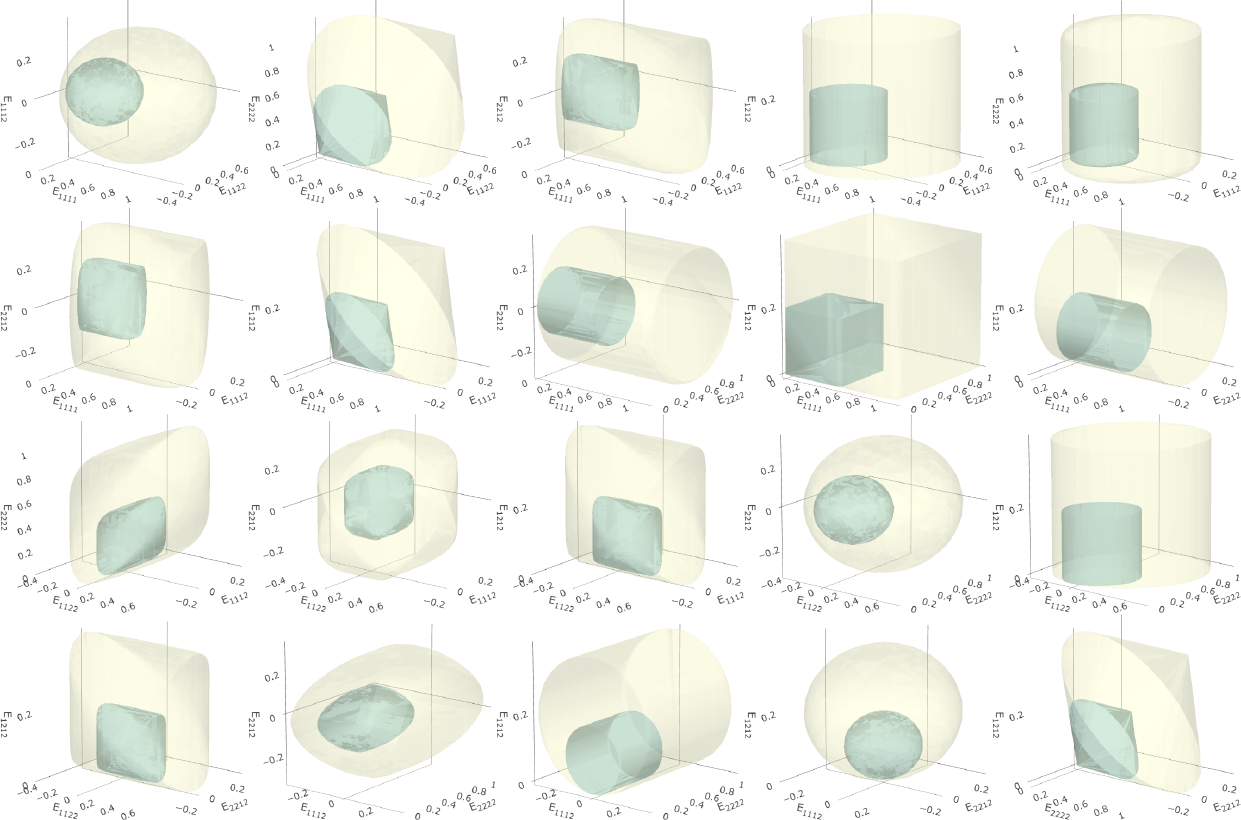}
    \caption{
        Comparison of the admissible sets $\mathcal A^{(0)}(0.5)$ (zeroth-order, yellow) and $\mathcal A^{(1)}(0.5)$ (Voigt, blue). The sets are shown in a selected three-dimensional projection of $\Ei$ as outer-envelope boundary surfaces, obtained from $750$ sampled strains $\strain$ with Frobenius norm $\sqrt{2}/2$. Material parameters: $\kappa^- = 0.714\times10^{-9}$, $\kappa^+ = 0.714$, $\mu^- = 0.385\times10^{-9}$, and $\mu^+ = 0.385$.
    }
    \label{fig:zero_voigt}
\end{figure}

The V-FMO problem \eqref{eq:fmo_voigt} is indeed tighter than the FMO problem under zeroth-order bounds \eqref{eq:fmo}. In particular, for $\overline{V} \in (0,1)$ and after projecting out the auxiliary variables $\revisionA{\sevek{\rho}}$, the feasible set in $(\revisionA{\sevek{c}},\revisionA{\boldsymbol{\mathcal E}})$-space induced by \eqref{eq:fmo_voigt} is strictly contained in the feasible set of \eqref{eq:fmo}. Figure~\ref{fig:zero_voigt} provides pointwise geometric intuition at fixed $\revisionA{\rho_i}$.

\begin{prpstn}[Strict feasible-set inclusion of V-FMO in ZO-FMO]\label{prop:voigt_inclusion}
    Assume $0<\overline{V}<1$ and $\Emin\prec\Emax$, i.e., the phase bulk and shear moduli satisfying $0<\kappa^-<\kappa^+$ and $0<\mu^-<\mu^+$. With
    $
    \mathcal F_{\mathrm{ZO}}
    :=\{(\revisionA{\sevek{c}},\revisionA{\boldsymbol{\mathcal E}}):\ (\revisionA{\sevek{c}},\revisionA{\boldsymbol{\mathcal E}})\ \text{feasible for \eqref{eq:fmo}}\},
    $ and $\mathcal F_{\mathrm{FO}}
    :=\{(\revisionA{\sevek{c}},\revisionA{\boldsymbol{\mathcal E}}):\ \exists\revisionA{\sevek{\rho}}\ \text{s.t.}\ (\revisionA{\sevek{c}},\revisionA{\boldsymbol{\mathcal E}},\revisionA{\sevek{\rho}})\ \text{feasible for \eqref{eq:fmo_voigt}}\},$ it holds that $\mathcal F_{\mathrm{FO}}\subsetneq\mathcal F_{\mathrm{ZO}}$.

    \begin{proof}
        Inclusion $\mathcal F_{\mathrm{FO}}\subseteq\mathcal F_{\mathrm{ZO}}$ is immediate: \eqref{eq:fmo_voigt_lmi} is the same as \eqref{eq:fmo_lmi}, \eqref{eq:fmo_voigt_fo} with $\revisionA{\rho_i}\in[0,1]$ implies \eqref{eq:fmo_zo}, and
        \begin{equation}
        \frac1\nel\sum_{i=1}^{\nel}\Tr(\Ei)\le
        \Tr(\Emin)+\frac{\Tr(\Emax)-\Tr(\Emin)}{\nel}\sum_{i=1}^{\nel}\revisionA{\rho_i}
        \le \overline{V}\Tr(\Emax)+(1-\overline{V})\Tr(\Emin),
        \end{equation}
        so \eqref{eq:fmo_volume} holds.

        For strictness, write the bulk and shear moduli differences as $\Delta\kappa:=\kappa^+-\kappa^->0$, $\Delta\mu:=\mu^+-\mu^->0$, and
        \begin{equation}
        \alpha:=\frac{\Delta\kappa}{\Delta\kappa+2\Delta\mu}\in(0,1),\qquad
        \theta:=\min\{1,\overline{V}/\alpha\},
        \end{equation}
        so $\theta>\overline{V}$ and $\alpha\theta\le\overline{V}$.
        Define an isotropic tensor $\widehat{\tensf E}$ by $\widehat\kappa:=\kappa^-+\theta\Delta\kappa$ and $\widehat\mu:=\mu^-$, and set $\Ei=\widehat{\tensf E}$ for all $i$. Then $\Emin\preceq\Ei\preceq\Emax$. Moreover, for $d=2$,
        \begin{equation}
        \frac{\Tr(\Ei)-\Tr(\Emin)}{\Tr(\Emax)-\Tr(\Emin)}
        =\frac{2\theta\Delta\kappa}{2\Delta\kappa+4\Delta\mu}
        =\alpha\theta\le\overline{V},
        \end{equation}
        hence \eqref{eq:fmo_volume} holds. Moreover, choosing $c_j:=\sevek f_j^\trn\semtrx K_j(\revisionA{\boldsymbol{\mathcal E}})^{-1}\sevek f_j$ gives \eqref{eq:fmo_lmi}, so $(\revisionA{\sevek{c}},\revisionA{\boldsymbol{\mathcal E}})\in\mathcal F_{\mathrm{ZO}}$.

        If $(\revisionA{\sevek{c}},\revisionA{\boldsymbol{\mathcal E}})\in\mathcal F_{\mathrm{FO}}$, then for some $\revisionA{\sevek{\rho}}$, \eqref{eq:fmo_voigt_fo} implies $\kappa^-+\theta\Delta\kappa=\widehat\kappa\le \kappa^-+\revisionA{\rho_i}\Delta\kappa$, hence $\revisionA{\rho_i}\ge\theta$ for all $i$. Therefore $\frac1\nel\sum_i \revisionA{\rho_i}\ge\theta>\overline{V}$, contradicting \eqref{eq:fmo_voigt_volume}. Thus, $(\revisionA{\sevek{c}},\revisionA{\boldsymbol{\mathcal E}})\notin\mathcal F_{\mathrm{FO}}$, and $\mathcal F_{\mathrm{FO}}\subsetneq\mathcal F_{\mathrm{ZO}}$.
    \end{proof}
\end{prpstn}

The endpoint cases satisfy $\mathcal F_{\mathrm{FO}}=\mathcal F_{\mathrm{ZO}}$ for $\overline{V}\in\{0,1\}$. Indeed, for $\overline{V}=0$, the Voigt volume constraint \eqref{eq:fmo_voigt_volume} forces $\revisionA{\rho_i}=0$ for all $i$, hence \eqref{eq:fmo_voigt_fo} reduces to $\Emin\preceq\Ei\preceq\Emin$, i.e., $\Ei=\Emin$, which is exactly what \eqref{eq:fmo_volume} enforces in the zeroth-order formulation. For $\overline{V}=1$, the volume constraints are inactive in both models, and any zeroth-order-feasible tensor field also satisfies the Voigt bounds by taking $\revisionA{\rho_i}=1$ for all $i$.

For the Voigt-constrained problem \eqref{eq:fmo_voigt}, we now show that isotropic optimal minimizers exist in both the discrete and continuum settings.

\begin{prpstn}[Existence of isotropic minimizers for V-FMO (discrete and continuum)]\label{prop:voigt_iso}
Let $\Emin$ and $\Emax$ be well-ordered isotropic elasticity tensors with $\Emin \preceq \Emax$. Then:
\begin{enumerate}
    \item The discrete V-FMO problem \eqref{eq:fmo_voigt} admits an optimal solution in which each elasticity tensor saturates the upper bound in \eqref{eq:fmo_voigt_fo}, i.e.,
    \[
        \Ei^\star = \EVs := (1 - \revisionA{\rho_i^\star}) \Emin + \revisionA{\rho_i^\star} \Emax, \quad \text{for all } i = 1,\dots,\nel.
    \]
    In particular, there always exists an optimal discrete solution consisting entirely of isotropic tensors.
    \item In the continuum setting of Section~\ref{sec:cont_fmo}, if $(\revisionA{\rho},\tensf E)$ satisfies the Voigt pointwise admissibility $\tensf E(\tens x)\in \mathcal A^{(1)}(\revisionA{\rho}(\tens x))$ for a.e.\ $\tens x\in\Omega$, with $\mathcal A^{(1)}$ from \eqref{eq:A1_def}, and
    \[
        \widetilde{\tensf E}(\tens x):=(1-\revisionA{\rho}(\tens x))\Emin+\revisionA{\rho}(\tens x)\Emax,
    \]
    then $(\revisionA{\rho}, \widetilde{\tensf E})$ is Voigt-admissible and $J(\widetilde{\tensf E})\le J(\tensf E)$. In particular, one may always choose an isotropic minimizer in the continuum Voigt-constrained model.
\end{enumerate}
\begin{proof}
For part~1, existence of a discrete minimizer for \eqref{eq:fmo_voigt} follows from finite-dimensional compactness: by \eqref{eq:fmo_voigt_fo} and $0\le \revisionA{\rho_i}\le1$, all $\Ei$ are uniformly bounded between $\Emin$ and $\Emax$, and for fixed $(\revisionA{\sevek{\rho}},\revisionA{\boldsymbol{\mathcal E}})$ the minimal admissible $c_j$ is given by the Schur-complement value $\sevek{f}_j^\trn \semtrx K_j(\revisionA{\boldsymbol{\mathcal E}})^{-1}\sevek{f}_j$. Let $\left(\revisionA{\boldsymbol{\mathcal E}^\star},\, \revisionA{\sevek{c}^\star},\, \revisionA{\sevek{\rho}^\star}\right)$ be such an optimal solution, and suppose that for some index $i$ we have $\Ei^\star \neq \EVs$. Since $\Ei^\star \preceq \EVs$, there exists a symmetric positive semidefinite tensor $\tensf{F}_i \succeq \tensf{0}$, $\tensf{F}_i \neq \tensf{0}$ such that $\EVs = \Ei^\star + \tensf{F}_i$.

\revisionA{Define a modified design $\widetilde{\boldsymbol{\mathcal E}}:=(\widetilde{\tensf E}_{\ell})_{\ell=1}^{\nel}$ by setting $\widetilde{\tensf E}_{i}:=\EVs$ and $\widetilde{\tensf E}_{\ell}:=\tensf E_{\ell}^{\star}$ for $\ell\neq i$, while keeping all other variables unchanged.} The constraints \eqref{eq:fmo_voigt_fo} and \eqref{eq:fmo_voigt_volume} remain satisfied by construction. Since the element stiffness contributions are linear in $\Ei$, the assembled matrix $\semtrx K_j$ is linear in $\revisionA{\boldsymbol{\mathcal E}}$. Therefore, $\tilde{\tensf{E}}_i\succeq \Ei^\star$ implies $\semtrx K_j(\revisionA{\widetilde{\boldsymbol{\mathcal E}}})\succeq \semtrx K_j(\revisionA{\boldsymbol{\mathcal E}^\star})$ and, thus,
\begin{equation}
\begin{pmatrix}
    c_j^\star & -\sevek{f}_j^\trn \\
    -\sevek{f}_j & \semtrx{K}_j(\revisionA{\widetilde{\boldsymbol{\mathcal E}}})
\end{pmatrix}
\succeq
\begin{pmatrix}
    c_j^\star & -\sevek{f}_j^\trn \\
    -\sevek{f}_j & \semtrx{K}_j(\revisionA{\boldsymbol{\mathcal E}^\star})
\end{pmatrix}
\succeq \semtrx{0}.
\end{equation}
By the Schur complement lemma, this implies $c_j^\star \ge \sevek{f}_j^\trn \semtrx{K}_j(\revisionA{\widetilde{\boldsymbol{\mathcal E}}})^{-1} \sevek{f}_j$, so the compliance remains at most $c_j^\star$, and the updated design is feasible. Repeating this replacement for all $i$ such that $\Ei^\star \neq \EVs$ yields a fully isotropic design, in which all elasticity tensors lie on the Voigt upper bound.

For part~2, from $\tensf E(\tens x)\in \mathcal A^{(1)}(\revisionA{\rho}(\tens x))$ and \eqref{eq:A1_def}, we have $\Emin\preceq \tensf E(\tens x)\preceq \widetilde{\tensf E}(\tens x)$ for a.e.\ $\tens x$, so $(\revisionA{\rho},\widetilde{\tensf E})$ is feasible. Moreover, for every loadcase $j$ and every $\tens u\in\mathcal U_0$, $a_{\widetilde{\tensf E}}(\tens u,\tens u)\ge a_{\tensf E}(\tens u,\tens u)$. Using the dual representation \eqref{eq:compliance_sup},
\begin{equation}
c_j(\widetilde{\tensf E})
=\sup_{\tens u\in\mathcal U_0}\big(2\ell_j(\tens u)-a_{\widetilde{\tensf E}}(\tens u,\tens u)\big)
\le
\sup_{\tens u\in\mathcal U_0}\big(2\ell_j(\tens u)-a_{\tensf E}(\tens u,\tens u)\big)
=c_j(\tensf E).
\end{equation}
Summing over $j=1,\dots,\nlc$ yields $J(\widetilde{\tensf E})\le J(\tensf E)$. Since existence of a continuum minimizer for the Voigt model is established in Section~\ref{sec:cont_fmo}, applying this isotropic replacement to any minimizer yields an isotropic minimizer.
\end{proof}
\end{prpstn}

Using the discrete part of Proposition~\ref{prop:voigt_iso}, optimal elasticities can be parameterized directly by the Voigt mixture. Hence, we reparametrize the optimized elasticities as
\begin{equation}
    \Ei (\revisionA{\rho_i}) := \Emin + (\Emax - \Emin) \revisionA{\rho_i}
\end{equation}
so that the constraint \eqref{eq:fmo_voigt_fo} is always satisfied. This allows us to eliminate the tensor variables and parametrize the tensors only using \revisionA{the volume-fraction vector $\sevek{\rho}$, with $\boldsymbol{\mathcal E}(\sevek{\rho}):=\bigl(\Ei(\rho_i)\bigr)_{i=1}^{\nel}$}. Then, the reduced version of the free material optimization problem under Voigt upper bounds reads as
\begin{subequations}\label{eq:fmo_voigtVTS}
\begin{align}
    \min_{\revisionA{\sevek{c}}, \revisionA{\sevek{\rho}}} \; & \sum_{j=1}^{\nlc} c_j \label{eq:fmo_voigtVTS_compliance}\\
    \mathrm{s.t.}\; & \begin{pmatrix}
        c_j & -\sevek{f}_j^\trn\\
        -\sevek{f}_j & \semtrx{K}_j\left(\revisionA{\boldsymbol{\mathcal E}(\sevek{\rho})}\right)
    \end{pmatrix}\succeq \semtrx{0}, \quad j = 1,\dots,\nlc,\label{eq:fmo_voigtVTS_lmi}\\
    & \frac{1}{\nel}\sum_{i=1}^{\nel} \revisionA{\rho_i} \le \overline{V},\label{eq:fmo_voigtVTS_volume}\\
    & 0 \le \revisionA{\rho_i} \le 1, \quad i = 1,\dots,\nel.
\end{align}
\end{subequations}
This corresponds to the standard variable-thickness-sheet problem \citep{Rossow1973} as also observed by \citet{Bendsoe1999}. We refer to \eqref{eq:fmo_voigtVTS} as the reduced (equivalent) V-FMO formulation.

\subsection{FMO Problem under Hashin--Shtrikman Upper Bounds}\label{sec:fmo_hashin}

Although the Voigt upper bound presented in the previous section tightened the space of elasticity tensors $\Ei$, it can be further sharpened by exploiting the second-order microstructure characteristics, provided by the Hashin--Shtrikman (upper) bounds. Nevertheless, there is an associated cost: loss of convexity of the optimization problem (Lemma \ref{lem:nonconvex}).

\subsubsection{Finite-Dimensional Semi-Infinite Formulation}

\revisionA{We use the classical optimal upper bounds on elastic energy for two-dimensional composites \citep{Allaire1993,Allaire1993wo}; related translation-method formulations are discussed in \citep{Cherkaev2000variational,Lewinski2000}. Rearranging the bound separates the Voigt contribution from a nonnegative correction:}
\begin{equation}\label{eq:hashin_shtrikman}
    \forall \strain \in \mathbb{S}^{2\times2}: \langle \Ei \strain, \strain \rangle \le f^\mathrm{HS}(\strain;\revisionA{\rho_i}),\quad f^\mathrm{HS}(\strain;\revisionA{\rho_i}):=\langle [(1-\revisionA{\rho_i})\Emin+\revisionA{\rho_i}\Emax]\,\strain,\strain\rangle
- q(\strain;\revisionA{\rho_i}),
\end{equation}
where $\strain$ is the strain tensor and $q(\strain;\revisionA{\rho_i})$ is \revisionA{the correction, whose nonnegativity is established in
Proposition~\ref{prop:q_ge_0}. Omitting this correction recovers the Voigt upper bound \eqref{eq:voigt_bound}.} In the notation of \eqref{eq:hierarchy_bounds}, the Hashin--Shtrikman-type admissible set reads
\begin{equation}\label{eq:A2_def}
\mathcal A^{(2)}(\revisionA{\rho})
:=\Big\{\tensf E\in\mathbb{S}^{2\times2\times2\times2}:\ \Emin\preceq \tensf E,\ 
\langle \tensf E \strain,\strain\rangle \le f^{\mathrm{HS}}(\strain;\revisionA{\rho}),\ \ \forall\,\strain\in\mathbb{S}^{2\times2}\Big\}.
\end{equation}
Thus, \eqref{eq:fmo_hashin_so} is the elementwise constraint $\Ei\in \mathcal A^{(2)}(\revisionA{\rho_i})$.

In contrast to the Voigt upper bound, Hashin--Shtrikman upper bounds are tight in the sense that, for each strain $\strain$, there exists a sequentially laminated composite whose effective tensor attains equality in \eqref{eq:hashin_shtrikman} \citep{Allaire1993wo,Allaire1993}. The corresponding optimization FMO problem then writes as
\begin{subequations}\label{eq:fmo_hashin}
\begin{align}
    \min_{\revisionA{\sevek{c}}, \revisionA{\boldsymbol{\mathcal E}}, \revisionA{\sevek{\rho}}} \; & \sum_{j=1}^{\nlc} c_j \label{eq:fmo_hashin_compliance}\\
    \mathrm{s.t.}\; & \begin{pmatrix}
        c_j & -\sevek{f}_j^\trn\\
        -\sevek{f}_j & \semtrx{K}_j\left(\revisionA{\boldsymbol{\mathcal E}}\right)
    \end{pmatrix}\succeq \semtrx{0}, \quad j = 1,\dots,\nlc,\label{eq:fmo_hashin_lmi}\\
    & \Emin \preceq \Ei,\quad i = 1,\dots,\nel,\\
    & \langle \Ei \strain, \strain \rangle \le \big\langle \big[(1-\revisionA{\rho_i})\Emin+\revisionA{\rho_i}\,\Emax\big]\strain,\strain\big\rangle - q(\strain;\revisionA{\rho_i}), i = 1,\dots,\nel, \forall\,\strain\in\mathbb{S}^{2\times2},\label{eq:fmo_hashin_so}\\
    & \frac{1}{\nel}\sum_{i=1}^{\nel} \revisionA{\rho_i} \le \overline{V},\label{eq:fmo_hashin_volume}\\
    & 0 \le \revisionA{\rho_i} \le 1, \quad i = 1,\dots,\nel.
\end{align}
\end{subequations}

Define $t = \lvert \Tr(\strain) \rvert$ and $s = \sqrt{\Tr(\strain)^2 - 4 \Det(\strain)} = \sqrt{(\varepsilon_{11} - \varepsilon_{22})^2 + 4\varepsilon_{12}^2}$. In addition, let $\Delta \kappa = \kappa^+ - \kappa^- > 0$ and $\Delta \mu = \mu^+ - \mu^- > 0$ be the contrast of the bulk and shear moduli of the isotropic constituents. Then, based on the result of \citet[Proposition 2.3]{Allaire1993}, for anisotropic composites made of well-ordered isotropic constituents, $q$ has the explicit form:
\begin{equation}\label{eq:q}
q(\strain;\revisionA{\rho_i}) = 
\begin{cases}
    \begin{aligned}[t]
        &q_1 (\strain; \revisionA{\rho_i}):= \frac{\revisionA{\rho_i}(1-\revisionA{\rho_i}) \left(\Delta\kappa t - \Delta\mu s \right)^{2}}{\kappa^+ + \mu^+ -(\Delta \kappa + \Delta \mu)\revisionA{\rho_i}}\\
        &\qquad\text{if }\; \revisionA{\rho_i}t \Delta \kappa \le (\kappa^+ + \mu^+ - \Delta \kappa \revisionA{\rho_i}) s  \;\text{ and }\; \revisionA{\rho_i} s \Delta \mu \le (\kappa^+ + \mu^+ - \Delta \mu \revisionA{\rho_i}) t,
    \end{aligned}\\
    \begin{aligned}[t]
        &q_2 (\strain; \revisionA{\rho_i}):= (1-\revisionA{\rho_i})\left[\frac{\revisionA{\rho_i} t^2 \Delta \kappa^2}{\kappa^+ + \mu^+ - \Delta\kappa \revisionA{\rho_i}} - s^2\Delta\mu\right]
        \\
        &\qquad\text{if }\; \revisionA{\rho_i}t \Delta \kappa > (\kappa^+ + \mu^+ - \Delta \kappa \revisionA{\rho_i}) s,
    \end{aligned}\\
    \begin{aligned}[t]
        & q_3 (\strain; \revisionA{\rho_i}):= \left(1-{\revisionA{\rho_i}}\right)\left[\frac{{\revisionA{\rho_i}} s^2 {\Delta \mu}^2}{\kappa^+ + \mu^+ - \Delta \mu \revisionA{\rho_i}}-t^2 \Delta\kappa\right]\\ 
        &\qquad \text{if }\; \revisionA{\rho_i} s \Delta \mu > (\kappa^{+} + \mu^{+} - \Delta \mu \revisionA{\rho_i}) t.
    \end{aligned}
\end{cases}
\end{equation}
\revisionA{The three branches are distinguished by the relative magnitudes of the spherical and deviatoric strain components. The first branch corresponds to the intermediate (sliding) regime, the second to the spherical-dominated regime, and the third to the deviatoric-dominated regime. Regime-wise expressions for the mutually dual stress and strain potentials are discussed in \citep{Lewinski2004}.}
The correction term $q(\strain;\revisionA{\rho_i})$ in \eqref{eq:q} is given by three explicit formulas chosen via two inequalities. These cases form a consistent partition (there is no overlap where both strict inequalities hold), and the expressions coincide on the branch boundaries (in particular $q_1=q_2$ on $\Gamma_{12}$ and $q_1=q_3$ on $\Gamma_{13}$), so that $q$ (and hence $f^{\mathrm{HS}}$) is continuous across branch transitions; see Proposition~\ref{prop:hs_continuous} in the Appendix.

With $q(\strain;\revisionA{\rho_i})$ defined in \eqref{eq:q}, \eqref{eq:hashin_shtrikman} is invariant under scalar scaling of the strain tensor. Indeed, for any $\alpha\neq 0$ we have $\langle \Ei(\alpha\strain),\alpha\strain\rangle=\alpha^2\langle \Ei\strain,\strain\rangle$, $\lvert\Tr(\alpha\strain)\rvert=\lvert\alpha\rvert\,\lvert\Tr(\strain)\rvert$, and $\Det(\alpha\strain)=\alpha^2\Det(\strain)$. Since the Hashin--Shtrikman envelope $f^{\mathrm{HS}}(\strain;\revisionA{\rho})$ is positively $2$-homogeneous in $\strain$, both sides of \eqref{eq:hashin_shtrikman} scale quadratically with $\alpha$ and the inequality is equivalent for all $\alpha\neq 0$. Hence, it suffices to enforce \eqref{eq:hashin_shtrikman} on any fixed sphere in $\mathbb{S}^{2\times2}$. We choose the normalization $\lVert \strain \rVert_F=\sqrt{2}/2$, and introduce the invariants 
\begin{equation}\label{eq:ts_invariants}
t:=\lvert\Tr(\strain)\rvert,\qquad
s:=\sqrt{\Tr(\strain)^2-4\Det(\strain)}
     =\sqrt{(\varepsilon_{11}-\varepsilon_{22})^2+4\varepsilon_{12}^2}.
\end{equation}
For $\lVert \strain \rVert_F=\sqrt{2}/2$ these satisfy $t\in[0,1]$, $s\in[0,1]$, and the identity $t^2+s^2= 2\lVert\strain \rVert_F^2 = 1$. Consequently, we may eliminate $s=\sqrt{1-t^2}$ and parameterize \revisionA{$q$} solely by $t\in[0,1]$
\begin{equation}\label{eq:qt}
q(t;\revisionA{\rho_i}) =
\begin{cases}
    \begin{aligned}[t]
        &q_1(t;\revisionA{\rho_i}) := 
        \frac{\revisionA{\rho_i}(1-\revisionA{\rho_i}) \left(\Delta\kappa t - \Delta\mu \sqrt{1-t^2} \right)^{2}}{\kappa^+ + \mu^+ -(\Delta \kappa + \Delta \mu)\revisionA{\rho_i}}\\
        &\qquad\text{if }\; \revisionA{\rho_i}t \Delta \kappa \le (\kappa^+ + \mu^+ - \Delta \kappa \revisionA{\rho_i}) \sqrt{1-t^2}\;\text{ and }\; \revisionA{\rho_i} \Delta \mu \sqrt{1-t^2} \le (\kappa^+ + \mu^+ - \Delta \mu \revisionA{\rho_i}) t,
    \end{aligned}\\
    \begin{aligned}[t]
        &q_2(t;\revisionA{\rho_i}) :=  \left(1-{\revisionA{\rho_i}}\right)\left[\frac{\revisionA{\rho_i} {t}^2 \Delta\kappa^2}{\kappa^+  + \mu^+ - \Delta \kappa \revisionA{\rho_i}}-(1-t^2) \Delta \mu \right]\\
        &\qquad\text{if }\; \revisionA{\rho_i}t \Delta \kappa > (\kappa^+ + \mu^+ - \Delta \kappa \revisionA{\rho_i}) \sqrt{1-t^2},
    \end{aligned}\\
    \begin{aligned}[t]
        &q_3(t;\revisionA{\rho_i}) :=
        \left(1-{\revisionA{\rho_i}}\right)\left[\frac{{\revisionA{\rho_i}} {\Delta \mu}^2 (1-t^2)}{\kappa^+ + \mu^+ - \Delta \mu \revisionA{\rho_i}}-t^2 \Delta\kappa\right]\\
        &\qquad\text{if }\; \revisionA{\rho_i} \Delta \mu \sqrt{1-t^2} > (\kappa^{+} + \mu^{+} - \Delta \mu \revisionA{\rho_i}) t,
    \end{aligned}
\end{cases}
\end{equation}
which shows that the correction term $q$ depends only on $t=\lvert\Tr(\strain)\rvert$ and the volume fraction $\revisionA{\rho_i}$. In contrast, the full Hashin--Shtrikman upper bound \eqref{eq:hashin_shtrikman} still depends on the whole strain tensor $\strain$ through the \revisionA{left-hand side $\langle\Ei \strain,\strain\rangle$}. At this stage, the $t$-parameterization is used only for the correction term; the corresponding one-parameter reduction of the activating-volume map is obtained later for orthotropic tensors in Section~\ref{sec:fmo_ortho} (Proposition~\ref{prop:trace_only_reduction}).

Using this form, we now prove that $q(t;\revisionA{\rho_i}) \ge 0$ for all $\revisionA{\rho_i} \in [0,1]$ and $\strain$, which shows that the Hashin--Shtrikman upper bound \eqref{eq:hashin_shtrikman} is tighter than the Voigt upper bound \eqref{eq:voigt_bound} (see also Figure~\ref{fig:voigt_hashin} for a graphical illustration). \revisionA{The nonnegativity of $q$ and its equality condition restate the comparison with the arithmetic-mean bound given by \citet[p.~693]{Allaire1993}. We record this classical property in our notation and include a direct verification for completeness, since it is used in the subsequent analysis of the admissible sets.}

\begin{figure}[!t]
    \includegraphics[width=\linewidth]{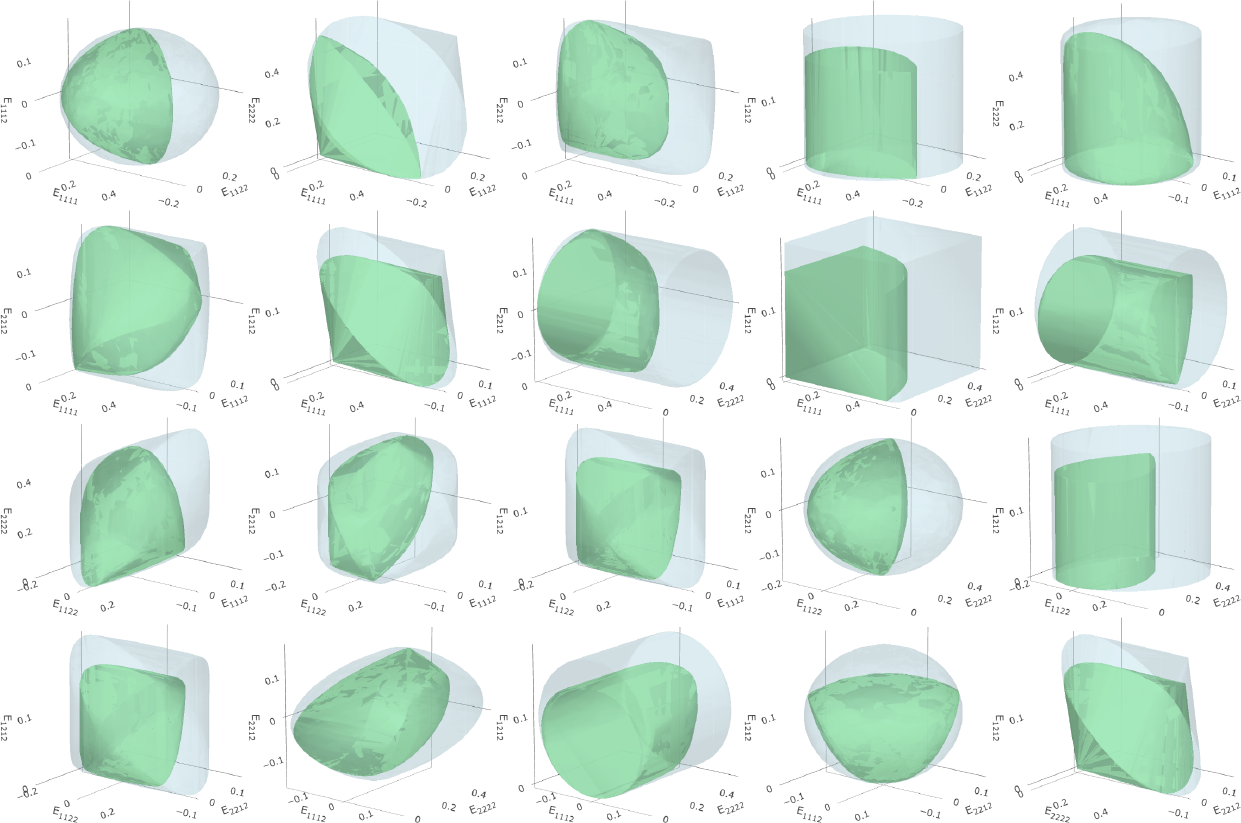}
    \caption{
        Comparison of the admissible sets $\mathcal A^{(2)}(0.5)$ (Hashin--Shtrikman, green) and $\mathcal A^{(1)}(0.5)$ (Voigt, blue). The sets are shown in a selected three-dimensional projection of $\Ei$ as outer-envelope boundary surfaces, obtained from $750$ sampled strains $\strain$ with Frobenius norm $\sqrt{2}/2$. Material parameters: $\kappa^- = 0.714\times10^{-9}$, $\kappa^+ = 0.714$, $\mu^- = 0.385\times10^{-9}$, and $\mu^+ = 0.385$.
    }
    \label{fig:voigt_hashin}
\end{figure}

\begin{prpstn}[Nonnegativity and equality cases of the HS correction term]\label{prop:q_ge_0}
    For the Hashin--Shtrikman upper bounds \eqref{eq:hashin_shtrikman} with $q(t;\revisionA{\rho_i})$ defined in \eqref{eq:qt}, it holds for all $\revisionA{\rho_i} \in [0,1]$ and $\strain \in \mathbb{S}^{2\times2}$ with the Frobenius norm $\sqrt{2}/2$ and $t = \lvert \Tr(\strain)\rvert$ that $q(t;\revisionA{\rho_i}) \ge 0$. Moreover, $q(t;\revisionA{\rho_i}) = 0$ if and only if $\revisionA{\rho_i} \in \{0,1\}$ or $\Delta \kappa t = \Delta \mu \sqrt{1-t^2}$.
\end{prpstn}
\begin{proof}
    We prove the statement by first analyzing the boundary values $\revisionA{\rho_i} \in \{0,1\}$ and then showing positivity or nonnegativity of each branch in \eqref{eq:qt} for $\revisionA{\rho_i} \in (0,1)$.

    Consider now $\revisionA{\rho_i} \in \{0,1\}$. For $\revisionA{\rho_i} = 0$, the conditions of the first branch are satisfied for all $t\in [0,1]$ and hence for all strains $\strain$, because the left-hand side of the inequalities are zero and the right-hand sides nonnegative. Using arbitrary $t$, $q(t;0) = q_1(t;0) = 0$ and, therefore, the Voigt upper bound is tight for $\revisionA{\rho_i} = 0$.  For $\revisionA{\rho_i} = 1$, all three branches in \eqref{eq:qt} evaluate as zero due to the $(1-\revisionA{\rho_i})$ prefactor. Hence, $q(t;1) = 0$, so the Voigt upper bound is also tight in this case.

    Let us further proceed with the branches of $q(t;\revisionA{\rho_i})$ for $\revisionA{\rho_i} \in (0,1)$. The first branch in \eqref{eq:qt} is nonnegative due to $\revisionA{\rho_i} \in (0,1)$, $0\le \mu^- <\mu^+$, $0 \le \kappa^- < \kappa^+$, hence $\revisionA{\rho_i} \Delta \kappa < \kappa^+$ and $\revisionA{\rho_i} \Delta \mu < \mu^+$. Moreover, $q_1(t;\revisionA{\rho_i})=0$ if and only if $\Delta\kappa t=\Delta\mu \sqrt{1-t^2}$, i.e., $t=\Delta \mu /\sqrt{\Delta\mu^2 + \Delta\kappa^2}$. Under this equality, the two branch conditions reduce to $\kappa^+ + \mu^+ - \revisionA{\rho_i} (\Delta\kappa + \Delta \mu) \ge 0$, which holds for all $\revisionA{\rho_i} \in (0,1)$, so the related strains are indeed in this branch. For any other strains $\{\strain \in \mathbb{S}^{2\times2}: \revisionA{\|\strain\|_F=\sqrt{2}/2,\ }\left|\Tr(\strain)\right|\neq\Delta \mu /\sqrt{\Delta\mu^2 + \Delta\kappa^2}\}$ satisfying the branch conditions (for example, set $t=1/\sqrt{1 + (\revisionA{\rho_i}\Delta\kappa/d_2)^2}$), $q_1(t;\revisionA{\rho_i})>0$ for all $\revisionA{\rho_i} \in (0,1)$.
    
    For the second branch in \eqref{eq:qt}, the denominator $d_2 := \kappa^+ + \mu^+ - \Delta \kappa \revisionA{\rho_i}$ is positive due to ordered materials, so we only need to show positivity of
    \begin{equation}\label{eq:q_case2}
        (1-\revisionA{\rho_i})\left[\revisionA{\rho_i} t^2 \Delta \kappa^2 - \Delta\mu d_2 (1-t^2)\right] \quad \text{for}\quad  \revisionA{\rho_i} t \Delta\kappa > d_2 \sqrt{1-t^2}.
    \end{equation}
    Since both sides are nonnegative, squaring preserves inequality and yields $\revisionA{\rho_i}^2 t^2\Delta\kappa^2 > d_2^2 (1-t^2)$, hence $d_2 (1-t^2) < \revisionA{\rho_i}^2 t^2\Delta\kappa^2/d_2$. Substituting this upper bound into the negative term in \eqref{eq:q_case2} gives the lower bound
    \begin{equation}\label{eq:q_case2_final}
        (1-\revisionA{\rho_i}) \revisionA{\rho_i} {t}^2\Delta \kappa^2 \frac{\kappa^+ + \mu^+ - (\Delta\kappa + \Delta\mu)\revisionA{\rho_i}}{d_2} > 0,
    \end{equation}
    which is positive due to $\revisionA{\rho_i} \in (0,1)$ and because $t^2 = 0$ violates the branch condition.

    For the third case in \eqref{eq:qt}, the denominator $d_3 = \kappa^{+} + \mu^{+} - \Delta \mu \revisionA{\rho_i}$ is again positive, so we only need to show positivity of
    \begin{equation}\label{eq:q_case3}
        (1-\revisionA{\rho_i}) \left[ \revisionA{\rho_i} \Delta \mu^2 (1-t^2) - \Delta\kappa d_3 t^2 \right] \quad \text{for} \quad \revisionA{\rho_i} \Delta \mu \sqrt{1-t^2} > d_3 t.
    \end{equation}
    Both sides of the right inequality are nonnegative, so it is after squaring equivalent to $\revisionA{\rho_i}^2 (1-t^2) \Delta\mu^2/d_3 > d_3 t^2$. Inserting this upper bound into the negative term in \eqref{eq:q_case3}, we obtain
    \begin{equation}\label{eq:q_case3_final}
        (1-\revisionA{\rho_i}){\revisionA{\rho_i}}(1-t^2) \Delta\mu^2 \frac{\kappa^+ + \mu^+ - (\Delta\kappa + \Delta\mu)\revisionA{\rho_i}}{d_3} > 0,
    \end{equation}
    which is positive for $\revisionA{\rho_i} \in (0,1)$ as $t^2 = 1$ violates the branch condition.
\end{proof}

\revisionA{We next record two elementary geometric consequences for the joint elastic-tensor--volume-fraction admissible set: nonconvexity follows from the strict improvement over the Voigt set, while the convex-hull identity additionally uses the endpoint structure of the imposed constraints.}

\subsubsection{Geometric Properties of the Hashin--Shtrikman Feasible Set}

For any fixed $\revisionA{\rho_i}$ and strain sample $\strain$, the Hashin--Shtrikman constraint \eqref{eq:hashin_shtrikman} is affine in the entries of $\Ei$, hence the feasible set in $\Ei$ is convex (see Figure~\ref{fig:voigt_hashin}); with an outer approximation by a finite set of sampled strains it becomes a polyhedron. However, \eqref{eq:hashin_shtrikman} is nonconvex in $(\revisionA{\rho_i}, \Ei)$, which we show in the following lemma.

\begin{lmm}[Nonconvexity of the HS-feasible set]\label{lem:nonconvex}
    Let $0< \kappa^- < \kappa^+$ and $0<\mu^- < \mu^+$. Then, the Hashin--Shtrikman upper bound \eqref{eq:hashin_shtrikman} with $q(t;\revisionA{\rho_i})$ defined in \eqref{eq:qt} and $t=\abstrstrain$ is nonconvex in $(\revisionA{\rho_i}, \Ei)$.
\end{lmm}
\begin{proof}
    Consider the feasible pairs $(0,\Emin)$ and $(1,\Emax)$. By Proposition \ref{prop:q_ge_0}, for any \revisionA{normalized} strain $\strain$ (and thus $t$) we have $q(t;0)=q(t;1)=0$, so \eqref{eq:hashin_shtrikman} holds with equality at both endpoints. Fix any $\theta\in(0,1)$ and set $\revisionA{\widehat{\rho}}:=\theta$ and $\widehat{\tensf{E}}:=(1-\theta)\Emin+\theta\,\Emax$. Evaluating \eqref{eq:hashin_shtrikman} at $(\revisionA{\widehat{\rho}},\widehat{\tensf{E}})$ gives, for any strain $\strain$,
    \begin{equation}
    \big\langle \widehat{\tensf{E}}\,\strain,\strain\big\rangle
    \;\le\;
    \big\langle \big[(1-\revisionA{\widehat{\rho}})\Emin+\revisionA{\widehat{\rho}}\,\Emax\big]\strain,\strain\big\rangle
    - q(\abstrstrain;\revisionA{\widehat{\rho}}).
    \end{equation}
    Since $\widehat{\tensf{E}}=(1-\revisionA{\widehat{\rho}})\Emin+\revisionA{\widehat{\rho}}\,\Emax$, the two quadratic terms are identical and the inequality reduces to $0 \;\le\; -\,q(\strain;\revisionA{\widehat{\rho}})$. However, for any interior $\revisionA{\widehat{\rho}}\in(0,1)$ the strains $\hat{\strain} \in \{\strain \in \mathbb{S}^{2\times2}: \revisionA{\|\strain\|_F=\sqrt{2}/2,\ }\left|\Tr(\strain)\right|\neq\Delta \mu /\sqrt{\Delta\mu^2 + \Delta\kappa^2}\}$ result in $q(\lvert\Tr(\hat{\strain})\rvert;\revisionA{\widehat{\rho}})>0$. Hence, the inequality fails for such a strain, and $(\revisionA{\widehat{\rho}},\widehat{\tensf{E}})$ is infeasible. Therefore, although $(0,\Emin)$ and $(1,\Emax)$ are feasible, every interior convex combination $\big(\theta, (1-\theta) \Emin + \theta\Emax\big)$ with $\theta\in(0,1)$ is infeasible. The feasible set in $(\revisionA{\rho_i},\Ei)$ is thus nonconvex.
\end{proof}

The nonconvexity of the Hashin--Shtrikman bounds in $(\revisionA{\rho_i}, \Ei)$ is illustrated in Figure~\ref{fig:voigt_hashin_vol}, which shows the feasible sets for varying $\revisionA{\rho_i}$. The figure also suggests that, when combined with the zeroth-order lower bound $\Emin \preceq \Ei$, the Voigt relation $\Ei \preceq \EV$ describes the closed convex hull of the Hashin--Shtrikman feasible set in $(\revisionA{\rho_i},\Ei)$, which we prove next.

\begin{lmm}[Convex hull of the HS-feasible set]\label{lem:HS_convex_hull}
Define the feasible sets
$\mathcal F_{\mathrm{HS}}
:=\{(\revisionA{\rho},\tensf E)\in[0,1]\times\mathbb{S}^{2\times2\times2\times2}:\ \tensf E\in\mathcal A^{(2)}(\revisionA{\rho})\},
$
and
$
\mathcal F_{\mathrm V}
:=\{(\revisionA{\rho},\tensf E)\in[0,1]\times\mathbb{S}^{2\times2\times2\times2}:\ \tensf E\in\mathcal A^{(1)}(\revisionA{\rho})\},
$
Then, taking the convex hull in the product space $[0,1]\times\mathbb{S}^{2\times2\times2\times2}$ (with coordinates $(\revisionA{\rho},\tensf E)$), we have
$
\Conv(\mathcal F_{\mathrm{HS}})=\mathcal F_{\mathrm V}.
$
\end{lmm}

\begin{proof}
By definition of $\mathcal A^{(2)}(\revisionA{\rho})$ we have $\Emin\preceq \tensf E$ for all $(\revisionA{\rho},\tensf E)\in\mathcal F_{\mathrm{HS}}$. Moreover, since $q\ge 0$ (cf.\ Proposition~\ref{prop:q_ge_0}), we obtain
\begin{equation}\label{eq:fHS_le_V}
f^{\mathrm{HS}}(\strain;\revisionA{\rho})=\langle \tensf E^{\mathrm V}(\revisionA{\rho})\strain,\strain\rangle-q(\strain;\revisionA{\rho})
\le \langle \tensf E^{\mathrm V}(\revisionA{\rho})\strain,\strain\rangle
\qquad \forall \strain,\ \forall \revisionA{\rho}\in[0,1].
\end{equation}
Hence, every $(\revisionA{\rho},\tensf E)\in\mathcal F_{\mathrm{HS}}$ satisfies $\tensf E\preceq \tensf E^{\mathrm V}(\revisionA{\rho})$, i.e., $\mathcal F_{\mathrm{HS}}\subseteq \mathcal F_{\mathrm V}$. Since $\mathcal F_{\mathrm V}$ is convex, this implies that $\Conv(\mathcal F_{\mathrm{HS}})\subseteq \mathcal F_{\mathrm V}$.

For the reverse inclusion, let $(\revisionA{\rho},\tensf E)\in\mathcal F_{\mathrm V}$ be arbitrary. If $\revisionA{\rho}=0$, then $\tensf E=\Emin$ and $(0,\Emin)\in\mathcal F_{\mathrm{HS}}$ (cf.\ the endpoint identity $f^{\mathrm{HS}}(\cdot;0)=\langle\Emin\cdot,\cdot\rangle$, e.g.\ Proposition~\ref{prop:q_ge_0}). Assume $\revisionA{\rho}\in(0,1]$ and define
\begin{equation}\label{eq:E1_def}
\tensf E_1:=\Emin+\frac{1}{\revisionA{\rho}}\bigl(\tensf E-\Emin\bigr).
\end{equation}
From $\tensf E\preceq \tensf E^{\mathrm V}(\revisionA{\rho})=\Emin+\revisionA{\rho}(\Emax-\Emin)$ we infer $\tensf E_1\preceq \Emax$ and clearly $\Emin\preceq \tensf E_1$. Using the endpoint identity $f^{\mathrm{HS}}(\cdot;1)=\langle\Emax\cdot,\cdot\rangle$ (Proposition~\ref{prop:q_ge_0}), it follows that $(1,\tensf E_1)\in\mathcal F_{\mathrm{HS}}$. Consequently,
\begin{equation}\label{eq:conv_representation}
(\revisionA{\rho},\tensf E)=(1-\revisionA{\rho})\,(0,\Emin)+\revisionA{\rho}\,(1,\tensf E_1)\in \Conv(\mathcal F_{\mathrm{HS}}),
\end{equation}
which shows $\mathcal F_{\mathrm V}\subseteq \Conv(\mathcal F_{\mathrm{HS}})$. Together with $\mathcal F_{\mathrm V}\supseteq \Conv(\mathcal F_{\mathrm{HS}})$ this yields $\Conv(\mathcal F_{\mathrm{HS}})=\mathcal F_{\mathrm V}$.
\end{proof}

\begin{figure}[!t]
    \includegraphics[width=\linewidth]{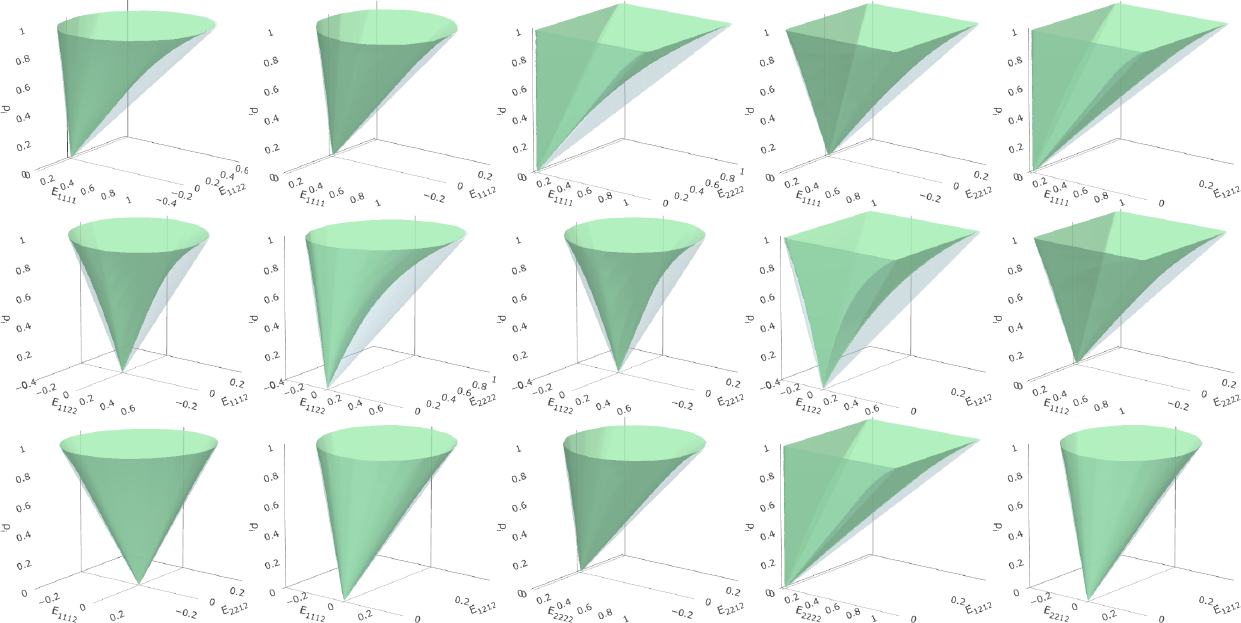}
    \caption{
        Comparison of the product-space feasible sets $\mathcal F_{\mathrm{HS}}$ (green) and $\mathcal F_{\mathrm V}$ (blue), defined in Lemma~\ref{lem:HS_convex_hull}. The figure shows projections in $(\revisionA{\rho_i},\Ei)$ for all pairs of selected components of $\Ei$ together with $\revisionA{\rho_i}$. Each plotted set is shown by its boundary surface, obtained by sampling $\revisionA{\rho_i}$ and taking the outer hull of feasible points, using $250$ sampled strains $\strain$ with Frobenius norm $\sqrt{2}/2$. Material parameters: $\kappa^- = 0.714\times10^{-9}$, $\kappa^+ = 0.714$, $\mu^- = 0.385\times10^{-9}$, and $\mu^+ = 0.385$.
    }
    \label{fig:voigt_hashin_vol}
\end{figure}

\subsubsection{Eliminating the Volume Fractions}

As in the Voigt case, the Hashin--Shtrikman constraint \revisionA{can be chosen} active at optimality, which allows us to reformulate \eqref{eq:fmo_hashin} by eliminating $\revisionA{\rho_i}$. By Propositions~\ref{prop:hs_continuous} and \ref{prop:hs_monotone} in the Appendix, for every fixed $\strain$ the map $\revisionA{\rho}\mapsto f^{\mathrm{HS}}(\strain;\revisionA{\rho})$ is continuous and nondecreasing. Hence, for fixed $(\Ei,\strain)$ the feasibility set in $\revisionA{\rho_i}$ is an interval,
\begin{equation}\label{eq:vi_interval}
\{\revisionA{\rho_i}\in[0,1]:\ \langle \Ei\strain,\strain\rangle \le f^{\mathrm{HS}}(\strain;\revisionA{\rho_i})\}=[\revisionA{\hat{\rho}_i}(\strain;\Ei),1],
\end{equation}
where $\revisionA{\hat{\rho}_i}(\strain;\Ei)\in[0,1]$ denotes the minimal admissible volume fraction (equivalently, the smallest $\revisionA{\rho}$ such that the inequality holds; when $\langle \Ei\strain,\strain\rangle$ lies on the graph this means $\langle \Ei\strain,\strain\rangle=f^{\mathrm{HS}}(\strain;\revisionA{\hat{\rho}_i}(\strain;\Ei))$). Imposing the Hashin--Shtrikman inequality for all normalized strains is, therefore, equivalent to
\begin{equation}\label{eq:vi_plus_sup}
\revisionA{\rho_i} \ge \revisionA{\overline{\rho}_i}(\Ei):=\sup_{\strain\in\mathbb{S}^{2\times2}:\ \lVert \strain \rVert_F=\sqrt{2}/2}\revisionA{\hat{\rho}_i}(\strain;\Ei).
\end{equation}
\begin{lmm}[Elementwise activity of the HS constraint at optimality]\label{lem:HS_active}
There exists an optimal solution of \eqref{eq:fmo_hashin} such that, for every element $i$, $\revisionA{\rho_i^\star} = \revisionA{\overline{\rho}_i}(\revisionA{\Ei^\star})=\inf\{\revisionA{\rho}\in[0,1]:\ \revisionA{\Ei^\star}\in\mathcal A^{(2)}(\revisionA{\rho})\}$.
\end{lmm}

\begin{proof}
Let $\revisionA{(\sevek{c}^\star,\boldsymbol{\mathcal E}^\star,\sevek{\rho}^\star)}$ be any optimal solution. By monotonicity (Proposition \ref{prop:hs_monotone}) and continuity (Proposition \ref{prop:hs_continuous}), for each $i$ we can decrease $\revisionA{\rho_i^\star}$ down to
\begin{equation}
\revisionA{\widetilde \rho_i} := \inf\{\revisionA{\rho_i} \in [0,1]:\ \langle \Ei^\star \strain, \strain\rangle \le f^{\mathrm{HS}}(\strain;\revisionA{\rho_i}), \forall \strain\}
= \revisionA{\overline{\rho}_i}(\Ei^\star),
\end{equation}
without violating \eqref{eq:fmo_hashin_so}. This operation leaves $\Ei^\star$ unchanged, so all matrices $\semtrx{K}_j (\revisionA{\boldsymbol{\mathcal E}^\star})$ and the objective $\sum_{j=1}^{\nlc} c_j$ remain the same. It only reduces the left-hand side of the volume constraint \eqref{eq:fmo_hashin_volume}, so feasibility is preserved. \revisionA{Writing $\widetilde{\sevek{\rho}}:=(\widetilde{\rho}_1,\ldots,\widetilde{\rho}_{\nel})^\trn$, the solution $(\sevek{c}^\star,\boldsymbol{\mathcal E}^\star,\widetilde{\sevek{\rho}})$ is therefore also optimal,} with each $\revisionA{\widetilde \rho_i}$ equal to the minimal admissible value for fixed $\Ei^\star$.
\end{proof}

In particular, activity here is with respect to the design variable $\revisionA{\rho_i}$ for fixed $\Ei$: $\revisionA{\rho_i}$ is reduced to its minimal admissible value, not interpreted as equality in \eqref{eq:hashin_shtrikman} for every strain.
It remains to compute $\revisionA{\hat{\rho}}(\strain;\Ei)$ explicitly. The following result provides closed-form expressions by \revisionA{inverting the classical bound: we solve} the equality $\langle\Ei\strain,\strain\rangle = f^{\mathrm{HS}}(\strain;\revisionA{\rho_i})$ for each branch of~\eqref{eq:qt}.
\begin{prpstn}[Closed-form minimal admissible volume for HS activity]\label{prop:explicit_volume}
    For any fixed $\Ei \in \{\tensf{E}: \Emin \preceq \tensf{E} \preceq \Emax\}$ and strain $\strain \in \mathbb{S}^{2\times2}$ with the Frobenius norm $\sqrt{2}/2$, the minimal $\revisionA{\rho_i} \in [0,1]$ for which the Hashin--Shtrikman upper bound \eqref{eq:hashin_shtrikman} with $q(t;\revisionA{\rho_i})$ defined in \eqref{eq:qt} is active, $\langle \Ei \strain, \strain\rangle = f^{\mathrm{HS}}(\strain;\revisionA{\rho_i})$, is given by
    \begin{equation}
        \revisionA{\hat{\rho}}(\strain;\Ei) =
    \begin{cases}
        \begin{aligned}[t]
            &\revisionA{\hat{\rho}_1} (\strain;\Ei) :=
                \frac{1}{2} + \frac{\langle\Ei\strain, \strain\rangle(\Delta\mu + \Delta\kappa)}{2 \Delta\mu \Delta\kappa (t + \sqrt{1 - t^2})^2}
                + \frac{(2t^2-1)(\Delta\kappa \mu^+ -\Delta\mu \kappa^+) - \sqrt{D}}{2 \Delta\mu \Delta\kappa (t + \sqrt{1 - t^2})^2}\\
            &\qquad \text{if } \revisionA{\hat{\rho}_1} \le \frac{\kappa^++\mu^+}{t + \sqrt{1-t^2}} \min\left\{ \frac{\sqrt{1-t^2}}{\Delta\kappa}, \frac{t}{\Delta \mu} \right\}
        \end{aligned}\\
        \begin{aligned}[t]
            &\revisionA{\hat{\rho}_2} (\strain;\Ei) := \frac{\kappa^+ + \mu^+}{\Delta\kappa}\left[1 - \frac{t^2 (\kappa^- + \mu^+)}{ \langle\Ei\strain, \strain\rangle - \mu^+ (1 - 2t^2)}\right] \\
            &\qquad \text{if }\; \revisionA{1\ge\hat{\rho}_2} > \frac{(\kappa^+ + \mu^+)\sqrt{1-t^2}}{\Delta\kappa(t + \sqrt{1-t^2})}
        \end{aligned}\\
        \begin{aligned}[t]
            & \revisionA{\hat{\rho}_3} (\strain;\Ei) := \frac{\kappa^+ + \mu^+}{\Delta\mu} \left[1 - \frac{(1-t^2) (\kappa^+ + \mu^-)}{\langle\Ei\strain, \strain\rangle + \kappa^+ (1 - 2t^2)}\right] \\
            & \qquad\text{if }\; \revisionA{1\ge\hat{\rho}_3} > \frac{(\kappa^+ + \mu^+)t}{\Delta\mu(t + \sqrt{1-t^2})}
        \end{aligned}
    \end{cases}
    \end{equation}
    with
    \begin{equation}\label{eq:det}
        \begin{multlined}
        D = \left[\langle\Ei\strain, \strain\rangle (\Delta\mu + \Delta\kappa) + (1 - 2t^2) (\Delta\mu \kappa^+ - \Delta\kappa\mu^+) + \Delta\mu \Delta\kappa (\sqrt{1-t^2} + t)^2\right]^2\\
         - 4 \Delta\mu \Delta\kappa (\kappa^+ + \mu^+) (\sqrt{1-t^2} + t)^2 (\langle\Ei\strain, \strain\rangle + (t^2-1)\mu^- - t^2\kappa^-).
        \end{multlined}
    \end{equation}
\end{prpstn}
\begin{proof}
Define, for $t \in [0,1]$,
\begin{equation}
    \revisionA{\rho_{12}} (t) = \frac{\kappa^+ + \mu^+}{\Delta\kappa} \frac{\sqrt{1-t^2}}{t + \sqrt{1-t^2}}, \quad \revisionA{\rho_{13}} (t) = \frac{\kappa^+ + \mu^+}{\Delta\mu} \frac{t}{t + \sqrt{1-t^2}}.
\end{equation}
Then, for $\revisionA{\rho_i}\in[0,1]$, the branch conditions in \eqref{eq:qt} are equivalently $\revisionA{\rho_i} \le \min\{\revisionA{\rho_{12}}(t), \revisionA{\rho_{13}}(t)\}$ in the first branch, $\revisionA{\rho_{12}}(t) < \revisionA{\rho_i} \revisionA{\le 1}$ in the second branch, and $\revisionA{\rho_{13}}(t) < \revisionA{\rho_i} \revisionA{\le 1}$ in the third branch. \revisionA{We solve the equality on each branch and retain
only roots satisfying the corresponding branch conditions. By Corollary~\ref{cor:uniqueness}, the admissible root is the unique activating volume fraction.}

For the first branch in \eqref{eq:qt}, after unification to a common denominator and grouping powers of $\revisionA{\rho_i}$, the Hashin--Shtrikman equality reads as
\begin{multline}
    0 = \frac{\revisionA{\rho_i}^2 (-\Delta\mu \Delta\kappa (\sqrt{1-t^2} + t)^2) - (\kappa^+ + \mu^+)(\langle\Ei\strain, \strain\rangle + \Delta\mu (1-t^2) + \Delta\kappa t^2 - \kappa^+ t^2 - \mu^+(1-t^2))}{\kappa^+ + \mu^+ - (\Delta\kappa + \Delta\mu) \revisionA{\rho_i}}\\
    + \frac{\revisionA{\rho_i}\left[\langle\Ei\strain, \strain\rangle (\Delta\mu + \Delta\kappa) + (1 - 2t^2) (\Delta\mu \kappa^+ - \Delta\kappa\mu^+)+ \Delta\mu \Delta\kappa (\sqrt{1-t^2} + t)^2\right]}{\kappa^+ + \mu^+ - (\Delta\kappa + \Delta\mu) \revisionA{\rho_i}}.
\end{multline}
Thus, we search for the roots of the quadratic polynomial in the numerator, which can be expressed as
\begin{equation}\label{eq:quad}
    \frac{1}{2} + \frac{\langle\Ei\strain, \strain\rangle(\Delta\mu + \Delta\kappa) + (2t^2-1)(\Delta\kappa \mu^+ -\Delta\mu \kappa^+) \pm \sqrt{D}}{2 \Delta\mu \Delta\kappa (t + \sqrt{1 - t^2})^2}
\end{equation}
with $D$ in \eqref{eq:det}. Due to uniqueness of the activating volume fraction (Corollary \ref{cor:uniqueness}), $D\ge 0$ but only one root can be admissible in $[0,1]$, so it remains to show which one. To this end, set $\Ei=\Emin$. Then, \eqref{eq:quad} simplifies to
\begin{multline}
    \frac{1}{2} + 
    \frac{(t^2\Delta\kappa + (1-t^2)\Delta\mu)(\kappa^- + \mu^-) \pm [(t^2\Delta\kappa + (1 - t^2)\Delta\mu) (\kappa^+ + \mu^+) - ( \Delta\kappa t - \Delta\mu \sqrt{1-t^2})^2]}{2 \Delta\mu \Delta\kappa (t + \sqrt{1 - t^2})^2} 
    = \\
    = \begin{cases}
        \begin{aligned}0\end{aligned}\\
        \begin{aligned}1 + \frac{(\Delta\mu (1-t^2) + \Delta\kappa t^2)(\kappa^- + \mu^-)}{\Delta\mu \Delta\kappa (t + \sqrt{1-t^2})^2 }.\end{aligned}
    \end{cases}
\end{multline}
Clearly, the ``negative'' root gives admissible $\revisionA{\rho_i} = 0$, whereas the ``positive'' root gives $\revisionA{\rho_i} > 1$. By strict monotonicity in $\revisionA{\rho_i}$ (Proposition \ref{prop:hs_monotone}), the admissible root for any $\Ei \in [\Emin, \Emax]$ is, therefore, the ``negative'' root, i.e.,
\begin{equation}
    \revisionA{\hat{\rho}_1} (\strain;\Ei) =
        \frac{1}{2} + \frac{\langle\Ei\strain, \strain\rangle(\Delta\mu + \Delta\kappa) + (2t^2-1)(\Delta\kappa \mu^+ -\Delta\mu \kappa^+) - \sqrt{D}}{2 \Delta\mu \Delta\kappa (t + \sqrt{1 - t^2})^2}.
\end{equation}

Consider now the second branch in \eqref{eq:qt}. Rearranging the equality $\langle \Ei \strain, \strain\rangle = f^{\mathrm{HS}}(\strain;\revisionA{\rho_i})$ gives
\begin{equation}
    0 = \frac{\Delta\kappa (\mu^+ (2t^2 -1) + \langle\Ei\strain, \strain\rangle)\revisionA{\rho_i} +(\kappa^+ + \mu^+) (\kappa^- t^2 + \mu^+ (1-t^2) - \langle\Ei\strain, \strain\rangle)}{\kappa^+ + \mu^+ - \Delta\kappa \revisionA{\rho_i}}.
\end{equation}
Equating the numerator gives the solution
\begin{equation}
    \revisionA{\hat{\rho}_2} (\strain;\Ei) = \frac{\kappa^+ + \mu^+}{\Delta\kappa}\left[1 - \frac{t^2 (\kappa^- + \mu^+)}{ \langle\Ei\strain, \strain\rangle - \mu^+ (1 - 2t^2)}\right].
\end{equation}

For the third branch in \eqref{eq:qt}, rearranging the equality $\langle \Ei \strain, \strain\rangle = f^{\mathrm{HS}}(\strain;\revisionA{\rho_i})$ gives
\begin{equation}
    0 = \frac{\Delta\mu (\kappa^+ (1-2t^2) + \langle\Ei\strain, \strain\rangle)\revisionA{\rho_i} + (\kappa^+ + \mu^+) (\kappa^+ t^2 + \mu^- (1-t^2) - \langle\Ei\strain, \strain\rangle) }{\kappa^+ + \mu^+ - \Delta\mu \revisionA{\rho_i}},
\end{equation}
with the numerator providing
\begin{equation}
    \revisionA{\hat{\rho}_3} (\strain;\Ei) = \frac{\kappa^+ + \mu^+}{\Delta\mu} \left[1 - \frac{(1-t^2) (\kappa^+ + \mu^-)}{\langle\Ei\strain, \strain\rangle + \kappa^+ (1 - 2t^2)}\right]
\end{equation}
completing the proof.
\end{proof}
By Lemma~\ref{lem:HS_active}, at optimality one may always take $\revisionA{\rho_i}=\revisionA{\overline{\rho}_i}(\Ei)$, with $\revisionA{\overline{\rho}_i}(\Ei)$ defined in~\eqref{eq:vi_plus_sup} and $\revisionA{\hat{\rho}}(\strain;\Ei)$ given in Proposition~\ref{prop:explicit_volume}. The free material optimization problem under Hashin--Shtrikman upper bounds \eqref{eq:fmo_hashin} can, therefore, be written as
\begin{subequations}\label{eq:fmo_hashin_reduced}
\begin{align}
    \min_{\revisionA{\sevek{c}}, \revisionA{\boldsymbol{\mathcal E}}} \; & \sum_{j=1}^{\nlc} c_j \label{eq:fmo_hashin_reduced_compliance}\\
    \mathrm{s.t.}\; & \begin{pmatrix}
        c_j & -\sevek{f}_j^\trn\\
        -\sevek{f}_j & \semtrx{K}_j\left(\revisionA{\boldsymbol{\mathcal E}}\right)
    \end{pmatrix}\succeq \semtrx{0}, \quad j = 1,\dots,\nlc,\label{eq:fmo_hashin_reduced_lmi}\\
    & \Emin \preceq \Ei \preceq \Emax,\quad i = 1,\dots,\nel,\label{eq:fmo_hashin_reduced_bounds}\\
    & \frac{1}{\nel}\sum_{i=1}^{\nel} \sup_{\strain\in\mathbb{S}^{2\times2}:\ \lVert\strain\rVert_F = \sqrt{2}/2} \revisionA{\hat{\rho}}(\strain;\Ei) \le \overline{V}\label{eq:fmo_hashin_reduced_volume}
\end{align}
\end{subequations}
in which \eqref{eq:fmo_hashin_reduced_bounds} guarantees that the volume given in Proposition \ref{prop:explicit_volume} is in the $[0,1]$ domain.

\subsubsection{Single-Loadcase Equivalence in the Continuum Setting}

The nonconvexity of the Hashin--Shtrikman upper bounds in $(\revisionA{\rho_i}, \Ei)$ (Lemma \ref{lem:nonconvex}) carries over to the nonconvexity of \eqref{eq:fmo_hashin} and \eqref{eq:fmo_hashin_reduced}. \revisionA{Moreover, the reduced formulation is, in general, nonsmooth because the supremum over strain directions in \eqref{eq:fmo_hashin_reduced_volume} need not be differentiable
with respect to the elasticity tensor. For a single loadcase ($\nlc=1$), however, the tensor-constrained HS-FMO formulation can be related to the classical relaxation by homogenization \citep{Allaire1993od,Allaire2002}. The following theorem establishes equality of their continuum optimal values, using the classical energy bounds and sequential-laminate recovery constructions.}

\begin{thrm}[Tightness of the HS/AK relaxation in the single-load\-case continuum setting]\label{th:tight}
Assume the standing continuum hypotheses of Section~\ref{sec:cont_fmo} (in particular, the regularity of $\Omega$, boundary conditions ensuring coercivity on $\mathcal U_0$, and a bounded linear load functional $\ell:=\ell_1$ from \eqref{eq:load_functional}), two-dimensional plane stress with well-ordered isotropic phases $\tensf{0}\prec\Emin\prec\Emax$, and a single-loadcase.
Define
\begin{equation}
\mathcal V_{\overline{V}}
:=
\left\{
\revisionA{\rho}\in L^\infty(\Omega;[0,1])\;:\;
\int_\Omega \revisionA{\rho}(\tens{x})\,\mathrm d\tens{x}\le \overline{V}\,\lvert\Omega\rvert
\right\}.
\end{equation}
Let $c_{\mathrm{HS\text{-}FMO}}^{\mathrm{cont}}$ be the infimum value of the continuum FMO problem \eqref{eq:fmo_continuum} from Section~\ref{sec:cont_fmo} (with $\nlc=1$), taken over all pairs $(\revisionA{\rho},\tensf E)\in \mathcal V_{\overline{V}}\times L^\infty(\Omega;\mathbb{S}^{2\times2\times2\times2})$ satisfying the pointwise admissibility $\tensf E(\tens{x})\in \mathcal A^{(2)}(\revisionA{\rho}(\tens{x}))$ for a.e. $\tens{x}\in\Omega$, where $\mathcal A^{(2)}$ is given by \eqref{eq:A2_def}. We refer to this model as the continuum Hashin--Shtrikman-constrained free material optimization problem (HS-FMO). Define the Allaire--Kohn primal relaxed value by
\begin{equation}\label{eq:ak_primal_cont}
c_{\mathrm{AK,pr}}^{\mathrm{cont}}
:=
\inf_{\revisionA{\rho}\in\mathcal V_{\overline{V}}}
\sup_{\tens{u}\in\mathcal U_0}
\left\{
2\,\ell(\tens{u})
-\int_\Omega f^{\mathrm{HS}}(\strain(\tens{u});\revisionA{\rho})\,\mathrm d\tens{x}
\right\}.
\end{equation}
Further, define the complementary-energy envelope
\begin{equation}\label{eq:gak_def}
f_{\mathrm{c}}^{\mathrm{HS}}(\stress;\revisionA{\rho})
:=
\sup_{\strain\in\mathbb{S}^{2\times2}}
\left\{
2\langle\stress,\strain\rangle-f^{\mathrm{HS}}(\strain;\revisionA{\rho})
\right\},
\end{equation}
the statically admissible stress set
\begin{equation}
\Sigma_{\mathrm{ad}}
:=
\left\{
\stress\in L^2(\Omega;\mathbb{S}^{2\times2})\;:\;
\int_\Omega \langle\stress,\strain(\tens{v})\rangle\,\mathrm d\tens{x}=\ell(\tens{v}),
\ \forall\,\tens{v}\in\mathcal U_0
\right\},
\end{equation}
and the dual relaxed value
\begin{equation}\label{eq:ak_dual_cont}
c_{\mathrm{AK,du}}^{\mathrm{cont}}
:=
\inf_{\revisionA{\rho}\in\mathcal V_{\overline{V}}}
\inf_{\stress\in\Sigma_{\mathrm{ad}}}
\int_\Omega f_{\mathrm{c}}^{\mathrm{HS}}(\stress;\revisionA{\rho})\,\mathrm d\tens{x}.
\end{equation}
Then,
\begin{equation}
c_{\mathrm{HS\text{-}FMO}}^{\mathrm{cont}}
=
c_{\mathrm{AK,pr}}^{\mathrm{cont}}
=
c_{\mathrm{AK,du}}^{\mathrm{cont}}
=:
c_{\mathrm{AK}}^{\mathrm{cont}}.
\end{equation}
Moreover, the common value is attained \revisionA{by a homogenized design whose effective tensors are} orthotropic sequential laminates (of rank at most two).
\end{thrm}
\begin{proof}
Fix any HS-FMO admissible pair $(\revisionA{\rho},\tensf E)$, i.e.,
$\tensf E\in L^\infty(\Omega;\mathbb{S}^{2\times2\times2\times2})$, $\revisionA{\rho}\in\mathcal V_{\overline{V}}$,
and $\tensf E(\tens{x})\in\mathcal A^{(2)}(\revisionA{\rho}(\tens{x}))$ for a.e.\ $\tens{x}\in\Omega$. For this fixed pair, the single-loadcase compliance admits the dual representation \eqref{eq:compliance_sup},
\begin{equation}
c(\revisionA{\rho},\tensf E)
=
\sup_{\tens{w}\in\mathcal U_0}
\left\{
2\,\ell(\tens{w})
-\int_\Omega \langle \tensf E\,\strain(\tens{w}),\strain(\tens{w})\rangle\,\mathrm d\tens{x}
\right\}.
\end{equation}
Because $\tensf E(\tens{x})\in\mathcal A^{(2)}(\revisionA{\rho}(\tens{x}))$, for a.e.\ $\tens{x}\in\Omega$ and all $\strain\in\mathbb{S}^{2\times2}$ we have $\langle \tensf E(\tens{x})\strain,\strain\rangle \le f^{\mathrm{HS}}(\strain;\revisionA{\rho}(\tens{x}))$. Hence, for every $\tens{w}\in\mathcal U_0$,
\begin{equation}
2\,\ell(\tens{w})
-\int_\Omega \langle \tensf E\,\strain(\tens{w}),\strain(\tens{w})\rangle\,\mathrm d\tens{x}
\ge
2\,\ell(\tens{w})
-\int_\Omega f^{\mathrm{HS}}(\strain(\tens{w});\revisionA{\rho})\,\mathrm d\tens{x}.
\end{equation}
Taking supremum in $\tens{w}$ yields
\begin{equation}
c(\revisionA{\rho},\tensf E)\ge
\sup_{\tens{w}\in\mathcal U_0}
\left\{
2\,\ell(\tens{w})
-\int_\Omega f^{\mathrm{HS}}(\strain(\tens{w});\revisionA{\rho})\,\mathrm d\tens{x}
\right\}.
\end{equation}
Taking infimum over all HS-FMO admissible $(\revisionA{\rho},\tensf E)$ gives
\begin{equation}\label{eq:tight_first_ineq}
c_{\mathrm{HS\text{-}FMO}}^{\mathrm{cont}}
\ge
c_{\mathrm{AK,pr}}^{\mathrm{cont}}.
\end{equation}

For the converse inequality, fix $\delta>0$ and define
\begin{equation}
\mathcal J(\revisionA{\rho})
:=
\sup_{\tens{u}\in\mathcal U_0}
\left\{
2\,\ell(\tens{u})
-\int_\Omega f^{\mathrm{HS}}(\strain(\tens{u});\revisionA{\rho})\,\mathrm d\tens{x}
\right\},
\qquad \revisionA{\rho}\in\mathcal V_{\overline{V}}.
\end{equation}
By definition of $c_{\mathrm{AK,pr}}^{\mathrm{cont}}$, there exists $\revisionA{\rho_\delta}\in\mathcal V_{\overline{V}}$ such that $\mathcal J(\revisionA{\rho_\delta})\le c_{\mathrm{AK,pr}}^{\mathrm{cont}}+\delta$. Now use the recovery result for the single-loadcase Allaire--Kohn relaxation \citep{Allaire1993,Allaire1993od}: there exists a sequence of sequential-laminate microstructures with characteristic scale tending to zero; denote by $(\revisionA{\rho_{\delta,m}},\tensf E_{\delta,m})$ with $m\to\infty$ their associated homogenized (effective) design fields. In the present 2D, well-ordered isotropic setting, these recovering laminates can be chosen orthotropic and of rank at most two (rank-one in one branch, rank-two in the others, with layering directions tied to principal directions), cf. \citep{Allaire1993,Allaire2002}. These satisfy $\revisionA{\rho_{\delta,m}}\in \mathcal V_{\overline{V}}$, and $\tensf E_{\delta,m}(\tens{x})\in\mathcal A^{(2)}(\revisionA{\rho_{\delta,m}}(\tens{x}))$ for a.e. $\tens{x}\in\Omega$, and $c(\revisionA{\rho_{\delta,m}},\tensf E_{\delta,m})\longrightarrow \mathcal J(\revisionA{\rho_\delta})$. Since $c_{\mathrm{HS\text{-}FMO}}^{\mathrm{cont}}$ is the infimum over all such admissible pairs, for every $m$ we have $c_{\mathrm{HS\text{-}FMO}}^{\mathrm{cont}} \le c(\revisionA{\rho_{\delta,m}},\tensf E_{\delta,m})$. Passing to the limit $m\to\infty$ gives $c_{\mathrm{HS\text{-}FMO}}^{\mathrm{cont}} \le \mathcal J(\revisionA{\rho_\delta}) \le c_{\mathrm{AK,pr}}^{\mathrm{cont}}+\delta$. Because $\delta>0$ is arbitrary, we obtain
\begin{equation}\label{eq:tight_second_ineq}
c_{\mathrm{HS\text{-}FMO}}^{\mathrm{cont}}
\le
c_{\mathrm{AK,pr}}^{\mathrm{cont}}.
\end{equation}
Combining \eqref{eq:tight_first_ineq} and \eqref{eq:tight_second_ineq}, $c_{\mathrm{HS\text{-}FMO}}^{\mathrm{cont}} = c_{\mathrm{AK,pr}}^{\mathrm{cont}}$.

It remains to identify the primal and dual AK values directly in the fixed-volume setting. By definition \eqref{eq:gak_def}, for every fixed $\revisionA{\rho}$, $f_{\mathrm{c}}^{\mathrm{HS}}(\cdot;\revisionA{\rho})$ is the Fenchel conjugate of $f^{\mathrm{HS}}(\cdot;\revisionA{\rho})$ \revisionA{evaluated at $2\stress$}. In the homogenized setting, combining the minimum complementary-energy principle \citep[Theorem~4.1.9 and Eqs.~(4.36)--(4.39)]{Allaire2002}, the local character of the $G$-closure \citep[Theorem~2.1.2]{Allaire2002}, and the explicit Hashin--Shtrikman complementary bound/regularity \citep[Proposition~2.3.25 and Remark~2.3.27]{Allaire2002} yields, for every fixed $\revisionA{\rho}\in\mathcal V_{\overline{V}}$,
\begin{equation}
\sup_{\tens{u}\in\mathcal U_0}
\left\{
2\,\ell(\tens{u})
-\int_\Omega f^{\mathrm{HS}}(\strain(\tens{u});\revisionA{\rho})\,\mathrm d\tens{x}
\right\}
=
\inf_{\stress\in\Sigma_{\mathrm{ad}}}
\int_\Omega f_{\mathrm{c}}^{\mathrm{HS}}(\stress;\revisionA{\rho})\,\mathrm d\tens{x}.
\end{equation}
Taking the infimum over $\revisionA{\rho}\in\mathcal V_{\overline{V}}$ on both sides yields $c_{\mathrm{AK,pr}}^{\mathrm{cont}} = c_{\mathrm{AK,du}}^{\mathrm{cont}}$. Hence, $c_{\mathrm{HS\text{-}FMO}}^{\mathrm{cont}} = c_{\mathrm{AK,pr}}^{\mathrm{cont}} = c_{\mathrm{AK,du}}^{\mathrm{cont}} = c_{\mathrm{AK}}^{\mathrm{cont}}$. \revisionA{Finally, the classical homogenized problem admits a minimizer by the existence argument of \citet[Theorem~4.1.7]{Allaire2002}, which also applies with the weak-$\star$ closed volume constraint. By \citet[Theorem~4.1.12]{Allaire2002}, its effective tensors can be replaced pointwise by orthotropic sequential laminates of rank at most two, preserving the density field and optimal compliance. Since this design is HS-FMO admissible, the equality of optimal values established above proves attainment also for the continuum HS-FMO problem.}
\end{proof}

Theorem~\ref{th:tight} \revisionA{establishes attainment at the level of homogenized effective tensors. Their realization by two-phase
geometries relies on scale separation, with} characteristic length\revisionA{s and ratios between successive lamination scales tending to zero. These constructions provide effective elasticity tensors and reference compliance values for the numerical comparisons. Their approximation by finite-scale geometries would introduce additional choices of microstructural length scales and connections between neighboring cells.} One should therefore not interpret Theorem~\ref{th:tight} as an exact equivalence for a fixed finite element discretization of \eqref{eq:fmo_hashin} or \eqref{eq:fmo_hashin_reduced}. Three issues are relevant.

First, while numerical stabilization (filters, perimeter or gradient penalties on design variables) is often necessary in practice \citep{Jog1994}, it is not innocuous from the viewpoint of laminate realizability: the set of orthotropic sequential laminate tensors (at fixed $\revisionA{\rho}$) is generally not convex in $\tensf{E}$ (see Fig.~\ref{fig:hashin_laminates}). Consequently, regularization mechanisms that enforce local averaging or smoothness of $\tensf{E}$ may produce tensors that are no longer representable in this class, even if each unregularized pointwise optimum is.

Second, the Hashin--Shtrikman envelope $f^{\mathrm{HS}}(\strain;\revisionA{\rho})$ is tight in the sense that, for each fixed strain $\strain$, there exist sequential laminates attaining equality in \eqref{eq:hashin_shtrikman}. In a discretization with nonconstant strains inside an element (e.g., higher-order elements or multi-point quadrature), the elemental contribution involves multiple strain modes. Optimality then drives $\Ei$ against a superposition of strain directions. There is, in general, no guarantee that a single orthotropic sequential laminate simultaneously saturates all active bounds; thus laminate constructions designed for a single strain direction need not be exact discrete optimizers on a fixed mesh.

Last, Allaire's continuum duality framework constructs laminates pointwise from the local macroscopic fields. In discrete implementations this is often approximated by designing laminates for an elementwise representative strain/stress (e.g., midpoint or average \citep{Allaire2002}). Such a replacement is theoretically justified precisely in the scale-separation regime where the macroscopic strain is (approximately) constant at the sampling scale; otherwise it should be viewed as an approximation rather than an exact characterization of the discrete HS-FMO optimum.

In summary, Theorem~\ref{th:tight} identifies the correct continuum limit of the HS-FMO formulation. In the single-load, well-ordered isotropic case, the HS-relaxed problem is tight: it attains the same optimal value as the relaxed compliance minimization problem of \citet{Allaire1993od}, and both optima are attained by sequential laminates. Nevertheless, in our formulation we do not restrict the feasible set of the elasticity tensors to sequential laminates or any other specific microstructure class \citep{Allaire1999}. Moreover, for a fixed finite element discretization, discrete equivalence with laminate-based constructions is expected only asymptotically, and it may be altered by numerical regularization. For multiple loadcases, the optimal solution of \eqref{eq:fmo_hashin} or \eqref{eq:fmo_hashin_reduced} is generally not tight and therefore provides only a lower bound on the optimal value of compliance minimization over all microstructures; tightness in the multi-loadcase setting would additionally require enforcing Hashin--Shtrikman upper bounds on the sum of elastic energies \citep{Allaire1993wo,Allaire1996,Avellaneda1987}.

\subsection{Free Orthotropic Material Optimization under Hashin--Shtrikman Upper Bound}\label{sec:fmo_ortho}

In what follows, we consider the FMO problem under the Hashin--Shtrikman upper bounds \eqref{eq:fmo_hashin_reduced}, with the additional restriction that \revisionA{every tensor $\Ei$ in $\boldsymbol{\mathcal E}$ is orthotropic}. We refer to this setting as free-orthotropic material optimization (FOMO), and to its Hashin--Shtrikman-constrained version as HS-FOMO. \revisionA{For multiple loadcases, the lower-bound interpretation of the global optimal value then applies to composites with orthotropic effective tensors, rather than to unrestricted composites.}

Orthotropic materials possess three mutually orthogonal planes of symmetry, which in two dimensions implies that the stiffness tensor has four independent components. In a suitable coordinate system, the off-diagonal shear couplings vanish, i.e., $E_{1112} = E_{2212} = 0$. To make the dependence on the material orientation explicit, for each orthotropic tensor $\Ei$ we denote by $\Ei^{\mathrm{b}}$ its representation in a material (base) coordinate system in which $(\Ei^{\mathrm{b}})_{1112}=(\Ei^{\mathrm{b}})_{2212}=0$. We parametrize its global orientation by an angle $\varphi_i\in[0,\pi)$ and the corresponding planar rotation
\begin{equation}\label{eq:rot_matrix_2d}
    \semtrx{Q}(\varphi):=
    \begin{pmatrix}
        \cos\varphi & -\sin\varphi\\
        \sin\varphi & \cos\varphi
    \end{pmatrix}.
\end{equation}
The stiffness tensor $\Ei$ in global coordinates is then obtained by the standard tensor rotation rule
\begin{equation}\label{eq:ortho_tensor_rotation}
    (\Ei)_{ijkl}=\sum_{p,q,r,s=1}^2 (\semtrx{Q}(\varphi_i))_{ip}(\semtrx{Q}(\varphi_i))_{jq}(\semtrx{Q}(\varphi_i))_{kr}(\semtrx{Q}(\varphi_i))_{ls}\,(\Ei^{\mathrm{b}})_{pqrs}.
\end{equation}

Since the Hashin--Shtrikman envelope $f^{\mathrm{HS}}(\strain;\revisionA{\rho})$ for isotropic phases is iso\-tropic in $\strain$, and both $\Tr(\strain)$ and $\lVert\strain\rVert_F$ are rotation-invariant, the minimal activating volume is invariant under simultaneous rotations of the microstructure and the macroscopic strain. In particular, with $\strain^{\mathrm{b}}:=\semtrx{Q}(\varphi_i)^\trn \strain\,\semtrx{Q}(\varphi_i)$ we have
\begin{equation}\label{eq:hatv_rotation_invariance}
    \revisionA{\hat{\rho}}(\strain;\Ei)=\revisionA{\hat{\rho}}(\strain^{\mathrm{b}};\Ei^{\mathrm{b}}).
\end{equation}
Consequently, the worst-case volume fraction $\revisionA{\overline{\rho}_i}(\Ei)=\sup_{\lVert\strain\rVert_F=\sqrt{2}/2}\revisionA{\hat{\rho}}(\strain;\Ei)$ is independent of the orientation angle $\varphi_i$ and depends on $\Ei$ only through its base coefficients.

As described in the previous section, for two-dimensional composites with well-ordered isotropic constituents, the Hashin--Shtrikman upper energy bound \eqref{eq:hashin_shtrikman}, with $q(t;\revisionA{\rho_i})$ defined in \eqref{eq:qt}, is tight and admits an explicit form that splits into three regimes \citep{Allaire1993}: The first regime is attained by rank-one sequential laminates (which are orthotropic), while the other two are realized by rank-two sequential laminates with orthogonal layering directions, so the resulting effective tensors are again orthotropic. Hence, the bound \eqref{eq:hashin_shtrikman} remains tight even when the admissible set of tensors $\Ei$ is restricted to the orthotropic symmetry class with free rotations of the material coordinate system.

Under this restriction, and using \eqref{eq:hatv_rotation_invariance} to work in the base coordinate system of $\Ei$, \revisionA{the maximization over normalized strains in \eqref{eq:vi_plus_sup} can be reduced to a one-dimensional maximization over $t=|\Tr(\strain)|$}, as stated in the following proposition.

\begin{prpstn}[Trace-only reduction of the HS activating-volume map for orthotropy]\label{prop:trace_only_reduction}
    For any fixed orthotropic, plane stress $\Ei \in \{\tensf{E}: \Emin \preceq \tensf{E} \preceq \Emax\}$, let $\Ei^{\mathrm{b}}$ denote its base orientation (so that $(\Ei^{\mathrm{b}})_{1112}=(\Ei^{\mathrm{b}})_{2212}=0$). For any strain $\strain\in\mathbb{S}^{2\times2}$ with Frobenius norm $\sqrt{2}/2$, set
    $$
        E_{\max}(t;\Ei) := \max\left\{\langle \Ei\strain,\strain\rangle: \strain\in\mathbb{S}^{2\times2},\ \lVert\strain\rVert_F=\sqrt{2}/2,\ \lvert\Tr(\strain)\rvert = t\right\}.
    $$
    \revisionA{Then, the maximum activating volume fraction over normalized strains with fixed $t=|\Tr(\strain)|$ is obtained by replacing $\langle \Ei\strain,\strain\rangle$ with $E_{\max}(t;\Ei)$ in the three closed-form branches} $\revisionA{\hat{\rho}_1},\revisionA{\hat{\rho}_2},\revisionA{\hat{\rho}_3}$ stated in Proposition \ref{prop:explicit_volume}, with the same branch-activation conditions evaluated at $\langle \Ei\strain,\strain\rangle=E_{\max}(t;\Ei)$. Moreover, $E_{\max}(t;\Ei)$ admits the explicit form
    \begin{equation}\label{eq:excplicitenergybranches}
    E_{\max}(t;\Ei)=
    \begin{cases}
    & \begin{aligned}
        &\frac{t^2}{4} \big((\Ei^{\mathrm{b}})_{1111}+2(\Ei^{\mathrm{b}})_{1122}+(\Ei^{\mathrm{b}})_{2222}\big) + (\Ei^{\mathrm{b}})_{1212}(1-t^2) - \frac{t^2 \big((\Ei^{\mathrm{b}})_{1111}-(\Ei^{\mathrm{b}})_{2222}\big)^2}{4 \Xi}\\
        &\qquad\text{if }\ \Xi<0 \;\text{ and }\; t\le\frac{\lvert \Xi \rvert}{\sqrt{\Xi^2 + \big((\Ei^{\mathrm{b}})_{1111}-(\Ei^{\mathrm{b}})_{2222}\big)^2}},
    \end{aligned}\\
    & \begin{aligned}
        &\frac{1}{4}\big((\Ei^{\mathrm{b}})_{1111}-2 (\Ei^{\mathrm{b}})_{1122}+(\Ei^{\mathrm{b}})_{2222}\big) + (\Ei^{\mathrm{b}})_{1122} t^2 + \frac{\big\lvert (\Ei^{\mathrm{b}})_{1111}-(\Ei^{\mathrm{b}})_{2222}\big\rvert}{2} t \sqrt{1-t^2}\\
        &\qquad\text{otherwise},
    \end{aligned}
    \end{cases}
    \end{equation}
    where $\Xi:=(\Ei^{\mathrm{b}})_{1111}-2(\Ei^{\mathrm{b}})_{1122}+(\Ei^{\mathrm{b}})_{2222}-4(\Ei^{\mathrm{b}})_{1212}$.
\end{prpstn}
\begin{proof}
Because $\revisionA{\rho_i} \mapsto f^{\mathrm{HS}}(\strain;\revisionA{\rho_i})$ is nondecreasing by Proposition \ref{prop:hs_monotone} for all $\strain$, the minimal zero $\revisionA{\hat{\rho}}$ of $f^{\mathrm{HS}}(\strain;\revisionA{\rho_i}) - \langle \Ei\strain,\strain\rangle$ is nondecreasing in the energy $\langle \Ei\strain,\strain\rangle$. Hence, for any fixed $t$, $\revisionA{\hat{\rho}}(\strain;\Ei)$ is the largest when $\langle \Ei\strain,\strain\rangle$ is the largest. Therefore, for purposes of evaluating the minimal activating volume at that $t$, it suffices to plug the maximizing energy $E_{\max}(t;\Ei)$ into the explicit branch formulas. This eliminates all dependence on the remaining components of $\strain$ beyond $t$.

By \eqref{eq:hatv_rotation_invariance}, we may rotate the pair and work with $(\Ei^{\mathrm{b}},\strain^{\mathrm{b}})$ instead of $(\Ei,\strain)$. In particular $E_{\max}(t;\Ei)=E_{\max}(t;\Ei^{\mathrm{b}})$. To this goal, we work in the base coordinate system of $\Ei$ and write $\strain$ in components as
\begin{equation}
\strain = \begin{pmatrix}
    \varepsilon_{11} & \varepsilon_{12}\\
    \varepsilon_{12} & \varepsilon_{22}
\end{pmatrix} = 
\begin{pmatrix}
    \hat{t}/2 + \zeta & r\\
    r & \hat{t}/2 - \zeta
\end{pmatrix}
\end{equation}
where $\hat{t} = \Tr(\strain)$. Notice that $t = \lvert \hat{t}\rvert$. Then, the strain energy reads as
\begin{multline}
    \langle \Ei^{\mathrm{b}}\strain,\strain\rangle = \frac{\hat{t}^2}{4}\left((\Ei^{\mathrm{b}})_{1111} + 2(\Ei^{\mathrm{b}})_{1122} + (\Ei^{\mathrm{b}})_{2222}\right) + \zeta^2 \left((\Ei^{\mathrm{b}})_{1111} - 2(\Ei^{\mathrm{b}})_{1122} + (\Ei^{\mathrm{b}})_{2222}\right)\\
    + 4(\Ei^{\mathrm{b}})_{1212} r^2 + \hat{t} \zeta ((\Ei^{\mathrm{b}})_{1111} - (\Ei^{\mathrm{b}})_{2222}).
\end{multline}
The Frobenius norm constraint on $\strain$ requires $\zeta^2 + r^2 = (1 - \hat{t}^2)/4$. On one hand, this allows us to eliminate $r^2$ through $r^2 = (1 - \hat{t}^2)/4 - \zeta^2$. On the other hand, it implies $\zeta^2 \le (1 - \hat{t}^2)/4$ and $\hat{t}^2 \le 1$. Hence,
\begin{multline}\label{eq:opt_emax}
    E_{\max}(\hat{t};\Ei) = \max_{\zeta^2 \le \frac{1 - \hat{t}^2}{4}} \left[\frac{\hat{t}^2}{4}\left((\Ei^{\mathrm{b}})_{1111} + 2(\Ei^{\mathrm{b}})_{1122} + (\Ei^{\mathrm{b}})_{2222} - 4(\Ei^{\mathrm{b}})_{1212}\right)\right.\\
    \left.\phantom{\frac{1 - \hat{t}^2}{4}}+ \zeta^2 \left((\Ei^{\mathrm{b}})_{1111} - 2(\Ei^{\mathrm{b}})_{1122} + (\Ei^{\mathrm{b}})_{2222} - 4(\Ei^{\mathrm{b}})_{1212}\right)
    + \hat{t} \zeta ((\Ei^{\mathrm{b}})_{1111} - (\Ei^{\mathrm{b}})_{2222}) + (\Ei^{\mathrm{b}})_{1212}\right],
\end{multline}
which is a quadratic function in $\zeta$. It is strictly concave in $\zeta$ iff $\Xi = (\Ei^{\mathrm{b}})_{1111} - 2(\Ei^{\mathrm{b}})_{1122} + (\Ei^{\mathrm{b}})_{2222} - 4(\Ei^{\mathrm{b}})_{1212} < 0$, i.e., shear \revisionA{stiffness exceeds the normal deviatoric stiffness}. Then, the maximum is attained at the stationary point $\zeta^\star = -\hat{t} ((\Ei^{\mathrm{b}})_{1111} - (\Ei^{\mathrm{b}})_{2222})/(2\Xi)$ if $(\zeta^\star)^2 \le (1-\hat{t}^2)/4$ or, equivalently, if $t = \lvert \hat{t}\rvert \le \lvert \Xi \rvert /\sqrt{\Xi^2 + ((\Ei^{\mathrm{b}})_{1111}-(\Ei^{\mathrm{b}})_{2222})^2}$. Inserting $\zeta^\star$ into the expression for $E_{\max}(\hat{t};\Ei)$ gives
\begin{equation}
    \frac{\hat{t}^2}{4} ((\Ei^{\mathrm{b}})_{1111}+2(\Ei^{\mathrm{b}})_{1122}+(\Ei^{\mathrm{b}})_{2222}) + (\Ei^{\mathrm{b}})_{1212}(1-\hat{t}^2)
        - \frac{\hat{t}^2 ((\Ei^{\mathrm{b}})_{1111}-(\Ei^{\mathrm{b}})_{2222})^2}{4 \Xi}.
\end{equation}
Since $\hat{t}^2 = t^2$ and the maximum value is even in $\hat{t}$, we get the first branch in \eqref{eq:excplicitenergybranches}.

In the remaining cases when $\zeta^2 > (1-\hat{t}^2)/4$ or $\Xi\ge0$, maximum of \eqref{eq:opt_emax} occurs at an endpoint, so that $(\zeta^\star)^2 = \frac{1 - \hat{t}^2}{4}$. Inserting this into \eqref{eq:opt_emax}, we receive the maximum energy
\begin{equation}
    \frac{1}{4}((\Ei^{\mathrm{b}})_{1111}-2 (\Ei^{\mathrm{b}})_{1122}+(\Ei^{\mathrm{b}})_{2222})
        + (\Ei^{\mathrm{b}})_{1122} \hat{t}^2
        + \frac{\lvert (\Ei^{\mathrm{b}})_{1111}-(\Ei^{\mathrm{b}})_{2222}\rvert}{2} \lvert\hat{t}\rvert \sqrt{1-\hat{t}^2},
\end{equation}
which is exactly the second branch in \eqref{eq:excplicitenergybranches} when using $t=\lvert \hat{t}\rvert$.
\end{proof}

By Propositions \ref{prop:explicit_volume} and \ref{prop:trace_only_reduction}, $\revisionA{\hat{\rho}}(t;\Ei^{\mathrm{b}})$ is univariate on $t\in[0,1]$, obtained by composing $E_{\max}(t;\Ei)$ with the branchwise formulas in Proposition \ref{prop:explicit_volume}, yielding the orthotropic FMO formulation
\begin{subequations}\label{eq:fmo_hashin_ortho}
\begin{align}
    \min_{\revisionA{\sevek{c}}, \revisionA{\boldsymbol{\mathcal E}}} \; & \sum_{j=1}^{\nlc} c_j \label{eq:fmo_hashin_ortho_compliance}\\
    \mathrm{s.t.}\; & \begin{pmatrix}
        c_j & -\sevek{f}_j^\trn\\
        -\sevek{f}_j & \semtrx{K}_j\left(\revisionA{\boldsymbol{\mathcal E}}\right)
    \end{pmatrix}\succeq \semtrx{0}, \quad j = 1,\dots,\nlc,\label{eq:fmo_hashin_ortho_lmi}\\
    & \Emin \preceq \Ei \preceq \Emax,\quad i = 1,\dots,\nel,\label{eq:fmo_hashin_ortho_bounds}\\
    & \frac{1}{\nel}\sum_{i=1}^{\nel} \sup_{t \in [0,1]} \revisionA{\hat{\rho}}(t;\Ei^{\mathrm{b}}) \le \overline{V},\label{eq:fmo_hashin_ortho_volume}\\
    & \tensf{E}_i \text{ orthotropic},\quad i = 1,\dots,\nel,\label{eq:fmo_hashin_ortho_ortho}
\end{align}
\end{subequations}
in which---for each $\Ei$---the univariate supremum in \eqref{eq:fmo_hashin_ortho_volume} can be computed to approximate global optimality by coarse sampling to bracket candidate intervals, followed by golden-section refinement on each bracket; empirical plots contained at most two maxima.

Problem \eqref{eq:fmo_hashin_ortho} is again nonconvex due to the volume constraint \eqref{eq:fmo_hashin_ortho_volume}. For a single-loadcase, its continuum relaxation value is consistent with the HS/AK value from Theorem~\ref{th:tight}; however, this does not imply exact attainment on a fixed finite-element mesh. Laminate attainability should be understood in the continuum relaxation (vanishing-scale) sense.

\subsection{Sequential Global Programming for Free-Orthotropic Material Optimization}\label{sec:sgp}

To solve problem \eqref{eq:fmo_hashin_ortho}, we apply the sequential global programming (SGP) approach, which was first suggested in \citep{Semmler-SIAM-2018} and later applied in different contexts, see, e.g. \citep{Nees-SAMO-2023, Vu-SAMO-2023}. The SGP algorithm is an iterative method in which, in each iteration, a first-order model is established that is separable with respect to the anisotropic material tensors $\Ei, i=1,\ldots,\nel$. To explain this in more detail, we first rewrite problem (\ref{eq:fmo_hashin_ortho}) in the following form:
\begin{subequations}\label{eq:fmo_hashin_ortho_nested}
\begin{align}
    \min_{\revisionA{\boldsymbol{\mathcal E}}} \; & \sum_{j=1}^{\nlc} c_j\left(\revisionA{\boldsymbol{\mathcal E}}\right) \label{eq:fmo_hashin_ortho_compliance2}\\
    \mathrm{s.t.}\; & \Emin \preceq \Ei \preceq \Emax,\quad i = 1,\dots,\nel,\label{eq:fmo_hashin_ortho_bounds2}\\
    & \frac{1}{\nel}\sum_{i=1}^{\nel} \sup_{t \in [0,1]} \revisionA{\hat{\rho}}(t;\Ei^{\mathrm{b}}) = \overline{V},\label{eq:fmo_hashin_ortho_volume2}\\
    & \Ei \text{ orthotropic},\quad i = 1,\dots,\nel,\label{eq:fmo_hashin_ortho_ortho2}
\end{align}
\end{subequations}
where $c_j\left(\revisionA{\boldsymbol{\mathcal E}}\right) = \sevek{f}_j^\trn \semtrx{K}_j\left(\revisionA{\boldsymbol{\mathcal E}}\right)^{-1}\sevek{f}_j$. \revisionA{The volume constraint can be imposed as an equality without changing the optimal value. Indeed, interpolating each tensor toward $\Emax$ preserves orthotropy and the tensor bounds while compliance cannot increase. Along this path, the minimum required volume fractions vary continuously and reach $1$ at $\Emax$. Thus, any unused volume can be exhausted without increasing compliance, so at least one optimum saturates the budget}. 

Since constraints (\ref{eq:fmo_hashin_ortho_bounds2}) and (\ref{eq:fmo_hashin_ortho_ortho2}) are formulated element-wise and the volume constraint (\ref{eq:fmo_hashin_ortho_volume2}) is separable with respect to $\revisionA{\boldsymbol{\mathcal E}}$, it remains to approximate the cost function (\ref{eq:fmo_hashin_ortho_compliance2}) by a separable first-order model. Following \citep{Vu-SAMO-2023}, we use a variant of the separable model introduced in \citep{Stingl-SIAM-2009} for FMO problems. To do so, we denote by $\revisionA{\boldsymbol{\mathcal E}^\mathrm{M}:=\left(\EiM\right)_{i=1}^{\nel}}$ the corresponding Kelvin--Mandel matrices and, at the expansion point $\revisionA{\overline{\boldsymbol{\mathcal E}}:=\left(\overline{\Ei}\right)_{i=1}^{\nel}}$, we define

\begin{equation}\label{eq:Gij_def}
\semtrx{G}_i^j := \revisionA{-}\frac{\partial c_j \left(\revisionA{\overline{\boldsymbol{\mathcal E}}}\right)}{\partial \EiM}\in \mathbb{S}^{3\times3}.
\end{equation}
We then approximate the compliance for the loadcase $j$ as
\begin{align}
\hat{c}_j\left(\revisionA{\boldsymbol{\mathcal E}};\revisionA{\overline{\boldsymbol{\mathcal E}}}\right) & = C + \sum_{i=1}^{\nel} \left\langle (\overline{\EiM} - \semtrx{L}_i^j) \semtrx{G}_i^j (\overline{\EiM} - \semtrx{L}_i^j), (\EiM - \semtrx{L}_i^j)^{-1}\right\rangle_{\mathbb{S}^{3\times3}},
\end{align}
where the constant $C$ is chosen such that the function value of $\hat{c}_j$ coincides with that of $c_j$ at the expansion point $\revisionA{\overline{\boldsymbol{\mathcal E}}}$ and $\revisionA{\left(\semtrx{L}_i^j\right)_{i=1,\ldots,\nel}^{j=1,\ldots,\nlc}}$ are generalized asymptotes, which are chosen to be zero matrices in the sequel. Using a Lagrangian approach, we obtain the model problem
\begin{equation}\label{eq:fmo_hashin_ortho_separable}
    \min_{\revisionA{\substack{\boldsymbol{\mathcal E}\\\Emin\preceq\Ei\preceq\Emax,\;\Ei\text{ orth.}\\i=1,\ldots,\nel}}} \max_{\lambda}\; \sum_{j=1}^{\nlc} \hat{c}_j\left(\revisionA{\boldsymbol{\mathcal E}};\revisionA{\overline{\boldsymbol{\mathcal E}}}\right) + \lambda\left( \frac{1}{\nel}\sum_{i=1}^{\nel} \sup_{t \in [0,1]} \revisionA{\hat{\rho}}(t;\Ei^{\mathrm{b}}) - \overline{V}\right).
\end{equation}
We note that, due to the orthotropy and volume constraints, problem (\ref{eq:fmo_hashin_ortho_separable}) is \revisionA{nonconvex. Interchanging the order of minimization and maximization yields the Lagrangian dual problem}
%
%
\begin{equation}\label{eq:fmo_hashin_ortho_separable_dual}
    \max_{\lambda}\; \psi(\lambda),
\end{equation}
where the function value of the dual function $\psi$ for a given $\lambda$ is computed by solving 
\begin{subequations}\label{eq:fmo_hashin_ortho_separable_lambda}
\begin{align}
    \min_{\revisionA{\boldsymbol{\mathcal E}}}\; & \sum_{j=1}^{\nlc} \hat{c}_j\left(\revisionA{\boldsymbol{\mathcal E}};\revisionA{\overline{\boldsymbol{\mathcal E}}}\right) + \lambda\left( \frac{1}{\nel}\sum_{i=1}^{\nel} \sup_{t \in [0,1]} \revisionA{\hat{\rho}}(t;\Ei^{\mathrm{b}}) - \overline{V}\right) \\
    \mathrm{s.t.} \; & \Emin \preceq \Ei \preceq \Emax,\quad i = 1,\dots,\nel,\\
    & \Ei \text{ orthotropic},\quad i = 1,\dots,\nel.
\end{align}
\end{subequations}
\revisionA{Because problem (\ref{eq:fmo_hashin_ortho_separable}) is nonconvex, strong duality and the existence of a saddle point are not guaranteed in general. We therefore use the dual problem (\ref{eq:fmo_hashin_ortho_separable_dual}) as a computational device rather than relying on an a priori strong-duality result. Nevertheless, the dual construction provides a conditional a posteriori optimality certificate for the model problem. For the fixed expansion point $\overline{\boldsymbol{\mathcal E}}$, let $L(\boldsymbol{\mathcal E},\lambda)$ denote the Lagrangian expression in \eqref{eq:fmo_hashin_ortho_separable}. In the idealized setting, let $\boldsymbol{\mathcal E}(\lambda)$ denote a global minimizer of (\ref{eq:fmo_hashin_ortho_separable_lambda}) for fixed $\lambda$. If, for some $\lambda^\star$, the corresponding minimizer $\boldsymbol{\mathcal E}(\lambda^\star)$ satisfies the volume equality (\ref{eq:fmo_hashin_ortho_volume2}), then, for every feasible design $\boldsymbol{\mathcal E}$ satisfying (\ref{eq:fmo_hashin_ortho_bounds2})--(\ref{eq:fmo_hashin_ortho_ortho2}), including the volume equality,
\begin{equation}
\sum_{j=1}^{\nlc}\hat{c}_j(\boldsymbol{\mathcal E};\overline{\boldsymbol{\mathcal E}})
= L(\boldsymbol{\mathcal E},\lambda^\star)
\geq L(\boldsymbol{\mathcal E}(\lambda^\star),\lambda^\star)
= \sum_{j=1}^{\nlc}\hat{c}_j(\boldsymbol{\mathcal E}(\lambda^\star);\overline{\boldsymbol{\mathcal E}}).
\end{equation}
Hence, $\boldsymbol{\mathcal E}(\lambda^\star)$ is a global solution of the SGP model problem (\ref{eq:fmo_hashin_ortho_separable}), and $(\boldsymbol{\mathcal E}(\lambda^\star),\lambda^\star)$ is a saddle point of the Lagrangian. In practice, we determine $\lambda$ by bisection and verify the volume equality a posteriori. The hierarchical sampling strategy described below provides only an approximation of the required global minimizers, so the above certificate should be understood for the idealized model problem.}
Rather than solving (\ref{eq:fmo_hashin_ortho_separable_lambda}) at once, thanks to separability, we can solve for each $\Ei$ individually, i.~e., we solve $\nel$ problems of the form:
\begin{equation}\label{eq:fmo_hashin_ortho_separable_lambda_i}
    \min_{\Emin \preceq \Ei \preceq \Emax, \Ei \text{ orth.}} \sum_{j=1}^{\nlc} \left\langle \overline{\EiM}\semtrx{G}_i^j\overline{\EiM}, (\EiM)^{-1}\right\rangle_{\mathbb{S}^{3\times3}} + \frac{\lambda}{\nel} \sup_{t \in [0,1]} \revisionA{\hat{\rho}}(t;\Ei^{\mathrm{b}}).
\end{equation}
We note that we skipped the constant $C-\lambda V$ in the objective here and set all asymptotes to zero. 

In order to treat the orthotropy constraint, we parametrize each element tensor $\Ei$ by its base coefficients and a rotation angle. In Kelvin--Mandel form, the base orientation reads
\begin{equation}\label{eq:ortho_param}
    \EiMb :=
    \begin{pmatrix}
        (\Ei^{\mathrm{b}})_{1111} & (\Ei^{\mathrm{b}})_{1122} & 0\\
        (\Ei^{\mathrm{b}})_{1122} & (\Ei^{\mathrm{b}})_{2222} & 0\\
        0 & 0 & 2(\Ei^{\mathrm{b}})_{1212}
    \end{pmatrix},
    \qquad
    \EiM = \semtrx{R}(\varphi_i)^\trn\EiMb\semtrx{R}(\varphi_i),
\end{equation}
with $\varphi_i\in[0,\pi)$. The matrix $\semtrx{R}(\varphi)\in\mathbb R^{3\times3}$ is the Kelvin--Mandel rotation induced by $\semtrx{Q}(\varphi)$ in \eqref{eq:rot_matrix_2d}, i.e., for every symmetric strain $\strain$, with Mandel vector $\strain^{\mathrm{M}}:=(\varepsilon_{11},\varepsilon_{22},\sqrt2\,\varepsilon_{12})^\trn$, we have $\big(\semtrx{Q}(\varphi)^\trn \strain\,\semtrx{Q}(\varphi)\big)^{\mathrm{M}}=\semtrx{R}(\varphi)\,\strain^{\mathrm{M}}$.

Since $\Emin$ and $\Emax$ are isotropic, the Loewner bounds are rotation\revisionA{ally} invariant and can be imposed directly on the base matrix:
\begin{subequations}\label{eq:ortho_param_voigt}
\begin{align}
    \begin{pmatrix}
        (\Emin)_{1111} & (\Emin)_{1122}\\
        (\Emin)_{1122} & (\Emin)_{2222}
    \end{pmatrix}
    \preceq
    \begin{pmatrix}
        (\EiMb)_{11} & (\EiMb)_{12}\\
        (\EiMb)_{12} & (\EiMb)_{22}
    \end{pmatrix}
    &\preceq
    \begin{pmatrix}
        (\Emax)_{1111} & (\Emax)_{1122}\\
        (\Emax)_{1122} & (\Emax)_{2222}
    \end{pmatrix},\\
    2(\Emin)_{1212} \leq (\EiMb)_{33} &\leq 2(\Emax)_{1212}.
\end{align}
\end{subequations}
With this parametrization, \eqref{eq:fmo_hashin_ortho_separable_lambda_i} can be written as
\begin{equation}\label{eq:fmo_hashin_ortho_separable_lambda_i2}
    \min_{\substack{\EiMb,\,\varphi_i\\ \text{s.t.\ }\eqref{eq:ortho_param_voigt}}}
    \sum_{j=1}^{\nlc} \left\langle \overline{\EiM}\semtrx{G}_i^j\overline{\EiM}, (\EiM)^{-1}\right\rangle_{\mathbb{S}^{3\times3}} + \frac{\lambda}{\nel} \sup_{t \in [0,1]} \revisionA{\hat{\rho}}(t;\Ei^{\mathrm{b}}),
\end{equation}
where $\EiM$ depends on $(\EiMb,\varphi_i)$ through \eqref{eq:ortho_param}.
Problem (\ref{eq:fmo_hashin_ortho_separable_lambda_i}) is a $5$-dimensio\-nal optimization problem with semidefinite constraints and, due to the properties of the volume constraint, a nondifferentiable objective. As no global optimization solvers are available, which could treat problems with this structure directly, we have implemented a sampling based solution approach. To do so, we discretized the interval $[0,\pi/2]$ equidistantly by $9001$ angles. The resulting discrete set of angles is denoted by $\Phi^\mathrm{G}$. Similarly, cuboid $[(\Emin)_{1111},(\Emax)_{1111}]\times [(\Emin)_{2222},(\Emax)_{2222}]\times [2(\Emin)_{1212},2(\Emax)_{1212}]$ is discretized equidistantly by $101^3$ points and the corresponding set of discrete points is denoted by $\Theta^\mathrm{G}$. Finally, for each grid point $(E^\mathrm{G}_{11},E^\mathrm{G}_{22},E^\mathrm{G}_{33})_i \in \Theta^\mathrm{G}$ we compute the interval $I_i$, which contains all choices $E_{12}^\mathrm{G}$ such that 
$$
\begin{pmatrix}
    (\Emin)_{1111} & (\Emin)_{1122}\\
    (\Emin)_{1122} & (\Emin)_{2222}
\end{pmatrix} 
\preceq
\begin{pmatrix}
    E^\mathrm{G}_{11} & E_{12}^\mathrm{G}\\
    E_{12}^\mathrm{G} & E^\mathrm{G}_{22}
\end{pmatrix} 
\preceq
\begin{pmatrix}
    (\Emax)_{1111} & (\Emax)_{1122}\\
    (\Emax)_{1122} & (\Emax)_{2222}
\end{pmatrix}.
$$
Then, if $I_i$ is nonempty, we again discretize this interval by $101$ points. We denote the corresponding grid by $\mathcal{X}_i$. If $I_i$ is empty, the corresponding point $(E^\mathrm{G}_{11},E^\mathrm{G}_{22},E^\mathrm{G}_{33})_i$ is eliminated from $\Theta^\mathrm{G}$. In this way, we sample only points that are feasible. Then, in principle, we could solve (\ref{eq:fmo_hashin_ortho_separable_lambda_i2}) approximately, by evaluating the cost function for any sample point and accept the one with lowest function value as approximate solution. However, taking into account that the 5D-grid has about 100 billion points, this strategy would not be tractable. Instead, we use a hierarchical grid approach, which only needs about 10 million function evaluations. 
\revisionA{Moreover, the inverse tensors and estimated volumes in~(\ref{eq:fmo_hashin_ortho_separable_lambda_i2}) are precomputed offline for each grid point. The resulting sampling data can be reused for subsequent optimization runs with the same phase parameters and grid resolution.}

\begin{rmrk}
\begin{enumerate}
\item We have tested the resolution of our sampling set
by projecting a solution we obtained from the algorithm described in Section \ref{sec:allaire_alg} onto our sampling grid. The loss in the objective function value we observed was typically very small.
\item We cannot guarantee that we find the best point in the grid using our hierarchical grid approach in general. However, in our experiments we observed that the function value for the points we find were either optimal (on the full grid) or correct up to at least 5 digits. 
\item Without globalization, the iterative solution strategy we described above for the solution of problem (\ref{eq:fmo_hashin_ortho_nested}) is not guaranteed to converge. A remedy is to add a line-search step after each solution of a subproblem of type (\ref{eq:fmo_hashin_ortho_separable}). However, our experiments revealed that for the compliance minimization problems we solve in this article no globalization is needed. Thus, we omit the description here.
\end{enumerate}
\end{rmrk}


\subsection{Single-Loadcase via Alternating Minimization and Sequential Laminates}\label{sec:allaire_alg}

After introducing the bound-based HS-FOMO models and the SGP solver, we now present the single-loadcase laminate-based alternating minimization baseline used in Section~\ref{sec:results}. Unlike the primal strain-energy formulation used above, this method updates material through the dual Hashin--Shtrikman complementary-energy envelope using explicit continuum minimizers of the local material subproblem; in 2D and single-loadcase, these minimizers are orthotropic sequential laminates of rank one or two. The complementary-energy envelope is dual to the primal HS envelope, so this viewpoint is consistent with the continuum tightness statement in Theorem~\ref{th:tight}. For completeness, we recall only the formulas needed for implementation; supporting derivations are deferred to Appendix~\ref{appendix:rederivation_burzain}.

\begin{dfntn}[Optimal lower bound on complementary energy {\citep[Definition~2.3.15]{Allaire2002}}]
    Let $\Emin \prec \Emax$ be two well-ordered isotropic elasticity tensors and let $\revisionA{\rho}\in[0,1]$ be the volume fraction of $\Emax$. Denote by $G_{\revisionA{\rho}}$ the set of all effective elasticity tensors $\Estar$ obtainable by homogenization of the two phases $(\Emin,\Emax)$ with volume fraction $\revisionA{\rho}$.
    For all $\stress\in\mathbb{S}^{2\times 2}$, the HS lower bound of complementary energy holds:
    $\langle {\Estar}^{-1}\stress,\stress\rangle\geq f_{\mathrm{c}}^{\mathrm{HS}}(\stress;\revisionA{\rho})$, where 
    \begin{equation}
        f_{\mathrm{c}}^{\mathrm{HS}}(\stress;\revisionA{\rho}):=\langle (\Emax)^{-1}\stress,\stress\rangle
        +(1-\revisionA{\rho})\max_{\strain\in\mathbb{S}^{d\times d}}\left\{ 2\langle\stress,\strain\rangle
        - \left\langle\left((\Emin)^{-1}-(\Emax)^{-1}\right)^{-1}\strain,\strain\right\rangle - \revisionA{\rho}\, g^{\mathrm{c}}(\strain)\right\}.
    \end{equation}
    Here, $g^\mathrm{c}(\strain)$ is the nonlocal term
    $g^{\mathrm{c}}(\strain):=\max_{\lVert\tens{e}\rVert=1}\langle \tensf{F}^{\mathrm{c}}_{\Emax}(\tens{e})\strain,\strain\rangle$. For each lamination direction $\lVert\tens{e}\rVert=1$, the tensor $\tensf{F}^{\mathrm{c}}_{\Emax}(\tens{e})\in\mathbb{R}^{d\times d\times d\times d}$ satisfies
    \begin{equation}
         \langle \tensf{F}_{\Emax}^\mathrm{c}(\tens{e})\strain,\strain\rangle = \langle \Emax\strain,\strain\rangle-\frac{1}{\mu^+}\lVert(\Emax\strain)\tens{e}\rVert^2
    + \left(\frac{1}{\mu^+} - \frac{1}{\kappa^+ + 2 \mu^+ (d - 1)/d}\right)
    \left\langle (\Emax\strain)\tens{e},\, \tens{e}\right\rangle^2. 
    \end{equation}
    %
\end{dfntn}

The quantity $f_{\mathrm{c}}^{\mathrm{HS}}(\stress;\revisionA{\rho})$ is the local envelope minimized in the material step of alternating minimization. For two-dimensional mixtures of well-ordered isotropic phases, this lower bound is tight in the relaxation sense, and equality is attained by finite-rank laminates (cf. \citep[Proposition~2.3.25]{Allaire2002} and \citep[Theorem~2]{Burazin2021}). We now state the explicit branch-wise formula used in the algorithm.

\begin{thrm}[Explicit formula of $f_{\mathrm{c}}^{\mathrm{HS}}$ in 2D {\citep[Theorem~2]{Burazin2021}}]\label{th:burazin_explicit_fc}
Assume $d=2$, $\tensf{0} \prec \Emin\prec \Emax$, and denote by $\sigma_1,\sigma_2$ the eigenvalues of $\stress$. Let
\begin{subequations}
\begin{align}
\Delta\kappa&:=\kappa^+-\kappa^-,\qquad
\Delta\mu:=\mu^+-\mu^-,\\
s_+&:=\lvert\sigma_1+\sigma_2\rvert,\qquad
s_-:=\lvert\sigma_1-\sigma_2\rvert,\\
d_1^{\mathrm c}(\revisionA{\rho})&:=(1-\revisionA{\rho})\kappa^-\mu^-(\kappa^++\mu^+)+\revisionA{\rho}\kappa^+\mu^+(\kappa^-+\mu^-),\\
d_2^{\mathrm c}(\revisionA{\rho})&:=\kappa^-(\kappa^++\mu^+)+\revisionA{\rho}\mu^+\Delta\kappa,\\
d_3^{\mathrm c}(\revisionA{\rho})&:=\mu^-(\kappa^++\mu^+)+\revisionA{\rho}\kappa^+\Delta\mu,\\
a(\stress)&:=\kappa^-\kappa^+\Delta\mu\,s_-+\mu^-\mu^+\Delta\kappa\,s_+.
\end{align}
\end{subequations}
Define
\begin{subequations}
\begin{align}
f_{\mathrm c,1}^{\mathrm{HS}}(\stress;\revisionA{\rho})
&:= (1-\revisionA{\rho})\langle (\Emin)^{-1}\stress,\stress\rangle+\revisionA{\rho}\langle (\Emax)^{-1}\stress,\stress\rangle
-\frac{\revisionA{\rho}(1-\revisionA{\rho})a(\stress)^2}{4\kappa^-\kappa^+\mu^-\mu^+\,d_1^{\mathrm c}(\revisionA{\rho})},\\
f_{\mathrm c,2}^{\mathrm{HS}}(\stress;\revisionA{\rho})
&:=\langle (\Emax)^{-1}\stress,\stress\rangle
+\frac{(1-\revisionA{\rho})\Delta\kappa(\kappa^++\mu^+)(\sigma_1+\sigma_2)^2}{4\kappa^+\,d_2^{\mathrm c}(\revisionA{\rho})},\\
f_{\mathrm c,3}^{\mathrm{HS}}(\stress;\revisionA{\rho})
&:=\langle (\Emax)^{-1}\stress,\stress\rangle
+\frac{(1-\revisionA{\rho})\Delta\mu(\kappa^++\mu^+)(\sigma_1-\sigma_2)^2}{4\mu^+\,d_3^{\mathrm c}(\revisionA{\rho})}.
\end{align}
\end{subequations}
Then
\begin{equation}\label{eq:burazin_explicit_fc}
f_{\mathrm{c}}^{\mathrm{HS}}(\stress;\revisionA{\rho})=
\begin{cases}
f_{\mathrm c,1}^{\mathrm{HS}}(\stress;\revisionA{\rho}),& \text{if } \revisionA{\rho}\mu^+\Delta\kappa\,s_+<d_2^{\mathrm c}(\revisionA{\rho})\,s_- \;\text{ and}\; \revisionA{\rho}\kappa^+\Delta\mu\,s_-<d_3^{\mathrm c}(\revisionA{\rho})\,s_+,\\
f_{\mathrm c,2}^{\mathrm{HS}}(\stress;\revisionA{\rho}),& \text{if }
\revisionA{\rho}\mu^+\Delta\kappa\,s_+\ge d_2^{\mathrm c}(\revisionA{\rho})\,s_-,\\
f_{\mathrm c,3}^{\mathrm{HS}}(\stress;\revisionA{\rho}),& \text{if }
\revisionA{\rho}\kappa^+\Delta\mu\,s_-\ge d_3^{\mathrm c}(\revisionA{\rho})\,s_+.
\end{cases}
\end{equation}
Branch-wise attainability is by finite-rank laminates:
\begin{enumerate}
\item In case $f_{\mathrm c,1}^{\mathrm{HS}}$, a simple laminate attains the bound, with lamination direction orthogonal to a principal direction of $\stress$.
\item In case $f_{\mathrm c,2}^{\mathrm{HS}}$, a rank-two sequential laminate attains the bound with principal directions of $\stress$ and
\begin{equation}
m_1=\frac12+\frac{d_2^{\mathrm c}(\revisionA{\rho})(\sigma_2-\sigma_1)}{2\revisionA{\rho}\mu^+\Delta\kappa(\sigma_1+\sigma_2)},
\qquad
m_2=\frac12+\frac{d_2^{\mathrm c}(\revisionA{\rho})(\sigma_1-\sigma_2)}{2\revisionA{\rho}\mu^+\Delta\kappa(\sigma_1+\sigma_2)}.
\end{equation}
\item In case $f_{\mathrm c,3}^{\mathrm{HS}}$, a rank-two sequential laminate attains the bound with principal directions of $\stress$ and
\begin{equation}
m_1=\frac12+\frac{d_3^{\mathrm c}(\revisionA{\rho})(\sigma_1+\sigma_2)}{2\revisionA{\rho}\kappa^+\Delta\mu(\sigma_2-\sigma_1)},
\qquad
m_2=\frac12+\frac{d_3^{\mathrm c}(\revisionA{\rho})(\sigma_1+\sigma_2)}{2\revisionA{\rho}\kappa^+\Delta\mu(\sigma_1-\sigma_2)}.
\end{equation}
\end{enumerate}
\end{thrm}
Figure~\ref{fig:hashin_laminates} illustrates the geometric relation between the HS bound and its laminate realizers. The corresponding branch-wise laminate parameters also provide a closed-form local tensor update: for each element $i$, let $\tens{e}_{i,j}$ and $m_{i,j}$ denote the lamination directions and weights, with $p_i=1$ in case 1 and $p_i=2$ in cases 2--3. Define
\begin{algorithm}[b]
    \caption{Laminate-based alternating minimization with bisection on $\lambda$}
    \label{algo:allaire_alt_main}
    \KwIn{Initial $\revisionA{\sevek{\rho}^{0}}$, $\revisionA{\boldsymbol{\mathcal E}^{0}}$, target volume fraction $\overline{V}$.}
    \For{$k=0,1,2,\dots$}{
      Solve state equation for $\sevek{u}^k$ and \revisionA{compute $\stress_i^k$ for all elements $i$}\;
    	  Find $\lambda$ (by bisection) such that, with
    	  \[
    	    \revisionA{\rho_i}(\lambda)\in\argmin_{\revisionA{\rho}\in[0,1]} f_{\mathrm{c}}^{\mathrm{HS}}(\stress_i^k;\revisionA{\rho})+\lambda \revisionA{\rho},
    	    \qquad
    	    \Ei(\lambda)=\textsc{LaminateUpdate}(\stress_i^k,\revisionA{\rho_i}(\lambda)),
    	  \]
    	  the volume constraint $\frac{1}{\nel}\sum_{i=1}^{\nel} \revisionA{\rho_i}(\lambda)=\overline{V}$ holds\;
    	  Set $\revisionA{\rho_i^{k+1}}\gets \revisionA{\rho_i}(\lambda)$ and $\Ei^{k+1}\gets \Ei(\lambda)$ for all $i$\;
    	}
\end{algorithm}
\begin{equation}
    \tensf{R}:=\left((\Emin)^{-1}-(\Emax)^{-1}\right)^{-1},
    \qquad
    \tensf{B}_i:=\tensf{R}+\revisionA{\rho_i}\sum_{j=1}^{p_i} m_{i,j}\,\tensf{F}_{\Emax}^{\mathrm c}(\tens{e}_{i,j}),
\end{equation}
and set
\begin{equation}
    \Ei=\left((\Emax)^{-1}+(1-\revisionA{\rho_i})\,\tensf{B}_i^{-1}\right)^{-1}.
\end{equation}
This defines the \textsc{LaminateUpdate} map used in Algorithm~\ref{algo:allaire_alt_main}.

\begin{figure}[!b]
    \includegraphics[width=\linewidth]{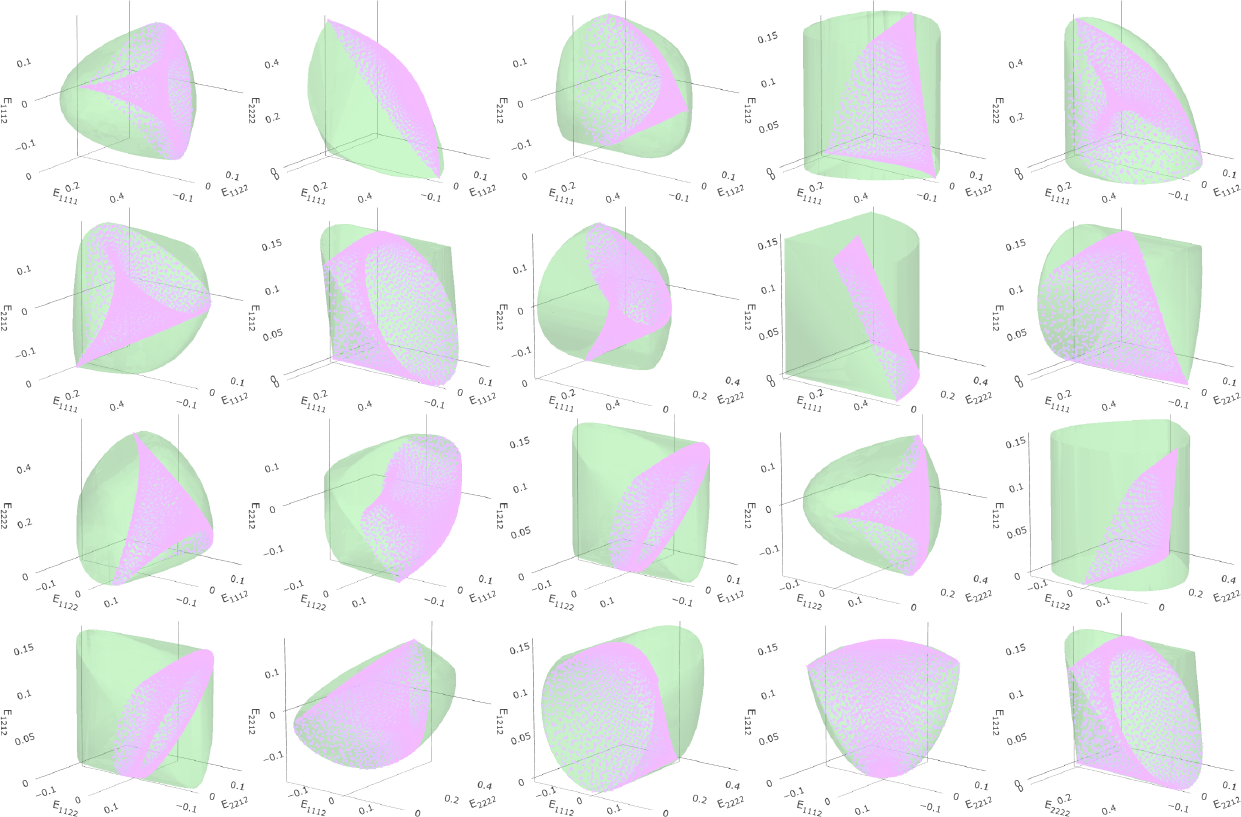}
    \caption{
        Comparison of the admissible set $\mathcal A^{(2)}(0.5)$ (Hashin--Shtrikman, green boundary) and orthotropic sequential-laminate tensors that attain the Hashin--Shtrikman energy bound (pink points). The sets are shown in a selected three-dimensional projection of $\tensf{E}$. For fixed $\revisionA{\rho}=0.5$, the Hashin--Shtrikman boundary was generated using $750$ sampled strains $\strain$ with Frobenius norm $\sqrt{2}/2$. The realizing laminate tensors were generated by sampling $5{,}000$ stresses $\stress$. Material parameters: $\kappa^- = 0.714\times10^{-9}$, $\kappa^+ = 0.714$, $\mu^- = 0.385\times10^{-9}$, and $\mu^+ = 0.385$.
    }
    \label{fig:hashin_laminates}
\end{figure}

Although one-dimensional, the local update $\argmin_{\revisionA{\rho}\in[0,1]} f_{\mathrm{c}}^{\mathrm{HS}}(\stress_i;\revisionA{\rho})+\lambda \revisionA{\rho}$
is not immediate. Regularity information on the HS lower bound $f_{\mathrm{c}}^{\mathrm{HS}}$ is therefore needed.
\begin{prpstn}[Regularity and convexity of the HS complemen\-tary-energy bound]
    For all $\stress\in\mathbb{S}^{2\times 2}$, the function $\revisionA{\rho}\mapsto f_{\mathrm{c}}^{\mathrm{HS}}(\stress;\revisionA{\rho})$ is $C^1$ and convex over $(0,1]$.
\end{prpstn}
\begin{proof}
    See Appendix~\ref{appendix:rederivation_burzain}. The derivative $D_{\revisionA{\rho}}f_{\mathrm{c}}^{\mathrm{HS}}(\stress;\revisionA{\rho})$ follows from basic perturbation theory for optimization problems.
\end{proof}
Accordingly, each local update of each element is computed by bisection on the one-dimensional optimality condition
$D_{\revisionA{\rho}}f_{\mathrm{c}}^{\mathrm{HS}}(\stress;\revisionA{\rho}) = -\lambda$.
Algorithm~\ref{algo:allaire_alt_main} is a descent method: compliance decreases along iterations, so the objective values converge. There is no guarantee of convergence to a global minimizer, but \citet{allaire1997shape} reported low sensitivity to initialization and similar final compliance values.

\section{Results}\label{sec:results}

Building on the theoretical developments of the previous sections, we present numerical examples that illustrate the hierarchy of bounds and evaluate the performance of the sequential global programming algorithm for free-ortho\-tro\-pic material optimization. We compare these results for ZO-FMO (Section \ref{sec:fmo_zero}), V-FMO (Section \ref{sec:fmo_voigt}), and finite-rank laminates computed by alternating minimization (Section \ref{sec:allaire_alg}).

In all examples, we use dimensionless material parameters with strong phase Young’s modulus $E=1$ and Poisson's ratio $\nu=0.3$. The weak phase has the same Poisson ratio and a smaller Young modulus ($10^{-2}$, $10^{-3}$ or $10^{-6}$). \revisionA{Because the elasticity tensors are proportional, we define the dimensionless phase contrast parameter $\chi_{\mathrm E}$ by $\Emin=\chi_{\mathrm E}\Emax$, $\chi_{\mathrm E}\in\{10^{-6},10^{-3},10^{-2}\}$.} The volume constraint is fixed at $\overline{V}=0.2$. The finite‑element discretization is based on Q$8$ serendipity elements. Since the elasticity tensors are piecewise constant at the element level and the element kinematics are sufficiently rich, no additional regularization is applied, in line with the observation of \citet{Jog1994}. 

For ZO-FMO and V-FMO, the resulting linear semidefinite programs \eqref{eq:fmo} and \eqref{eq:fmo_voigtVTS} are solved using the Mosek optimizer with default settings, and we implement elementwise arrow decomposition of the compliance constraints for improved scalability \citep{Kocvara2020}.

\revisionA{
For the sequential global programming algorithm, we stop if the relative change in the merit function
\begin{equation}\label{eq:merit_SGP}
    \sum_{j=1}^{\nlc} c_j\left(\revisionA{\boldsymbol{\mathcal E}}\right) + \lambda\left( \frac{1}{\nel}\sum_{i=1}^{\nel} \sup_{t \in [0,1]} \revisionA{\hat{\rho}}(t;\Ei^{\mathrm{b}}) - \overline{V}\right)
\end{equation}
is below $10^{-6}$. This stopping criterion was active in all HS-FOMO computations reported below. Since the local material problems are solved on a finite sampling grid, the final raw SGP iterate may exhibit a small residual in the global volume constraint. After termination, we therefore apply a continuous post-processing correction to enforce the prescribed volume. A single scalar parameter is used to move the interior material tensors towards the weak or strong phase, as required by the sign of the volume residual, while their orientations are kept fixed and tensors already at the phase bounds remain unchanged. The scalar parameter is determined using the continuous HS volume estimator such that the global volume constraint is satisfied to numerical precision. The compliance values and material fields reported below correspond to these projected designs; the reported volume multipliers are those of the final SGP iterates.}

\revisionA{For the SGP computations reported below, unless stated otherwise, all element tensors are initialized with the strong-phase tensor $\Emax$. As a robustness check, we repeated the $60\times 30$ multi-loadcase example with $\Emin=10^{-3}\Emax$ using three alternative initializations: an isotropic tensor and horizontally and vertically oriented orthotropic tensors. The final compliance values for the default, isotropic, horizontal, and vertical initializations were $31.696$, $31.684$, $31.690$, and $31.674$, respectively, corresponding to a total spread below $0.07\%$. The resulting material distributions and effective elasticity fields were also very similar; relative to the default initialization, the median elementwise relative difference of the elasticity tensors expressed in the global coordinate system remained below $1\%$ for all alternative initializations. Additional tests for the other examples considered in this paper led to the same observation. Thus, for the numerical examples studied in this article, the SGP method exhibits only low sensitivity to the choice of the initial point.}

\revisionA{All SGP timings reported below were measured on a 2023 Apple Mac Studio equipped with an Apple M2 Ultra processor (ARM64 architecture) and $192$\,GB of memory, using MATLAB R2026a Update~2 with a parallel pool of $12$ workers. The reported wall-clock times refer to the online SGP stage and exclude loading of the precomputed sampling data and visualization. For each set of phase parameters, the sampling data are generated only once and can subsequently be reused. At the sampling resolution employed throughout this work, one such offline precomputation required approximately $40$\,s.}

In the alternating‑minimization scheme of \citet{allaire1997shape}, orthotropic finite‑rank laminates are designed elementwise from average stresses \citep{Jog1994}; the volume fraction update follows the compliance optimality condition, and we stop when (i) the maximum density change is below $10^{-5}$ and (ii) the relative change \revisionA{in the sum of local complementary-energy values} is below $10^{-8}$. The method is initialized with the strong material $\Emax$. 

\subsection{Cantilever Problem}

\begin{figure}
    \centering
    \includegraphics{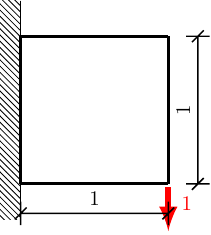}
    \caption{Cantilever beam: geometry and boundary conditions}
    \label{fig:cantilever_bc}
\end{figure}

As the first example, we consider a cantilever beam of the width $1$ and height $1$, clamped at the left edge and subjected to a vertical load of magnitude $1$ at the bottom right corner, as shown in Fig.~\ref{fig:cantilever_bc}. The beam is discretized using $30\times30$ serendipity elements. We solve the optimization problem using the different bounds for three different values of the weak phase Young's modulus: $10^{-2}$, $10^{-3}$, and $10^{-6}$. 

\subsubsection{Phase Contrast \texorpdfstring{$10^{-6}$}{10e-6}}

\begin{figure}[!b]
\centering
\includegraphics[width=\linewidth]{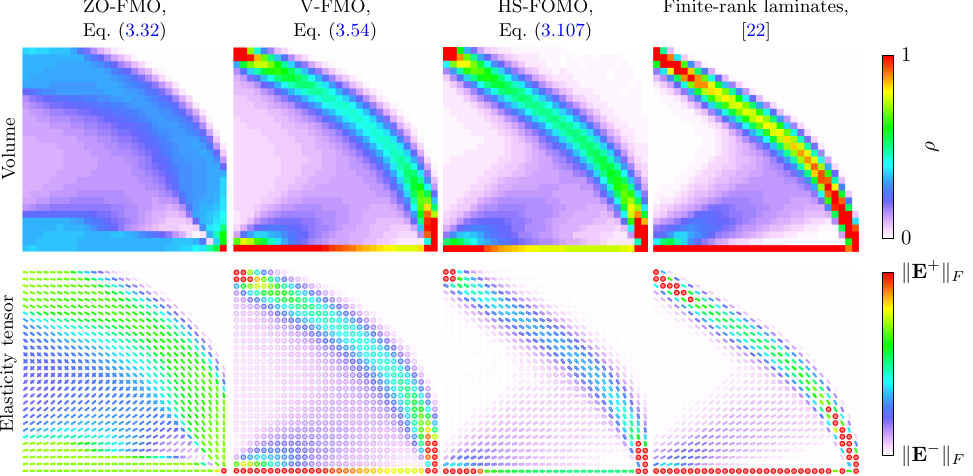}
\caption{Cantilever beam optimization for $\Emin = 10^{-6}\Emax$: volume fraction (top row) and elasticity tensor rosettes (bottom row) for ZO-FMO, V-FMO, HS‑FOMO, and sequential laminate designs.}
\label{fig:cantilever_1e-06}
\end{figure}

For the $10^{-6}$ case, the ZO-FMO \eqref{eq:fmo} yields a compliance lower bound of $18.978$ with a volume-constraint multiplier $\lambda = 15.038$. The corresponding material distribution and elasticity tensors are shown in the first column of Fig.~\ref{fig:cantilever_1e-06}. The tensors are visualized by rosette plots, i.e., polar plots of the directional strain energy $\langle \tensf{E}\strain,\strain\rangle$ for $\lVert\strain\rVert_F = 1$. As seen from these plots, the tensors are nearly orthotropic, although orthotropy is not enforced in this formulation.

The V-FMO \eqref{eq:fmo_voigtVTS} gives a substantially tighter compliance lower bound of $39.843$, and the volume multiplier increases to $\lambda = 165.439$. The increase in both compliance and $\lambda$ reflects the tighter coupling between elasticity tensors and element-wise volume estimates; consequently, a larger $\lambda$ is required to satisfy the volume constraint. This is also reflected in the material distribution: the resulting isotropic tensors are more localized than in the ZO-FMO case (compare first and second columns of Fig.~\ref{fig:cantilever_1e-06}).

Solving the HS-FOMO problem \eqref{eq:fmo_hashin_ortho_nested} with the sequential global programming algorithm yields a compliance of \revisionA{$43.155$} with \revisionA{$\lambda = 192.772$}, so the objective increases by only $8.31\%$ relative to V-FMO. The multiplier $\lambda$ likewise increases only moderately. The relative-volume distribution is very similar to the V-FMO case (compare the second and third columns of Fig.~\ref{fig:cantilever_1e-06}), but the tensors are now orthotropic and have distinguishably smaller norms; the V-FMO tensors would require more material volume to be realizable.

Finally, we compare the HS-FOMO compliance with the result of alternating minimization using sequential laminates. \revisionA{The sequential-laminate tensors are realizable in the homogenization sense and satisfy the HS constraints. Their compliance therefore provides an upper bound on the global optimum of the corresponding discrete HS-FOMO problem when evaluated using the same finite-element model.} The resulting compliance is $\revisionA{43.658}$, i.e., \revisionA{$1.17\%$} higher than in the HS-FOMO setting, with $\lambda = \revisionA{174.162}$. In this case, alternating minimization converged very slowly due to the high contrast between the strong and weak phases, requiring $\revisionA{217{,}055}$ iterations. Comparing the third and fourth columns of Fig.~\ref{fig:cantilever_1e-06}, the finite-rank laminate design shows a more localized volume distribution, likely due to a narrower admissible class of elasticity tensors. The higher multiplier in the HS-FOMO solution can be interpreted as a larger shadow price of volume: because HS-FOMO resolves the full finite-element response (including stress peaks), it predicts a stronger marginal compliance reduction from additional material, whereas stress averaging in alternating minimization smooths these peaks and yields a smaller effective marginal value, hence a lower $\lambda$.

\begin{figure}[!b]
\centering
\includegraphics[width=\linewidth]{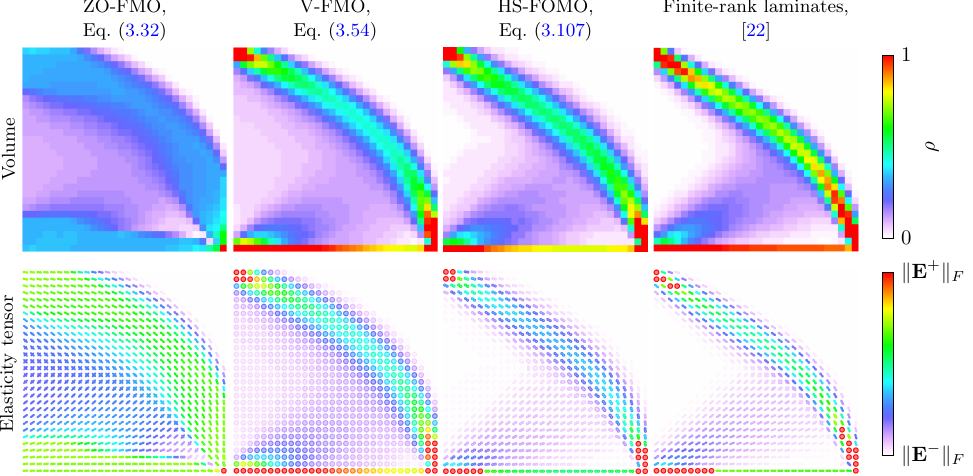}
\caption{Cantilever beam optimization for $\Emin = 10^{-3}\Emax$: volume fraction (top row) and elasticity tensor rosettes (bottom row) for ZO-FMO, V-FMO, HS‑FOMO, and sequential laminate designs.}
\label{fig:cantilever_1e-03}
\end{figure}

\subsubsection{Phase Contrast \texorpdfstring{$10^{-3}$}{10e-3}}

For the phase-contrast setting $10^{-3}$, the same qualitative picture as in the $10^{-6}$ case is preserved, so we focus only on what changes. All compliance values are slightly lower, which is expected for the milder contrast between phases: ZO-FMO gives $18.954$, V-FMO gives $39.721$, HS-FOMO gives \revisionA{$42.451$}, and sequential laminates give $42.968$. Compared with the $10^{-6}$ case, the relative gap between V-FMO and HS-FOMO is smaller (about \revisionA{$6.87\%$} instead of $8.31\%$), suggesting that for $10^{-3}$ contrast V-FMO already captures more of the attainable improvement.

The HS-FOMO and sequential-laminate results remain very close, with only a $\revisionA{1.22\%}$ difference in compliance, so the near-tight agreement observed before still holds. The multipliers follow the same pattern as in the high-contrast case: $\lambda$ increases from ZO-FMO to V-FMO and then to HS-FOMO, while the alternating minimization value stays below HS-FOMO ($14.937$, $164.027$, \revisionA{$184.854$}, and $173.626$, respectively). This is consistent with the same interpretation as before: HS-FOMO sees the full element-wise finite-element response, including local stress peaks, whereas alternating minimization with sequential laminates uses stress averaging, which smooths sensitivities and leads to a lower effective shadow price of volume.

For the $10^{-3}$ contrast, the visual trends in Fig.~\ref{fig:cantilever_1e-03} are largely consistent with the $10^{-6}$ case in Fig.~\ref{fig:cantilever_1e-06}. The ZO-FMO model again produces a diffused distribution, while the V-FMO, HS-FOMO, and finite-rank laminate solutions concentrate material along the same main load-carrying arc. As before, the V-FMO column shows isotropic tensors, whereas HS-FOMO and sequential laminates recover orthotropic tensors with very similar orientation fields. Compared with $10^{-6}$, the $10^{-3}$ designs appear slightly less extreme (smoother transitions and marginally weaker localization), which is consistent with the milder phase contrast and \revisionA{the smaller V-FMO-to-HS-FOMO gap} reported above.

\begin{figure}[!b]
\centering
\includegraphics[width=\linewidth]{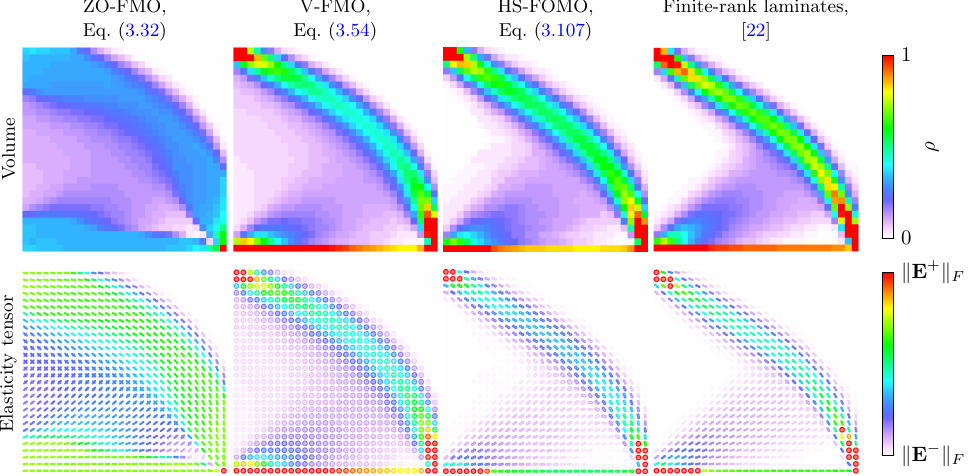}
\caption{Cantilever beam optimization for $\Emin = 10^{-2}\Emax$: volume fraction (top row) and elasticity tensor rosettes (bottom row) for ZO-FMO, V-FMO, HS‑FOMO, and sequential laminate designs.}
\label{fig:cantilever_1e-02}
\end{figure}

\subsubsection{Phase Contrast \texorpdfstring{$10^{-2}$}{10e-2}}

\begin{figure}[!t]
\centering
\includegraphics[width=\linewidth]{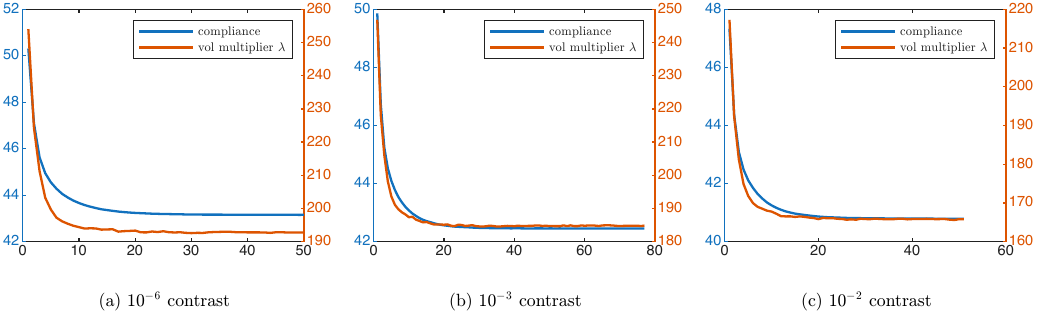}
\caption{Convergence of the sequential global programming algorithm for the HS-FOMO problem across different phase contrasts. The blue curve shows compliance and the orange curve shows the volume multiplier $\lambda$.}
\label{fig:cantilever_sgpfmo_convergence}
\end{figure}

For the $10^{-2}$ phase contrast in Fig.~\ref{fig:cantilever_1e-02}, the same qualitative ordering is preserved as in Figs.~\ref{fig:cantilever_1e-06} and~\ref{fig:cantilever_1e-03}, but the differences between models become smaller. The compliance values are $18.827$ (ZO-FMO), $38.675$ (V-FMO), \revisionA{$40.789$} (HS-FOMO), and $41.106$ (finite-rank laminates), so all values decrease further as the contrast weakens. The V-FMO-to-HS-FOMO gap drops to \revisionA{$5.47\%$} (from \revisionA{$6.87\%$} at $10^{-3}$ and $8.31\%$ at $10^{-6}$), and the laminate-to-HS-FOMO gap drops to \revisionA{$0.78\%$} (from \revisionA{$1.22\%$ and $1.17\%$}).

The visual comparison supports this trend: the main load-carrying arc remains in the same location, but the distributions are slightly smoother and less sharply localized than in the $10^{-6}$ and $10^{-3}$ cases. As before, the V-FMO column remains isotropic, while HS-FOMO and sequential laminates show closely aligned orthotropic orientation fields. The multipliers also decrease ($14.527$, $152.165$, \revisionA{$165.808$}, and $158.800$, respectively), while retaining the same pattern, including a higher $\lambda$ for HS-FOMO than for sequential laminates.

\begin{table}[!b]
\centering
\small
\setlength{\tabcolsep}{6pt}
\begin{tabular}{l l l r r}
\hline
Contrast \revisionA{$\chi_{\mathrm E}$} & Model & Problem & Compliance & $\lambda$ \\
\hline
$10^{-6}$ & ZO-FMO & Eq.~\eqref{eq:fmo} & $18.978$ & $15.038$ \\
 & V-FMO & Eq.~\eqref{eq:fmo_voigtVTS} & $39.843$ & $165.439$ \\
 & HS-FOMO & Eq.~\eqref{eq:fmo_hashin_ortho_nested} & \revisionA{$43.155$} & \revisionA{$192.772$} \\
 & Finite-rank laminates & Sec.~\ref{sec:allaire_alg} & $\revisionA{43.658}$ & $\revisionA{174.162}$ \\
\hline
$10^{-3}$ & ZO-FMO & Eq.~\eqref{eq:fmo} & $18.954$ & $14.937$ \\
 & V-FMO & Eq.~\eqref{eq:fmo_voigtVTS} & $39.721$ & $164.027$ \\
 & HS-FOMO & Eq.~\eqref{eq:fmo_hashin_ortho_nested} & \revisionA{$42.451$} & \revisionA{$184.854$} \\
 & Finite-rank laminates & Sec.~\ref{sec:allaire_alg} & $42.968$ & $173.626$ \\
\hline
$10^{-2}$ & ZO-FMO & Eq.~\eqref{eq:fmo} & $18.827$ & $14.527$ \\
 & V-FMO & Eq.~\eqref{eq:fmo_voigtVTS} & $38.675$ & $152.165$ \\
 & HS-FOMO & Eq.~\eqref{eq:fmo_hashin_ortho_nested} & \revisionA{$40.789$} & \revisionA{$165.808$} \\
 & Finite-rank laminates & Sec.~\ref{sec:allaire_alg} & $41.106$ & $158.800$ \\
\hline
\end{tabular}
\caption{Cantilever problem results across phase contrasts and different formulations.}
\label{tab:cantilever_phase_contrast_summary}
\end{table}

\revisionA{For completeness, Table~\ref{tab:cantilever_phase_contrast_summary} summarizes the compliance and volume multipliers corresponding to the designs shown in Figs.~\ref{fig:cantilever_1e-06}--\ref{fig:cantilever_1e-02}. As the SGP method is less known compared to the other approaches, we present convergence plots in Fig.~\ref{fig:cantilever_sgpfmo_convergence}. The objective values decrease almost monotonically, while the corresponding volume multipliers stabilize as the iterations progress. For phase contrasts $\chi_\mathrm{E}=10^{-6}$, $10^{-3}$, and $10^{-2}$, the algorithm required $50$, $77$, and $51$ SGP updates, respectively. The corresponding online wall-clock times were $1158$\,s, $1759$\,s, and $1170$\,s, i.e., approximately $19.3$, $29.3$, and $19.5$ minutes. This corresponds to an average of approximately $23.2$\,s, $22.8$\,s, and $22.9$\,s per SGP update, respectively. For comparison, the alternating-minimization optimization loop required approximately $77{,}606$\,s, $242$\,s, and $72$\,s for $\chi_{\mathrm E}=10^{-6}$, $10^{-3}$, and $10^{-2}$, respectively. These timings correspond to the respective stopping criteria stated above and show that the runtime of alternating minimization varies strongly with phase contrast.}

\subsection{Multi-Loadcase Problem}

As the second illustration, we consider a multi-loadcase compliance problem on a rectangular domain of width $2$ and height $1$, see Fig.~\ref{fig:multi-loadcase}. The structure is supported at the four corner points, and two independent unit point loads are applied: a downward load at the midpoint of the top edge (first loadcase) and an upward load at the midpoint of the bottom edge (second loadcase). In contrast to the cantilever problem, the objective is now the sum of compliances, $c_1+c_2$, under the same global volume constraint $0.2$. We report results for the phase contrasts \revisionA{$\chi_{\mathrm E} = 10^{-6}$}, $10^{-3}$, and $10^{-2}$ on $40\times20$ and $60\times30$ discretizations, and compare the ZO-FMO, V-FMO, and HS-FOMO formulations. \revisionA{These mesh resolutions refer to the full domain on which equilibrium is solved for both loadcases. In the SGP material subproblem, reflection symmetry about both centerlines reduces the number of independent element updates by a factor of four.} The alternating-minimization results are not included in this example because its multi-loadcase extension \citep{Allaire1996} does not enforce orthotropic microstructures. \revisionA{The comparison therefore examines how the imposed energy bounds affect the computed compliance values and how the SGP algorithm converges under multiple loadcases. Realizability of the resulting tensors by composites of the prescribed phases is not established.}

\begin{figure}[!t]
    \centering
    \includegraphics{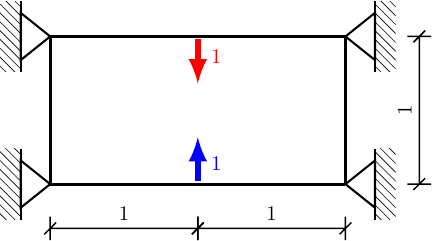}
    \caption{Multi-loadcase problem: geometry and boundary conditions}
    \label{fig:multi-loadcase}
\end{figure}

\subsubsection{Quantitative Overview}

Table~\ref{tab:multi_loadcase_summary} summarizes the quantitative comparison of ZO-FMO, V-FMO, and HS-FOMO formulations for all considered contrasts and both discretizations. Across all cases, the compliance values follow the consistent ordering $c^{\mathrm{ZO\text{-}FMO}} < c^{\mathrm{V\text{-}FMO}} < c^{\mathrm{HS\text{-}FOMO}}$, with a large increase from ZO-FMO to V-FMO and a smaller, but still systematic, increase from V-FMO to HS-FOMO.

The volume multipliers show the same hierarchy, $\lambda^{\mathrm{ZO\text{-}FMO}} < \lambda^{\mathrm{V\text{-}FMO}} < \lambda^{\mathrm{HS\text{-}FOMO}}$, indicating progressively tighter coupling between stiffness admissibility and effective material usage from ZO-FMO to V-FMO and further to HS-FOMO. Detailed trend interpretations are discussed separately in Section~\ref{sec:multi_loadcase_phase_contrast} (phase-contrast effects) and Section~\ref{sec:multi_loadcase_discretization} (discretization and convergence effects).

\begin{table}[!b]
\centering
\small
\setlength{\tabcolsep}{5pt}
\begin{tabular}{l l l l r r}
\hline
Discretization & Contrast \revisionA{$\chi_{\mathrm E}$} & Model & Problem & Compliance & $\lambda$ \\
\hline
$40\times20$ & $10^{-6}$ & ZO-FMO & Eq.~\eqref{eq:fmo} & $14.213$ & $13.743$ \\
 &  & V-FMO & Eq.~\eqref{eq:fmo_voigtVTS} & $28.459$ & $123.643$ \\
 &  & HS-FOMO & Eq.~\eqref{eq:fmo_hashin_ortho_nested} & \revisionA{$32.015$} & \revisionA{$153.061$} \\
\hline
$60\times30$ & $10^{-6}$ & ZO-FMO & Eq.~\eqref{eq:fmo} & $14.710$ & $13.101$ \\
 &  & V-FMO & Eq.~\eqref{eq:fmo_voigtVTS} & $28.992$ & $121.908$ \\
 &  & HS-FOMO & Eq.~\eqref{eq:fmo_hashin_ortho_nested} & $32.256$ & \revisionA{$150.536$} \\
\hline
$40\times20$ & $10^{-3}$ & ZO-FMO & Eq.~\eqref{eq:fmo} & $14.189$ & $13.652$ \\
 &  & V-FMO & Eq.~\eqref{eq:fmo_voigtVTS} & $28.361$ & $122.476$ \\
 &  & HS-FOMO & Eq.~\eqref{eq:fmo_hashin_ortho_nested} & \revisionA{$31.496$} & \revisionA{$146.948$} \\
\hline
$60\times30$ & $10^{-3}$ & ZO-FMO & Eq.~\eqref{eq:fmo} & $14.687$ & $13.012$ \\
 &  & V-FMO & Eq.~\eqref{eq:fmo_voigtVTS} & $28.897$ & $120.811$ \\
 &  & HS-FOMO & Eq.~\eqref{eq:fmo_hashin_ortho_nested} & \revisionA{$31.696$} & \revisionA{$144.301$} \\
\hline
$40\times20$ & $10^{-2}$ & ZO-FMO & Eq.~\eqref{eq:fmo} & $14.007$ & $13.006$ \\
 &  & V-FMO & Eq.~\eqref{eq:fmo_voigtVTS} & $27.534$ & $112.883$ \\
 &  & HS-FOMO & Eq.~\eqref{eq:fmo_hashin_ortho_nested} & \revisionA{$30.110$} & \revisionA{$131.553$} \\
\hline
$60\times30$ & $10^{-2}$ & ZO-FMO & Eq.~\eqref{eq:fmo} & $14.524$ & $12.419$ \\
 &  & V-FMO & Eq.~\eqref{eq:fmo_voigtVTS} & $28.092$ & $111.756$ \\
 &  & HS-FOMO & Eq.~\eqref{eq:fmo_hashin_ortho_nested} & \revisionA{$30.361$} & \revisionA{$129.205$} \\
\hline
\end{tabular}
\caption{Multi-loadcase problem results across phase contrasts, discretizations, and formulations.}
\label{tab:multi_loadcase_summary}
\end{table}

\subsubsection{Effect of Phase Contrast}\label{sec:multi_loadcase_phase_contrast}

\begin{figure}[!t]
\centering
\newlength{\mlside}
\includegraphics[width=\linewidth]{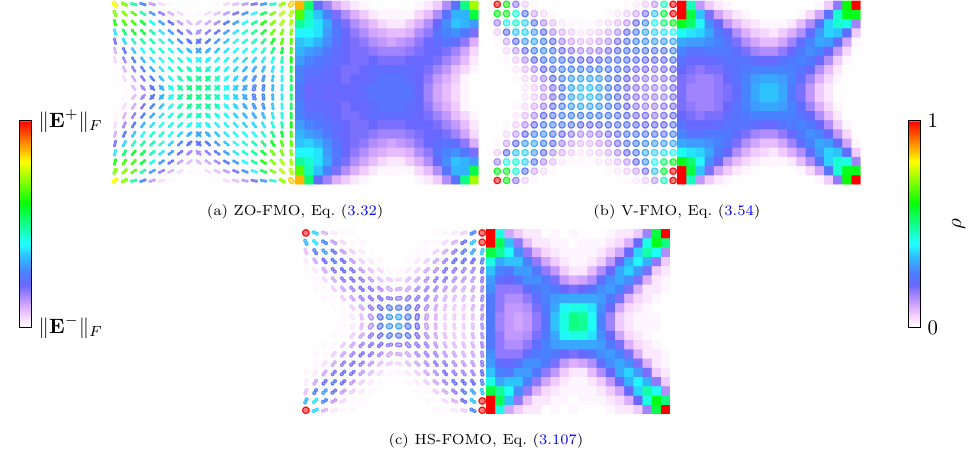}
\caption{Multi‑loadcase optimization for discretization $40\times20$ and $\Emin = 10^{-6}\Emax$. In each image, the left half shows the elasticity‑tensor rosettes and the right half shows the volume fraction $\revisionA{\rho}$.}
\label{fig:ml_40_1e-6}
\end{figure}

To isolate the role of phase contrast, we compare results at fixed mesh for $\revisionA{\chi_\mathrm{E}}\in\{10^{-6},10^{-3},10^{-2}\}$ using Table~\ref{tab:multi_loadcase_summary} and the field plots in Figs.~\ref{fig:ml_40_1e-6}--\ref{fig:ml_60_1e-2}. For both discretizations, all compliance values decrease monotonically as the contrast weakens: on $40\times20$, from $14.213\to14.189\to14.007$ (ZO-FMO), $28.459\to28.361\to27.534$ (V-FMO), and \revisionA{$32.015\to31.496\to30.110$} (HS-FOMO); on $60\times30$, from $14.710\to14.687\to14.524$, $28.992\to28.897\to28.092$, and \revisionA{$32.256\to31.696\to30.361$}, respectively.

\begin{figure}[!b]
\centering
\includegraphics[width=\linewidth]{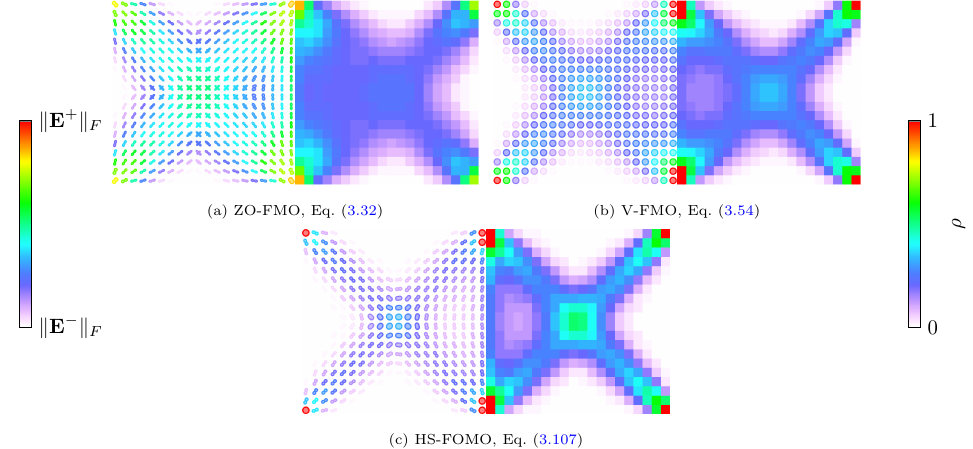}
\caption{Multi‑loadcase optimization for discretization $40\times20$ and $\Emin = 10^{-3}\Emax$. In each image, the left half shows the elasticity‑tensor rosettes and the right half shows the volume fraction $\revisionA{\rho}$.}
\label{fig:ml_40_1e-3}
\end{figure}

\begin{figure}[!t]
\centering
\includegraphics[width=\linewidth]{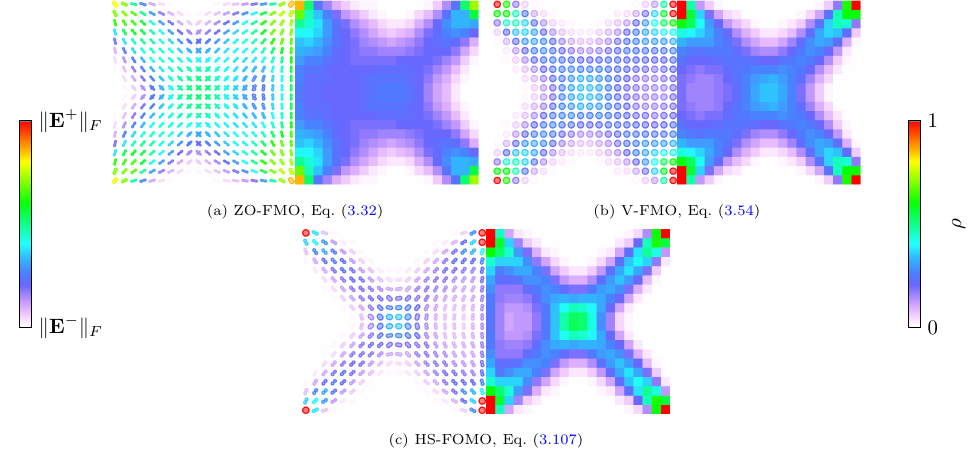}
\caption{Multi‑loadcase optimization for discretization $40\times20$ and $\Emin = 10^{-2}\Emax$. In each image, the left half shows the elasticity‑tensor rosettes and the right half shows the volume fraction $\revisionA{\rho}$.}
\label{fig:ml_40_1e-2}
\end{figure}

\begin{figure}[!b]
\centering
\includegraphics[width=\linewidth]{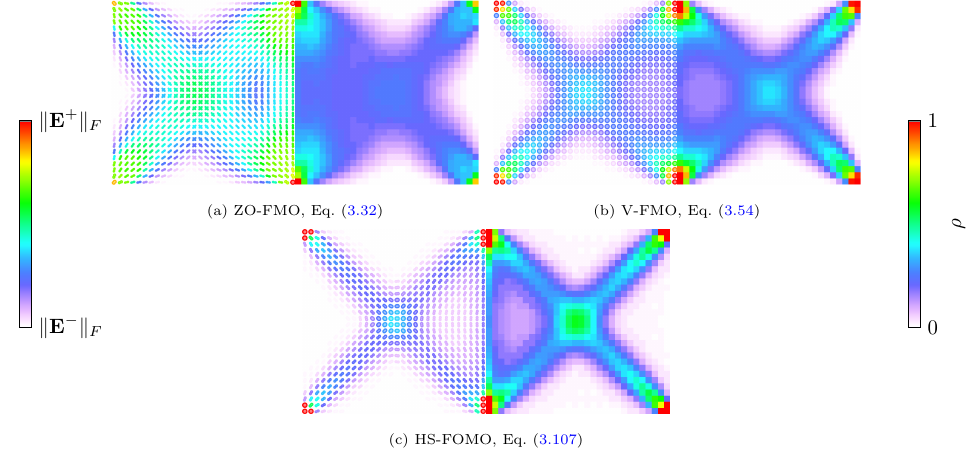}
\caption{Multi‑loadcase optimization for discretization $60\times30$ and $\Emin = 10^{-6}\Emax$. In each image, the left half shows the elasticity‑tensor rosettes and the right half shows the volume fraction $\revisionA{\rho}$.}
\label{fig:ml_60_1e-6}
\end{figure}

\begin{figure}[!t]
\centering
\includegraphics[width=\linewidth]{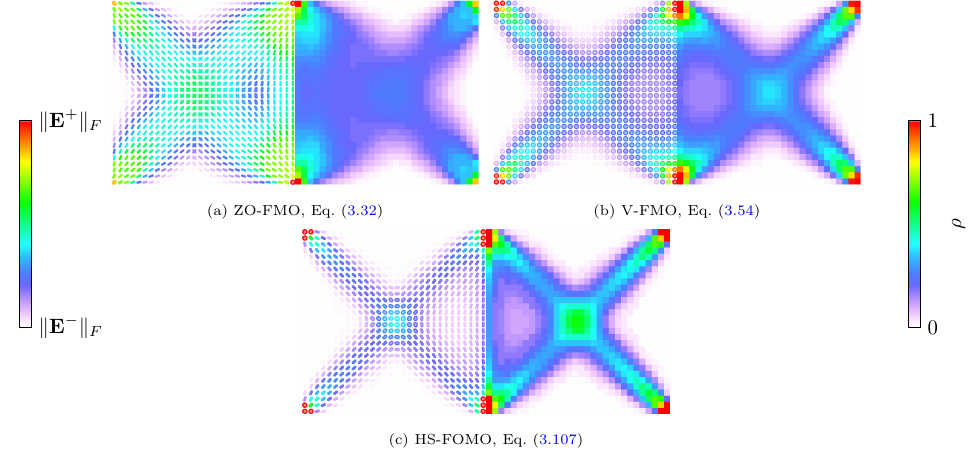}
\caption{Multi‑loadcase optimization for discretization $60\times30$ and $\Emin = 10^{-3}\Emax$. In each image, the left half shows the elasticity‑tensor rosettes and the right half shows the volume fraction $\revisionA{\rho}$.}
\label{fig:ml_60_1e-3}
\end{figure}

\begin{figure}[!b]
\centering
\includegraphics[width=\linewidth]{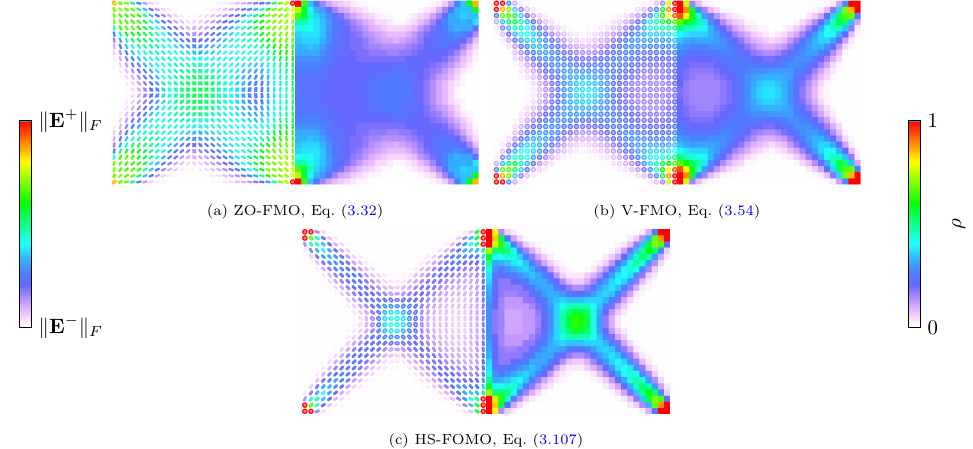}
\caption{Multi‑loadcase optimization for discretization $60\times30$ and $\Emin = 10^{-2}\Emax$. In each image, the left half shows the elasticity‑tensor rosettes and the right half shows the volume fraction $\revisionA{\rho}$.}
\label{fig:ml_60_1e-2}
\end{figure}

The magnitude of this decrease is model-dependent. From $10^{-6}$ to $10^{-2}$, the reduction is largest for HS-FOMO (about $6\%$), intermediate for V-FMO (about $3\%$), and smallest for ZO-FMO (\revisionA{about $1.3$--$1.5\%$}). This is consistent with model fidelity: as $\Emin$ increases, the weak phase carries load more effectively, and tighter formulations capture this gain more directly, whereas ZO-FMO is comparatively less sensitive. The same mechanism explains the shrinking V-FMO-to-HS-FOMO gap with weaker contrast (from \revisionA{$12.50\%\to11.05\%\to9.36\%$} on $40\times20$ and from \revisionA{$11.26\%\to9.69\%\to8.08\%$} on $60\times30$), because the Hashin--Shtrikman correction terms are driven by phase-stiffness contrast.

The figures indicate that phase contrast affects the amplitude of localization more than the topology of load-transfer paths. Here, amplitude refers to how pronounced the high-value regions are in the volume field $\revisionA{\rho}$ and in the rosette-magnitude colors (i.e., how close local values are to the $\lVert\tensf{E}^+\rVert_F$ end of the scale). For each model, the dominant load-transfer pattern is preserved across contrasts: ZO-FMO remains comparatively diffused, V-FMO solutions retain isotropic tensors, as confirmed by circular rosettes, and HS-FOMO preserves the orthotropic pattern with similar orientation fields. As contrast weakens, peaks near loads, supports and along the main transfer paths are reduced, and the volume field becomes slightly more evenly distributed, reflecting reduced stiffness mismatch and, consequently, reduced need for strongly concentrated stiff-phase routing.

The multipliers in Table~\ref{tab:multi_loadcase_summary} also decrease monotonically with weaker contrast for all formulations. This follows the same physical interpretation: when the weak phase is less compliant (i.e., $\Emin$ increases), the marginal compliance gain from additional stiff material decreases, and the shadow price of the volume constraint is lower.

\subsubsection{Effect of Discretization and Convergence}\label{sec:multi_loadcase_discretization}

To assess discretization effects, we compare the two meshes at fixed phase contrast using Table~\ref{tab:multi_loadcase_summary}. Defining
\begin{equation}
\Delta_h c := \frac{c_{60\times30}-c_{40\times20}}{c_{40\times20}},
\end{equation}
the compliance increase under refinement is moderate: about $3.5$--$3.7\%$ for ZO-FMO, $1.9$--$2.0\%$ for V-FMO, and $0.6$--$0.9\%$ for HS-FOMO.

\begin{figure}[!t]
\centering
\includegraphics[width=\linewidth]{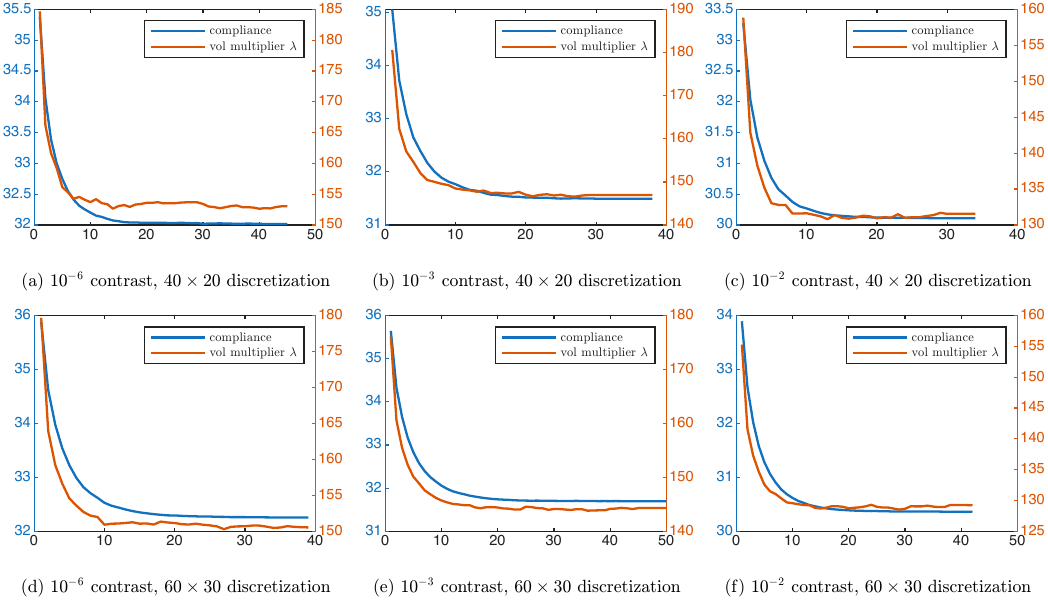}
\caption{Convergence of the sequential global programming algorithm for the HS-FOMO problem across different phase contrasts. The blue curve shows compliance and the orange curve shows the volume multiplier $\lambda$.}
\label{fig:sgpfmo_convergence_multi_loadcase}
\end{figure}

The sign of this change is explained by two competing mechanisms. On one hand, finite element modeling correction tends to increase compliance, because coarse displacement-based meshes are typically over-stiff and therefore underestimate the objective. On the other hand, discretization refinement enlarges the admissible design space by allowing finer spatial variation of $\Ei$, which tends to reduce the objective. In the present results, the former effect dominates. This tendency is further amplified by point loads and point supports, since refinement resolves stress concentrations more sharply and reduces artificial coarse-mesh stiffening.

The refinement trend is also consistent in the relative gap between V-FMO and HS-FOMO compliances. Moving from $40\times20$ to $60\times30$, this gap decreases from $12.50\%$ to $11.26\%$ at contrast $10^{-6}$, from $\revisionA{11.05\%}$ to $\revisionA{9.69\%}$ at $10^{-3}$, and from $\revisionA{9.36\%}$ to $\revisionA{8.08\%}$ at $10^{-2}$, i.e., by about $\revisionA{1.2}$--$1.4$ percentage points in all three cases. The multipliers show the same refinement direction: they are reduced by approximately $4.5$--$4.7\%$ for ZO-FMO, $1.0$--$1.4\%$ for V-FMO, and \revisionA{$1.6$--$1.8\%$} for HS-FOMO, consistent with a lower marginal value of additional stiff material on the refined discretization.

\revisionA{The field plots in Figs.~\ref{fig:ml_40_1e-6}--\ref{fig:ml_60_1e-2} indicate robust topology under mesh refinement: the same global transfer paths are preserved, while the $60\times30$ solutions mainly sharpen interfaces and load paths. The convergence curves in Fig.~\ref{fig:sgpfmo_convergence_multi_loadcase} show a similar two-stage behavior for all HS-FOMO cases: a rapid initial decrease in compliance is followed by a slower final phase, while the volume multiplier approaches its final value with small oscillations.}

\revisionA{The small oscillations of $\lambda$ are a consequence of the discrete sampling used for the elementwise material subproblems. For fixed $\lambda$, the selected element tensors can change only between sampled candidates, so that the resulting global volume is effectively a stepwise function of $\lambda$. Consequently, the prescribed volume equality cannot in general be attained exactly on the sampling grid: near the solution, the bisection may alternate between an iterate with a slight volume excess and one with a slight volume deficit. This discrete effect leads to the observed small oscillations of the corresponding volume multiplier. Apart from these sampling-induced oscillations, the convergence profiles are qualitatively similar across all considered phase contrasts.}

\revisionA{For the $40\times20$ discretization, the phase contrasts $\chi_\mathrm{E}=10^{-6},10^{-3}$, and $10^{-2}$ required $45$, $38$, and $34$ SGP updates, respectively, with online wall-clock times of $258$\,s, $233$\,s, and $202$\,s. This corresponds to approximately $5.7$\,s, $6.1$\,s, and $5.9$\,s per update. For the $60\times30$ discretization, the corresponding numbers were $39$, $50$, and $42$ updates with wall-clock times of $507$\,s, $588$\,s, and $471$\,s, or approximately $13.0$\,s, $11.8$\,s, and $11.2$\,s per update.}

\section{Conclusions}\label{sec:summary}

This paper considers the two-dimensional plane-stress setting with well-ordered isotropic weak and strong phases. In this setting, we establish a clear hierarchy of free material optimization models based on increasingly tight energy bounds, from zeroth-order to Voigt and then to Hashin--Shtrikman. The continuum problem \eqref{eq:fmo_continuum} is used as a theoretical reference that organizes this hierarchy through \eqref{eq:hierarchy_bounds} and supports existence statements. Within the convex formulations considered here, the Voigt model is the strongest relaxation. It is strictly tighter than the zeroth-order model for $0<\overline{V}<1$ (Proposition~\ref{prop:voigt_inclusion}) and coincides with it at $\overline{V}\in\{0,1\}$. It also admits an optimal isotropic representation, which reduces the problem to variable-thickness-sheet optimization with one scalar variable per element (Proposition~\ref{prop:voigt_iso} and \eqref{eq:fmo_voigtVTS}). In addition, for the single-loadcase continuum zeroth-order model, at least one optimal solution can be chosen orthotropic (Proposition~\ref{prop:zo_orthotropic_representative}). For continuum well-posedness, minimizers exist directly for the zeroth-order and Voigt models, while for Hashin--Shtrikman bounds we obtain existence in the $H$-relaxed sense (Section~\ref{sec:cont_fmo}).

For the Hashin--Shtrikman model, the results clarify both its strength and its algorithmic implications. The HS correction is nonnegative and vanishes at endpoints (Proposition~\ref{prop:q_ge_0}). For every interior volume fraction $\revisionA{\rho_i}\in(0,1)$, there exist strains for which this correction is strictly positive (while equality can still occur for a specific invariant relation), giving strict interior improvement over Voigt and leading to nonconvexity of the feasible set in $(\revisionA{\rho_i},\Ei)$ (Lemma~\ref{lem:nonconvex}). At the same time, its convex hull is exactly the Voigt feasible set (Lemma~\ref{lem:HS_convex_hull}), which explains why no tighter convex alternative is available in the same setting. Geometric visualizations of the admissible sets (Figs.~\ref{fig:zero_voigt}, \ref{fig:voigt_hashin}, and \ref{fig:voigt_hashin_vol}) make these relations directly visible. Separately, we show that the local HS volume variables can be eliminated through active-constraint characterization and explicit activating-volume formulas (Lemma~\ref{lem:HS_active}, Proposition~\ref{prop:explicit_volume}), giving the reduced model \eqref{eq:fmo_hashin_reduced}. \revisionA{A key by-product is an explicit volume estimator that quantifies the minimum stiff-phase volume fraction compatible with the imposed HS upper energy bounds through a worst-case strain evaluation. This provides a lower bound on the stiff-phase volume fraction required for realizability.} For orthotropic tensors this further reduces to a one-parameter search in $t=\lvert\Tr(\strain)\rvert$ (Proposition~\ref{prop:trace_only_reduction} and \eqref{eq:fmo_hashin_ortho}).

To connect the discrete HS-FMO model with classical homogenization theory, we additionally analyze the single-loadcase continuum \revisionA{problem}: in that setting, the HS-FMO relaxation is tight and has the same optimal value as the Allaire--Kohn relaxed formulation (Theorem~\ref{th:tight}). \revisionA{This common optimal value is attained by a homogenized design whose effective tensors are orthotropic sequential laminates.} The corresponding laminate-construction viewpoint is illustrated in Fig.~\ref{fig:hashin_laminates}. The tightness statement should be interpreted as a continuum relaxation result, not as exact equivalence on a fixed finite-element mesh. For multiple loadcases, \revisionA{the imposed HS bounds remain necessary but generally insufficient for realizability.}

Computationally, the nonconvex orthotropic HS problem is handled with sequential global programming (Section~\ref{sec:sgp}). Numerically, both benchmarks confirm the predicted ordering $c^{\mathrm{ZO\text{-}FMO}}<c^{\mathrm{V\text{-}FMO}}<c^{\mathrm{HS\text{-}FOMO}}$ (Tables~\ref{tab:cantilever_phase_contrast_summary} and \ref{tab:multi_loadcase_summary}). In the cantilever test, HS-FOMO remains close to finite-rank laminate solutions (compliance gap about $0.78$--\revisionA{$1.22\%$}), while in the multi-loadcase test the ordering $c^{\mathrm{ZO\text{-}FMO}}<c^{\mathrm{V\text{-}FMO}}<c^{\mathrm{HS\text{-}FOMO}}$ is preserved across phase contrasts and mesh refinements, with stable SGP empirical convergence behavior (Figs.~\ref{fig:cantilever_sgpfmo_convergence} and \ref{fig:sgpfmo_convergence_multi_loadcase}). \revisionA{Across the considered phase contrasts, weakening the contrast reduces the localization of the material distributions and narrows the gap between the Voigt and Hashin--Shtrikman models, while mesh refinement preserves the dominant load-transfer patterns and leads only to moderate changes in compliance. The SGP computations further show low empirical sensitivity to the initialization and stable convergence across all examples; the small oscillations of the volume multiplier are attributable to the discrete sampling of the elementwise material subproblems.}

Natural extensions of this work include several directions. \revisionA{A first direction is to construct connectable microstructures through inverse homogenization \citep{Sigmund1994}, using tensors sampled from the energy-bound feasible set as targets. Accounting for finite feature sizes and connections between neighboring cells would allow the resulting structures to be assessed against the homogenized compliance predictions. Successfully reconstructed tensors would also provide inner approximations of the realizability set.} A second direction is to move beyond compliance minimization and consider other objectives and constraints within the same free material framework \citep{Tyburec2025}. A third direction is to further analyze the sequential global programming approach from both computational and theoretical viewpoints, in particular with respect to numerical efficiency, scalability, and convergence behavior. Finally, on the modeling side, it is important to investigate tighter energy bounds, broader material symmetry classes, and fully three-dimensional formulations \citep{Burazin2023}.

\medskip
\noindent\textbf{Acknowledgements.}
The authors acknowledge support from mobility project 8J24DE005, ``Free material optimization for manufacturable modular structures,'' funded jointly by the German Academic Exchange Service (DAAD) and the Ministry of Education, Youth and Sports of the Czech Republic (M\v{S}MT). Marek Tyburec and Shenyuan Ma were additionally co-funded by the European Union through the ROBOPROX project (reg.\ no.\ CZ.02.01.01/00/22\_008/0004590).

\medskip
\noindent\textbf{Generative AI use.}
During the preparation of this work, Marek Tyburec used ChatGPT to improve manuscript clarity. After using this tool, the authors reviewed and edited the content as needed and take full responsibility for the content of the published article.

\medskip
\noindent\textbf{Data availability.}
The code and data used to produce the results in this paper are available on request from the corresponding author.

\medskip
\noindent\textbf{CRediT authorship contribution statement.} 
MT: Conceptualization, Methodology, Software, Formal analysis, Investigation, Writing -- original draft, Visualization. MS: Conceptualization, Methodology, Software, Formal analysis, Investigation, Writing -- original draft, Visualization. SM: Investigation, Writing -- original draft.

\appendix
\renewcommand{\thesection}{\Alph{section}}

\section{Properties of the Hashin--Shtrikman Upper Bound Function}

This section collects regularity properties of the one-dimensional map
$\revisionA{\rho_i}\mapsto f^{\mathrm{HS}}(\strain;\revisionA{\rho_i})$ (for fixed $\strain$) that are used in the main text when eliminating local volume fractions and characterizing activating volumes. We first prove continuity, then monotonicity, and finally derive existence and uniqueness of the activating volume fraction solving
$\langle \Ei\strain,\strain\rangle=f^{\mathrm{HS}}(\strain;\revisionA{\rho_i})$.

\begin{prpstn}[Continuity of the HS upper-bound function in $\revisionA{\rho_i}$]\label{prop:hs_continuous}
    For any fixed strain $\strain \in \mathbb{S}^{2\times2}$ with the Frobenius norm $\sqrt{2}/2$ and $t = \abstrstrain$, the Hashin--Shtrikman upper bound function $f^{\mathrm{HS}} (\revisionA{\rho_i})=\langle [(1-\revisionA{\rho_i})\Emin + \revisionA{\rho_i} \Emax ]\strain,\strain \rangle - q(t;\revisionA{\rho_i})$ with $q(t;\revisionA{\rho_i})$ defined in \eqref{eq:qt} is continuous in the volume fraction $\revisionA{\rho_i} \in [0,1]$.
\end{prpstn}
\begin{proof}
    The Voigt part is affine in $\revisionA{\rho_i}$, hence continuous. Thus, it suffices to show that $q(\strain;\revisionA{\rho_i})$ is continuous.

    For the denominators in \eqref{eq:qt}, set
    \begin{subequations}
    \begin{align}
        d_1 &= \kappa^+ + \mu^+ - (\Delta\kappa + \Delta\mu) \revisionA{\rho_i},\\
        d_2 &= \kappa^+ + \mu^+ - \Delta\kappa \revisionA{\rho_i},\\
        d_3 &= \kappa^+ + \mu^+ - \Delta\mu \revisionA{\rho_i}.
    \end{align}
    \end{subequations}
    Because $\Delta\kappa = \kappa^+-\kappa^- >0$ and $\Delta\mu = \mu^+ - \mu^- >0$ with $\revisionA{\rho_i}\in[0,1]$, we have $d_1, d_2, d_3 > 0$. Hence, the branch formulas are rational with strictly positive denominators, so $q(t;\revisionA{\rho_i})$ is continuous in $\revisionA{\rho_i}$ on each branch interior. It remains to match the boundary values.

    Define the branch-separation sets
    \begin{subequations}
    \begin{align}
    \Gamma_{12}:\; &\revisionA{\rho_i} \Delta\kappa t= d_2\sqrt{1-t^2},\\
    \Gamma_{13}:\; &\revisionA{\rho_i} \Delta\mu \sqrt{1-t^2} = d_3 t,\\
    \Gamma_{23}:\; & \frac{d_3 t}{\revisionA{\rho_i} \Delta \mu} < \sqrt{1-t^2} < \frac{\revisionA{\rho_i}\Delta \kappa t} {d_2}.\label{eq:gamma23}
    \end{align}
    \end{subequations}
    On $\Gamma_{12}$, substituting $\revisionA{\rho_i} \Delta\kappa\, t = d_2 \sqrt{1-t^2}$ into the formulas for $q_1$ and $q_2$ from \eqref{eq:qt} gives
    \begin{equation}
         q_1(t;\revisionA{\rho_i}) = q_2(t;\revisionA{\rho_i}) = \frac{\Delta\kappa^2 \revisionA{\rho_i} (1 - \revisionA{\rho_i}) t^2 d_1}{d_2^2} \quad \text{on}\quad \Gamma_{12}.
    \end{equation}
    On $\Gamma_{13}$, substituting $\revisionA{\rho_i} \Delta\mu\sqrt{1-t^2} = d_3\, t$ into $q_1$ and $q_3$ yields
    \begin{equation}
        q_1(t;\revisionA{\rho_i}) = q_3(t;\revisionA{\rho_i}) = \frac{\Delta\mu^2 \revisionA{\rho_i} (1 - \revisionA{\rho_i}) (1 - t^2) d_1}{d_3^2}
 \quad \text{on}\quad \Gamma_{13}.
    \end{equation}
    Finally, let \eqref{eq:gamma23} hold. Multiplying the strict inequalities gives $t d_2 d_3 < \revisionA{\rho_i}^2 \Delta\mu \Delta\kappa t$, which is equivalent to
    \begin{equation}
        t(d_1 + \revisionA{\rho_i} \Delta\kappa)(d_1 + \revisionA{\rho_i}\Delta\mu) < \revisionA{\rho_i}^2 t \Delta\kappa \Delta\mu.
    \end{equation}
    However, this cannot hold for any $\revisionA{\rho_i} \in (0,1)$ since $d_1 > 0$. Thus, $\Gamma_{23} = \emptyset$.
    
    At $\revisionA{\rho_i}=0$ we have $q(t;0)=0$ (the first branch applies), and at $\revisionA{\rho_i}=1$ each branch carries a factor $(1-\revisionA{\rho_i})$, so $q(t;1)=0$. Moreover, using the explicit boundary expressions above and $\revisionA{\rho_i}(1-\revisionA{\rho_i})\to 0$ as $\revisionA{\rho_i}\to 0^+$ or $1^-$, we see $\lim_{\revisionA{\rho_i}\to 0^+}q(t;\revisionA{\rho_i})=q(t;0)$ and $\lim_{\revisionA{\rho_i}\to 1^-}q(t;\revisionA{\rho_i})=q(t;1)$. Therefore, $q$ is continuous on $[0,1]$, and so is $f^{\mathrm{HS}}$.
\end{proof}

With continuity established, we next prove monotonicity with respect to $\revisionA{\rho_i}$.

\begin{prpstn}[Monotonicity of the HS upper-bound function in $\revisionA{\rho_i}$]\label{prop:hs_monotone}
    For any fixed strain $\strain \in \mathbb{S}^{2\times2}$ with the Frobenius norm $\sqrt{2}/2$, the Hashin--Shtrikman upper bound function $f^{\mathrm{HS}} (\revisionA{\rho_i})=\langle [(1-\revisionA{\rho_i})\Emin + \revisionA{\rho_i} \Emax ]\strain,\strain \rangle - q(t;\revisionA{\rho_i})$ with $q(t;\revisionA{\rho_i})$ defined in \eqref{eq:qt} is nondecreasing in the volume fraction $\revisionA{\rho_i} \in [0,1]$.
\end{prpstn}
\begin{proof}
    By Proposition \ref{prop:hs_continuous}, $f^{\mathrm{HS}} (\revisionA{\rho_i})$ is continuous in $\revisionA{\rho_i}$. It, therefore, suffices to show $\partial f^{\mathrm{HS}} (\revisionA{\rho_i})/\partial \revisionA{\rho_i} \ge 0$ on the interior of each branch and then pass to the boundary by continuity.

    Write the Voigt term as
    \begin{equation}
        \langle [(1-\revisionA{\rho_i})\Emin + \revisionA{\rho_i} \Emax ]\strain,\strain \rangle = 
        (\kappa^- + \revisionA{\rho_i}\Delta\kappa) t^2 + (\mu^- + \revisionA{\rho_i} \Delta\mu) (1-t^2),
    \end{equation}
    whose derivative in $\revisionA{\rho_i}$ is $\Delta\kappa t^2 + \Delta\mu (1-t^2) \ge 0$.

    For the first branch in \eqref{eq:qt}, let $\eta_1 = d_2 \sqrt{1-t^2} - \revisionA{\rho_i} t \Delta\kappa$, $\eta_2 = d_3 t - \revisionA{\rho_i} \sqrt{1-t^2} \Delta\mu$. A direct computation yields
    \begin{equation}
    \begin{aligned}
        \frac{\partial f^{\mathrm{HS}}_1 (\revisionA{\rho_i})}{\partial \revisionA{\rho_i}}&= \frac{\Delta\mu \eta_1^2 + \Delta \kappa \eta_2^2 - \frac{(\Delta\kappa\eta_2 -\Delta\mu\eta_1)^2}{\kappa^+ + \mu^+}}{d_1^2}.
    \end{aligned}
    \end{equation}
    Using the elementary identity
    \begin{equation}
        (\Delta\kappa + \Delta\mu)(\Delta\kappa\eta_2^2 + \Delta\mu\eta_1^2) - (\Delta\kappa\eta_2 - \Delta\mu\eta_1)^2 = \Delta\mu \Delta\kappa (\eta_1 + \eta_2)^2 \ge 0.
    \end{equation}    
    and since $\kappa^+ + \mu^+ > \kappa^- + \mu^- > 0$, we may write
    \begin{equation}
    \begin{aligned}
        \frac{\partial f^{\mathrm{HS}}_1 (\revisionA{\rho_i})}{\partial \revisionA{\rho_i}}&= \frac{(\Delta\kappa\eta_2^2 + \Delta\mu\eta_1^2)(\kappa^- + \mu^-) + \Delta\mu \Delta\kappa(\eta_1 + \eta_2)^2}{d_1^2(\kappa^+ + \mu^+)} \ge 0.
    \end{aligned}
    \end{equation}
    For the second and third branch in \eqref{eq:qt}, differentiation gives
    \begin{subequations}\label{eq:hs_derivative23}
    \begin{align}
        \frac{\partial f^{\mathrm{HS}}_2 (\revisionA{\rho_i})}{\partial \revisionA{\rho_i}} &= \frac{t^2 (\kappa^- + \mu^+) (\kappa^++\mu^+) \Delta\kappa}{d_2^2} > 0,\\
        \frac{\partial f^{\mathrm{HS}}_3 (\revisionA{\rho_i})}{\partial \revisionA{\rho_i}} &= \frac{(1-t^2) (\kappa^+ + \mu^-) (\kappa^++\mu^+) \Delta\mu}{d_3^2} > 0,
    \end{align}
    \end{subequations}
    with the strict inequality due to the branch conditions. Since the three expressions are nonnegative on the interiors and the branch values of $q$ agree on $\Gamma_{12} \cup \Gamma_{13}$ (Proposition \ref{prop:hs_continuous}), continuity implies that $f^{\mathrm{HS}}(\revisionA{\rho_i})$ is globally nondecreasing in $\revisionA{\rho_i} \in [0,1]$.
\end{proof}

Combining endpoint values with continuity and monotonicity gives the following existence and uniqueness consequence.

\begin{crllr}[Existence and uniqueness of the activating HS volume fraction]\label{cor:uniqueness}
    For any fixed strain $\strain \in \mathbb{S}^{2\times2}$ with the Frobenius norm $\sqrt{2}/2$ and elasticity tensor $\Ei \in \mathbb{S}^{2\times2\times2\times2}$ with $\Emin \preceq \Ei \preceq \Emax$, there exists a unique volume fraction $\revisionA{\rho_i} \in [0,1]$ such that $\langle \Ei \strain, \strain\rangle = f^{\mathrm{HS}}(\strain;\revisionA{\rho_i})$.
\end{crllr}
\begin{proof}
    Since $\Emin \preceq \Ei \preceq \Emax$, we have $\langle\Emin\strain,\strain\rangle \le \langle\Ei\strain,\strain\rangle \le \langle\Emax\strain,\strain\rangle$. Moreover, by the endpoint evaluation in Proposition~\ref{prop:q_ge_0}, $q(\strain;0) = q(\strain;1) = 0$. Hence, $f^\mathrm{HS}(\strain;0) = \langle\Emin\strain,\strain\rangle$, $f^\mathrm{HS}(\strain;1) = \langle\Emax\strain,\strain\rangle$, and $f^\mathrm{HS}(\strain;\cdot)$ is continuous on $[0,1]$ by Proposition \ref{prop:hs_continuous}. Therefore, the intermediate value theorem gives a (possibly nonunique) $\revisionA{\rho_i} \in [0,1]$ with $\langle\Ei \strain,\strain \rangle = f^\mathrm{HS}(\strain;\revisionA{\rho_i})$.

    By Proposition~\ref{prop:hs_monotone}, $f^\mathrm{HS}(\strain;\cdot)$ is nondecreasing on $[0,1]$. We now show it is in fact strictly increasing, which implies uniqueness. On each branch interior, the derivative $\partial f^{\mathrm{HS}}(\strain;\revisionA{\rho_i})/\partial \revisionA{\rho_i}$ is nonnegative, and in the second and third branches the derivatives are strictly positive (cf. Eq. \eqref{eq:hs_derivative23}). In the first branch, the derivative is nonnegative and vanishes only at points where $\eta_1 = \eta_2 = 0$, \revisionA{which is impossible: since $d_2d_3-\rho_i^2\Delta\kappa\Delta\mu =(\kappa^++\mu^+)d_1>0$, these equalities would imply $t=\sqrt{1-t^2}=0$}. Thus, $f^{\mathrm{HS}}(\strain;\cdot)$ cannot be flat on any nontrivial interval contained in a branch interior. At the boundaries, adjacent-branch formulas coincide, and $\Gamma_{23} = \emptyset$. Hence, any interval $[\revisionA{\rho_1},\revisionA{\rho_2}] \subset [0,1]$ decomposes into finitely many subintervals lying in branch interiors, separated by isolated boundary points. Integrating the (strictly) positive derivatives over each interior subinterval shows $\revisionA{\rho_2}>\revisionA{\rho_1} \Rightarrow f^\mathrm{HS}(\strain;\revisionA{\rho_2})>f^\mathrm{HS}(\strain;\revisionA{\rho_1})$. Therefore, $f^{\mathrm{HS}}(\strain;\cdot)$ is strictly increasing on $[0,1]$, and the solution to $\langle\Ei \strain,\strain \rangle = f^{\mathrm{HS}}(\strain;\revisionA{\rho_i})$ is unique.
\end{proof}

\section{Regularity of the Hashin--Shtrikman Lower Complementary-Energy Bound}\label{appendix:rederivation_burzain}

Following \citet{Burazin2021}, the cell-wise material update contains the one-dimensional volume-fraction subproblem
\begin{equation}
    \text{find } \revisionA{\rho}\in\argmin_{\revisionA{r}\in[0,1]} f_{\mathrm{c}}^{\mathrm{HS}}(\stress;\revisionA{r})+\lambda \revisionA{r}.
\end{equation}
In this section, for fixed $\stress$, we establish regularity properties of the map $\revisionA{\rho}\mapsto f_{\mathrm{c}}^{\mathrm{HS}}(\stress;\revisionA{\rho})$. These properties justify an efficient and globally optimal solution strategy for this one-dimensional minimization problem.

\begin{dfntn}[Optimal complementary energy lower bound {\citep[Def.~2.3.15]{Allaire2002}}]\label{def:lower_bound}
    Let $\Emin \preceq \Emax$ be two well-ordered isotropic elasticity tensors and let $\revisionA{\rho}\in[0,1]$ be the volume fraction of $\Emax$. Denote by $G_{\revisionA{\rho}}$ the set of all effective elasticity tensors $\Estar$ obtainable by homogenization of the two phases $(\Emin,\Emax)$ with volume fraction $\revisionA{\rho}$.

    A function $f(\stress;\revisionA{\rho})$ is called a (complementary-energy) lower bound if, for every $\stress\in\mathbb{S}^{d\times d}$ and every $\Estar\in G_{\revisionA{\rho}}$,
    \[
    \langle (\Estar)^{-1}\stress,\stress\rangle \;\ge\; f(\stress;\revisionA{\rho}).
    \]
    The lower bound is optimal if, for every $\stress\in\mathbb{S}^{d\times d}$, there exists $\Estar(\stress)\in G_{\revisionA{\rho}}$ such that equality holds, i.e.,
    \[
    \langle (\Estar(\stress))^{-1}\stress,\stress\rangle \;=\; f(\stress;\revisionA{\rho}).
    \]
\end{dfntn}

The Hashin--Shtrikman lower complementary-energy bound is optimal. Let $\kappa^+$ and $\mu^+$ denote the bulk and shear moduli of $\Emax$, and let $\tens{e}\in\mathbb{R}^{d}$ satisfy $\lVert\tens{e}\rVert=1$. Define $\tensf{F}_{\Emax}^\mathrm{c}(\tens{e})\in\mathbb{S}^{d\times d\times d\times d}$ by, for all $\strain\in\mathbb{S}^{d\times d}$,
\begin{equation}
    \langle \tensf{F}_{\Emax}^\mathrm{c}(\tens{e})\strain,\strain\rangle = \langle \Emax\strain,\strain\rangle-\frac{1}{\mu^+}\lVert(\Emax\strain)\tens{e}\rVert^2+
    \left(\frac{1}{\mu^+} - \frac{1}{\kappa^+ + 2 \mu^+ (d - 1)/d}\right)
    \left\langle (\Emax\strain)\tens{e},\, \tens{e}\right\rangle^2. 
\end{equation}
The tensor $\tensf{F}_{\Emax}^\mathrm{c}(\tens{e})$ admits a geometric interpretation. For a fixed unit direction $\tens{e}$, define the subspaces of $\mathbb{S}^{d\times d}$:
\begin{align}
    V(\tens{e})&:=\mathrm{span}\{\tens{e}\hat{\tens{e}}^\trn+\hat{\tens{e}}{\tens{e}}^\trn;\hat{\tens{e}}\in\mathbb{R}^d\},\\
    W(\tens{e})&:=\{\strain\in\mathbb{S}^{d\times d}:\strain\tens{e}=\tens{0}\}.
\end{align}
Since $\Emax$ is positive definite, it admits a unique square root $\left({\Emax}^{\frac{1}{2}}\right)^2=\Emax$, and $\mathbb{S}^{d\times d}$ has the orthogonal decomposition
$$
    \mathbb{S}^{d\times d}=(\Emax)^{\frac{1}{2}}V(\tens{e})\oplus(\Emax)^{-\frac{1}{2}}W(\tens{e}).
$$
Let $\pi_{\mathcal{W}}$ be the projection onto a subspace $\mathcal{W}\subset\mathbb{S}^{d\times d}$. Then,
$$
    \langle \tensf{F}_{\Emax}^\mathrm{c}(\tens{e})\strain,\strain\rangle = \left\lVert\pi_{(\Emax)^{-\frac{1}{2}}W(\tens{e})}\left((\Emax)^{\frac{1}{2}}\strain\right)\right\rVert_F^2=\left\lVert(\Emax)^{\frac{1}{2}}\strain\right\rVert_F^2-\left\lVert\pi_{(\Emax)^{\frac{1}{2}}V(\tens{e})}\left((\Emax)^{\frac{1}{2}}\strain\right)\right\rVert_F^2.
$$
\begin{thrm}[HS lower bound on complementary energy; optimality
{\citep[Proposition~2.3.25]{Allaire2002}}]
Assume two well-ordered isotropic phases $\Emin\preceq\Emax$ and let $\revisionA{\rho}\in[0,1]$. Let $G_{\revisionA{\rho}}$ be the corresponding $G$-closure. Define, for $\stress\in\mathbb{S}^{d\times d}$,
\begin{equation}\label{eq:hs_lower_bound}
f_{\mathrm{c}}^{\mathrm{HS}}(\stress;\revisionA{\rho})
:= \langle (\Emax)^{-1}\stress,\stress\rangle
 + (1-\revisionA{\rho})\max_{\strain\in\mathbb{S}^{d\times d}}
\Big\{ 2\langle\stress,\strain\rangle
 - \big\langle\big((\Emin)^{-1}-(\Emax)^{-1}\big)^{-1}\strain,\strain\big\rangle
 - \revisionA{\rho}\, g^{\mathrm{c}}(\strain)\Big\},
\end{equation}
where $g^{\mathrm{c}}(\strain):=\max_{\lVert\tens{e}\rVert=1}\langle \tensf{F}^{\mathrm{c}}_{\Emax}(\tens{e})\strain,\strain\rangle$. Then, $f_{\mathrm{c}}^{\mathrm{HS}}$ is a (complementary-energy) lower bound in the sense of Definition~\ref{def:lower_bound}: for all $\stress\in\mathbb{S}^{d\times d}$ and all $\Estar\in G_{\revisionA{\rho}}$, $\langle (\Estar)^{-1}\stress,\stress\rangle \ge f_{\mathrm{c}}^{\mathrm{HS}}(\stress;\revisionA{\rho})$. Moreover, the bound is optimal: for every $\stress$ there exists $\Estar(\stress)\in G_{\revisionA{\rho}}$ generated by a finite-rank sequential laminate such that equality holds. In $d=2$, orthotropic rank-two laminates aligned with the principal stress directions are sufficient.
\end{thrm}

The nonlocal term $g^\mathrm{c}$ admits a closed-form expression.
\begin{lmm}[{\citep[Eq. (4.15)]{Allaire1993od}}]
    Let $\varepsilon_1,\varepsilon_2$ be the two eigenvalues of $\strain\in\mathbb{S}^{2\times 2}$, and let $\mu^+$ and $\kappa^+$ be the shear and bulk moduli of $\tensf{E}^+$. Then,
	    \begin{equation}
	        g^\mathrm{c}(\strain)=\frac{4\kappa^+\mu^+}{\kappa^++\mu^+}\max\{\varepsilon_1^2,\varepsilon_2^2\}.
	        \label{eq:gc_2d}
	    \end{equation}
    \revisionA{The prefactor $4\kappa^+\mu^+/(\kappa^++\mu^+)$ is the Young modulus $E^+$ of the stiff phase in the present two-dimensional plane-stress setting. Its reciprocal appears in the conjugate expressions below.} Moreover, if $\tens{e}_1,\tens{e}_2$ are the eigenvectors of $\strain$ associated with the eigenvalues $\varepsilon_1,\varepsilon_2$, then $\max_{\lVert\tens{e}\rVert = 1}\langle \tensf{F}_{\Emax}^\mathrm{c}(\tens{e})\strain,\strain\rangle$ is attained by choosing $\tens{e} = \tens{e}_i$, where $i \in \{1,2\}$ satisfies $\varepsilon_i^{\revisionA{2}}=\min\{\varepsilon_1^2,\varepsilon_2^2\}$.
\end{lmm}

The Fenchel conjugate of $g^\mathrm{c}$ is explicitly known.
\begin{lmm}[{\citep{Allaire1993od}}]
    Let $\sigma_1,\sigma_2$ be the two eigenvalues of $\stress$, let $\revisionA{\rho}>0$, and let $\mu^+$ and $\kappa^+$ be the shear and bulk moduli of $\tensf{E}^+$. Then,
    \begin{equation}\label{eq:allaire_kohn_gc}
        \max_{\strain\in\mathbb{S}^{2\times 2}}\left\{2\langle\stress,\strain\rangle-\revisionA{\rho} g^\mathrm{c}(\strain)\right\}=\frac{\kappa^++\mu^+}{4\kappa^+\mu^+\revisionA{\rho}}(\lvert\sigma_1\rvert+\lvert\sigma_2\rvert)^2.
    \end{equation}
\end{lmm}

Identity \eqref{eq:allaire_kohn_gc} can be interpreted as the convex (Fenchel) conjugate of the function $\strain\mapsto \revisionA{\rho} g^\mathrm{c}(\strain)$, evaluated at $2\stress$. More precisely,
\begin{equation}
    [\revisionA{\rho} g^\mathrm{c}]^*(2\stress)=\frac{\kappa^++\mu^+}{4\kappa^+\mu^+\revisionA{\rho}}(\lvert\sigma_1\rvert+\lvert\sigma_2\rvert)^2.
\end{equation}
Moreover, for any convex function $\phi$ and $\revisionA{\rho}>0$, one has the scaling identity $(\revisionA{\rho}\phi)^*(\tenss{\tau})=\revisionA{\rho}\phi^*(\tenss{\tau}/\revisionA{\rho})$. Applying this with $\phi=g^\mathrm{c}$ gives
\begin{equation}
[\revisionA{\rho} g^\mathrm{c}]^*(2\tenss{\tau})=\revisionA{\rho}[g^\mathrm{c}]^*\!\left(\frac{2\tenss{\tau}}{\revisionA{\rho}}\right).
\end{equation}
Hence, $(\tenss{\tau},\revisionA{\rho})\mapsto [\revisionA{\rho} g^\mathrm{c}]^*(2\tenss{\tau})$ is the perspective of $[g^\mathrm{c}]^*$ and is jointly convex in $(\tenss{\tau},\revisionA{\rho})$.

Define $\tensf{R}=\left((\Emin)^{-1}-(\Emax)^{-1}\right)^{-1}$. By standard convex duality, we obtain:
\begin{equation}
    \max_{\strain\in\mathbb{S}^{2\times 2}}
\left\{ 2\langle\stress,\strain\rangle
 - \left\langle\tensf{R}\strain,\strain\right\rangle
 - \revisionA{\rho}\, g^{\mathrm{c}}(\strain)\right\}
    = \inf_{\tenss{\tau}\in\mathbb{S}^{2\times 2}}\langle\tensf{R}^{-1}(\stress-\tenss{\tau}),(\stress-\tenss{\tau})\rangle+[\revisionA{\rho} g^\mathrm{c}]^*(2\tenss{\tau})
\end{equation}

We now analyze the regularity of $\revisionA{\rho}\mapsto f_{\mathrm{c}}^{\mathrm{HS}}(\stress;\revisionA{\rho})$ for fixed $\stress\neq \tenss{0}$. Define
\begin{align}
    \psi(\tenss{\tau};\revisionA{\rho})&:=\langle\tensf{R}^{-1}(\stress-\tenss{\tau}),(\stress-\tenss{\tau})\rangle+[\revisionA{\rho} g^\mathrm{c}]^*(2\tenss{\tau})\\
    \Psi(\revisionA{\rho})&:=\inf_{\tenss{\tau}\in\mathbb{S}^{2\times 2}}\psi(\tenss{\tau};\revisionA{\rho})\\
    \mathcal{S}(\revisionA{\rho})&:=\argmin_{\tenss{\tau}\in\mathbb{S}^{2\times 2}}\psi(\tenss{\tau};\revisionA{\rho})
\end{align}
so that
\begin{equation}
    f_{\mathrm{c}}^{\mathrm{HS}}(\stress;\revisionA{\rho})=\langle(\Emax)^{-1}\stress,\stress\rangle+(1-\revisionA{\rho})\Psi(\revisionA{\rho}).
\label{eq:burzain_valuefunction}
\end{equation}
By the explicit expression of $[\revisionA{\rho} g^\mathrm{c}]^*(2\tenss{\tau})$ (cf.\ \eqref{eq:allaire_kohn_gc}), for each fixed $\tenss{\tau}$ the map $\revisionA{\rho}\mapsto \psi(\tenss{\tau};\revisionA{\rho})$ is differentiable on $(0,1]$, with
\begin{equation}
    D_{\revisionA{\rho}} \psi(\tenss{\tau};\revisionA{\rho})=-\frac{\kappa^++\mu^+}{4\kappa^+\mu^+(\revisionA{\rho})^2}(\lvert\tau_1\rvert+\lvert\tau_2\rvert)^2.
\end{equation}
In particular, $(\tenss{\tau},\revisionA{\rho})\mapsto D_{\revisionA{\rho}}\psi(\tenss{\tau};\revisionA{\rho})$ is continuous on $\mathbb{S}^{2 \times 2}\times (0,1]$. Hence, regularity of $f_{\mathrm{c}}^{\mathrm{HS}}(\stress;\revisionA{\rho})$ reduces to regularity of the value function $\Psi$.

To apply the Danskin-type results of \citet{bonnans2013perturbation} to
$
\Psi(\revisionA{\rho})=\inf_{\tenss{\tau}}\psi(\tenss{\tau};\revisionA{\rho}),
$
we need two ingredients: (i) inf-compactness, ensuring bounded minimizing sequences, and (ii) uniqueness of the minimizer for each fixed $\revisionA{\rho}>0$. We prove these next.
 \begin{prpstn}
     For all $0<\revisionA{\rho_0}\leq 1$, $\psi$ has the inf-compactness property at $\revisionA{\rho_0}$.
 \end{prpstn}
 \begin{proof}
     It is enough to construct a compact set $C\subset \mathbb{S}^{2\times 2}$, a constant $\alpha\in \mathbb{R}$, and a neighborhood $V$ of $\revisionA{\rho_0}$ such that, for every $\revisionA{\rho}\in V$, the sublevel set $\{\tenss{\tau}\in\mathbb{S}^{2\times 2}:\psi(\tenss{\tau};\revisionA{\rho})\leq \alpha\}$ is contained in $C$.

	     Set $V=(\frac{\revisionA{\rho_0}}{2},\revisionA{\rho_0}+1)$. By the explicit expression of $[\revisionA{\rho} g^\mathrm{c}]^*(2\tenss{\tau})$ (cf.\ \eqref{eq:allaire_kohn_gc}), the map $\revisionA{\rho}\mapsto [\revisionA{\rho} g^\mathrm{c}]^*(2\tenss{\tau})$ is decreasing on $(0,1]$ for each fixed $\tenss{\tau}$, hence
	     $
	         \psi(\tenss{\tau};\revisionA{\rho})\leq \psi\left(\tenss{\tau};\tfrac{\revisionA{\rho_0}}{2}\right)$
          for all $\revisionA{\rho}\in V \text{ and } \tenss{\tau}\in\mathbb{S}^{2\times 2}.
     $
     Now fix $\tenss{\tau}_0\in\mathbb{S}^{2\times 2}$ and set $\alpha=\psi\left(\tenss{\tau}_0;\frac{\revisionA{\rho_0}}{2}\right)$, so that the above sublevel set is nonempty. Let $r>0$ be the smallest eigenvalue of $\tensf{R}^{-1}$ and define
     $
         C:=\left\{\tenss{\tau}\in\mathbb{S}^{2\times 2}:\ r\lVert\stress-\tenss{\tau}\rVert_F^2+\frac{\kappa^++\mu^+}{4\kappa^+\mu^+(\revisionA{\rho_0}+1)}\left(\lvert\tau_1\rvert+\lvert\tau_2\rvert\right)^2\leq \alpha\right\}.
     $
     The set $C$ is closed and bounded in finite dimension, hence compact. Finally, if $\revisionA{\rho}\in V$ and $\psi(\tenss{\tau};\revisionA{\rho})\leq \alpha$, then using $r\lVert\stress-\tenss{\tau}\rVert_F^2\leq \langle \tensf{R}^{-1}(\stress-\tenss{\tau}),(\stress-\tenss{\tau})\rangle$ and $\revisionA{\rho}\leq \revisionA{\rho_0}+1$ we obtain
     \begin{multline}
         r\lVert\stress-\tenss{\tau}\rVert_F^2+\frac{\kappa^++\mu^+}{4\kappa^+\mu^+(\revisionA{\rho_0}+1)}\left(\lvert\tau_1\rvert+\lvert\tau_2\rvert\right)^2
         \\ \leq \langle \tensf{R}^{-1}(\stress-\tenss{\tau}),(\stress-\tenss{\tau})\rangle+\frac{\kappa^++\mu^+}{4\kappa^+\mu^+\revisionA{\rho}}\left(\lvert\tau_1\rvert+\lvert\tau_2\rvert\right)^2
         =\ \psi(\tenss{\tau};\revisionA{\rho})\leq \alpha,
     \end{multline}
     which implies $\tenss{\tau}\in C$.
 \end{proof}
The strong convexity of the quadratic term then gives uniqueness of the minimizer.
 \begin{prpstn}
     For all $0<\revisionA{\rho_0}\leq 1$, $\mathcal{S}(\revisionA{\rho_0})$ is single valued.
 \end{prpstn}
 \begin{proof}
     This follows from the strong convexity of $\tenss{\tau}\mapsto\psi(\tenss{\tau};\revisionA{\rho_0})$.
 \end{proof}
With inf-compactness and uniqueness established, we can invoke Danskin's theorem to differentiate $\Psi$ with respect to $\revisionA{\rho}$.
	 \begin{thrm}
	     $\Psi$ is Fréchet differentiable in $(0,1]$.
	 \end{thrm}
	 \begin{proof}
     Fix $\revisionA{\rho_0}>0$. We verify the assumptions of \citet[Theorem~4.13]{bonnans2013perturbation}:
     \begin{enumerate}
         \item $\forall\tenss{\tau}\in\mathbb{S}^{2\times 2}$, $\psi(\tenss{\tau};\cdot)$ is Gâteaux differentiable near $\revisionA{\rho_0}$,
         \item $(\tenss{\tau},\revisionA{\rho})\mapsto\psi(\tenss{\tau};\revisionA{\rho})$ and $(\tenss{\tau},\revisionA{\rho})\mapsto D_{\revisionA{\rho}}\psi(\tenss{\tau};\revisionA{\rho})$ are continuous near $\revisionA{\rho_0}$,
         \item $\psi$ has the inf-compactness property at $\revisionA{\rho_0}$
     \end{enumerate}
     Since $\mathcal{S}(\revisionA{\rho_0})$ is single valued, Fréchet differentiability and the gradient formula follow (cf.\ \citep[Remark~4.14]{bonnans2013perturbation}):
     \begin{equation}
         D_{\revisionA{\rho}}\Psi(\revisionA{\rho_0})=-\frac{\kappa^++\mu^+}{4\kappa^+\mu^+(\revisionA{\rho_0})^2}(\lvert\tau_1(\revisionA{\rho_0})\rvert+\lvert\tau_2(\revisionA{\rho_0})\rvert)^2,
     \end{equation}
     concluding the proof.
 \end{proof}
 
Consequently, for fixed $\stress$, the map $\revisionA{\rho}\mapsto f_{\mathrm{c}}^{\mathrm{HS}}(\stress;\revisionA{\rho})$ is differentiable on $(0,1]$. Using \eqref{eq:burzain_valuefunction},
\begin{equation}
    D_{\revisionA{\rho}}f_{\mathrm{c}}^{\mathrm{HS}}(\stress;\revisionA{\rho})=-\Psi(\revisionA{\rho})+(1-\revisionA{\rho})D_{\revisionA{\rho}}\Psi(\revisionA{\rho}).
\end{equation}
We now prove convexity of $f_{\mathrm{c}}^{\mathrm{HS}}$ with respect to $\revisionA{\rho}$.
\begin{thrm}
    The lower bound $f_{\mathrm{c}}^{\mathrm{HS}}$ is convex in $\revisionA{\rho}$ on $(0,1]$ for fixed $\stress$.
\end{thrm}
\begin{proof}
    By \eqref{eq:burzain_valuefunction}, it suffices to prove that $\revisionA{\rho}\mapsto (1-\revisionA{\rho})\Psi(\revisionA{\rho})$ is convex on $(0,1]$. We first note that $\Psi$ is nonincreasing on $(0,1]$: for each fixed $\tenss{\tau}$, the map $\revisionA{\rho}\mapsto \psi(\tenss{\tau};\revisionA{\rho})$ is nonincreasing (cf.\ the explicit expression \eqref{eq:allaire_kohn_gc}), and taking the infimum over $\tenss{\tau}$ preserves monotonicity. Next, $\Psi$ is convex as a partial minimization of a jointly convex function: indeed,
    $[\revisionA{\rho} g^\mathrm{c}]^*(2\tenss{\tau})=\revisionA{\rho}(g^\mathrm{c})^*\!\left(\frac{2\tenss{\tau}}{\revisionA{\rho}}\right)$ for $\revisionA{\rho}>0$, which is the perspective of $(g^\mathrm{c})^*$ and therefore jointly convex in $(\tenss{\tau},\revisionA{\rho})$; hence $\Psi(\revisionA{\rho})=\inf_{\tenss{\tau}}\psi(\tenss{\tau};\revisionA{\rho})$ is convex.

    Finally, let $\revisionA{\rho_\theta}=\theta \revisionA{\rho_1}+(1-\theta)\revisionA{\rho_2}$ with $\revisionA{\rho_1},\revisionA{\rho_2}\in(0,1]$ and $\theta\in[0,1]$. By convexity of $\Psi$, $(1-\revisionA{\rho_\theta})\Psi(\revisionA{\rho_\theta})\leq (1-\revisionA{\rho_\theta})\bigl(\theta\Psi(\revisionA{\rho_1})+(1-\theta)\Psi(\revisionA{\rho_2})\bigr)$. Since $\Psi$ is nonincreasing, the latter is bounded above by $\theta(1-\revisionA{\rho_1})\Psi(\revisionA{\rho_1})+(1-\theta)(1-\revisionA{\rho_2})\Psi(\revisionA{\rho_2})$, because the difference equals $\theta(1-\theta)(\revisionA{\rho_2}-\revisionA{\rho_1})\bigl(\Psi(\revisionA{\rho_1})-\Psi(\revisionA{\rho_2})\bigr)\geq 0$.
\end{proof}
Since $\revisionA{\rho}\mapsto f_{\mathrm{c}}^{\mathrm{HS}}(\stress;\revisionA{\rho})$ is convex and differentiable on $(0,1]$, the elementwise update
\begin{equation}
\revisionA{\rho}\in\argmin_{\revisionA{r}\in[0,1]}\bigl(f_{\mathrm{c}}^{\mathrm{HS}}(\stress;\revisionA{r})+\lambda \revisionA{r}\bigr)
\end{equation}
can be performed efficiently. Since $\revisionA{\rho}\mapsto D_{\revisionA{\rho}} f_{\mathrm{c}}^{\mathrm{HS}}(\stress;\revisionA{\rho})+\lambda$ is nondecreasing, we use bisection: if $D_{\revisionA{\rho}}f_{\mathrm{c}}^{\mathrm{HS}}(\stress;1)<-\lambda$, set $\revisionA{\rho}=1$; otherwise, bracket a root of $D_{\revisionA{\rho}} f_{\mathrm{c}}^{\mathrm{HS}}(\stress;\revisionA{\rho})=-\lambda$ by dyadic steps (or set $\revisionA{\rho}=0$ at tolerance) and bisect on $[\revisionA{\rho^{1}},\revisionA{\rho^{0}}]$.

\bibliographystyle{abbrvnat}
\bibliography{liter}

\end{document}